\documentclass[letterpaper,10pt]{amsart}
\usepackage{latexsym,array,delarray,amsthm,amssymb,epsfig,hyperref}
\usepackage{color}
\usepackage{tikz}
\usepackage{cleveref}

\theoremstyle{plain}
\newtheorem{thm}{Theorem}[section]
\newtheorem{lemma}[thm]{Lemma}
\newtheorem{prop}[thm]{Proposition}
\newtheorem{cor}[thm]{Corollary}
\newtheorem{conj}[thm]{Conjecture}
\newtheorem*{thm*}{Theorem}
\newtheorem*{lemma*}{Lemma}
\newtheorem*{prop*}{Proposition}
\newtheorem*{cor*}{Corollary}
\newtheorem*{conj*}{Conjecture}

\theoremstyle{definition}
\newtheorem{defn}[thm]{Definition}
\newtheorem{ex}[thm]{Example}

\theoremstyle{remark}
\newtheorem*{rmk}{Remark}

\newcommand{\kk}{\mathbb{K}}

\newcommand{\mhat}{\widehat{m}}

\newcommand{\ind}{\mbox{$\perp \kern-5.5pt \perp$}}
\newcommand{\cone}{\operatorname{cone}}
\newcommand{\init}{\operatorname{in}}

\newcommand{\Tor}{\text{Tor}}

\newcommand{\Bruhatl}{\mathrel{<_{\mathrm{Br}}}} 
\newcommand{\Bruhatg}{\mathrel{>_{\mathrm{Br}}}}
\newcommand{\Bruhatle}{\mathrel{\le_{\mathrm{Br}}}}
\newcommand{\Bruhatge}{\mathrel{\ge_{\mathrm{Br}}}}

\makeatletter
\renewcommand*\env@matrix[1][*\c@MaxMatrixCols c]{%
  \hskip -\arraycolsep
  \let\@ifnextchar\new@ifnextchar
  \array{#1}}
\makeatother

\let\oldmarginpar\marginpar
\renewcommand\marginpar[1]{\-\oldmarginpar[\footnotesize #1]{\raggedright\footnotesize #1}}

\def\textcross{\begin{tikzpicture}	
	\draw[thick] (0,.2)  to (.4,.2)
	 [thick](.2,.4)  to (.2,0);
	\end{tikzpicture}}
\def\textelbow{\begin{tikzpicture}	
	\draw[thick] (0,.2)  to [bend left=50] (.2,0)
	[thick] (.2,.4)  to [bend right=50] (.4,.2);
	\end{tikzpicture}
	}
\def\textcelbow{\begin{tikzpicture}	
	\draw[thick] (0,.2)  to [bend left=50] (.2,0);
	\end{tikzpicture}}
\newcommand{\+}{
	\draw[thick] (0,.25)  to (.5,.25)
	[thick] (.25,.5)  to (.25,0);
	 }
\def\elbow{
	\draw[thick] (0,.25)  to [bend left=50] (.25,0)
	[thick] (.25,.5)  to [bend right=50] (.5,.25);
	}
\def\celbow{
	\draw[thick] (0,.25)  to [bend left=50] (.25,0);
	}

\DeclareMathOperator{\Spec}{Spec}

\begin{document}

\title[Gr\"obner bases, symmetric matrices, and type C Kazhdan-Lusztig varieties]{Gr\"obner bases, symmetric matrices,\\ and type C Kazhdan-Lusztig varieties}

\author{Laura Escobar}
\address{Laura Escobar \\ Department of Mathematics and Statistics\\ Washington University in St.~Louis \\ St.~Louis, MO, USA }
\email{laurae@wustl.edu}

\author{Alex Fink}
\address{Alex Fink\\School of Mathematical Sciences\\ Queen Mary University of London\\ London, UK}
\email{a.fink@qmul.ac.uk}

\author{Jenna Rajchgot}
\address{Jenna Rajchgot\\ Department of Mathematics and Statistics\\ McMaster University\\Hamilton, ON, Canada}
\email{rajchgot@math.mcmaster.ca}

\author{Alexander Woo}
\address{Alexander Woo\\Department of Mathematics\\University of Idaho\\Moscow, ID, USA}
\email{awoo@uidaho.edu}

\maketitle

\begin{abstract}
We study a class of combinatorially-defined polynomial ideals which are generated by  minors of a generic symmetric matrix. Included within this class are the symmetric determinantal ideals, the symmetric ladder determinantal ideals, and the symmetric  Schubert determinantal ideals of A. Fink, J. Rajchgot, and S. Sullivant. Each ideal in our class is a type~C analog of a Kazhdan-Lusztig ideal of A. Woo and A. Yong; that is, it is the scheme-theoretic defining ideal of the intersection of a type~C Schubert variety with a type~C opposite Schubert cell, appropriately coordinatized.
The Kazhdan-Lusztig ideals that arise are exactly those where the opposite cell is $123$-avoiding.
Our main results include Gr\"obner bases for these ideals, prime decompositions of their initial ideals (which are Stanley-Reisner ideals of subword complexes) and combinatorial formulas for their multigraded Hilbert series in terms of pipe dreams.

\end{abstract}

\setcounter{tocdepth}{1}
\tableofcontents

\section{Introduction}

Let $\mathbb{K}$ be a field of characteristic zero.
In this paper, we study a class of generalized symmetric determinantal ideals. 
Each ideal in our class is defined by imposing certain \emph{combinatorial} southwest rank conditions on an $n\times n$ symmetric matrix $M$ whose $i,j$ entry is either zero or an indeterminate $z_{ij}=z_{ji}$ and whose nonzero entries lie in a skew partition, in English conventions. Among the ideals in our class are the symmetric determinantal ideals, the symmetric ladder determinantal ideals  \cite{Gorla,Gorla-Migliore-Nagel}, and the symmetric Schubert determinantal ideals of \cite{Fink-Rajchgot-Sullivant}. We plan to describe in detail the connection with symmetric ladder determinantal ideals in a separate paper.

Let $R= \mathbb{K}[z_{ij}]$ be the polynomial ring in the variables appearing in a matrix $M$ as above.
We interpret this ring in terms of a type~C opposite Schubert cell.
Let $G$ be the symplectic group $Sp_{2n}(\mathbb{K})$,
represented as the group of $2n\times 2n$ matrices preserving the form $e_1\wedge e_{2n}+\cdots+e_n\wedge e_{n+1}$.
We will work with the following Borel subgroups of $G$\/:
	$$
	B^+_G:=\{\text{upper triangular matrices in }G\}
	\quad
	\text{and}
	\quad
	B^-_G:=\{\text{lower triangular matrices in }G\}.
	$$
The type~C flag variety is $G/B^+_G$, and an opposite Schubert cell is a $B^-_G$-orbit in $G/B^+_G$.
These cells are indexed by elements of the Weyl group $C_n$ of $G$, which can be identified with the  set of permutations
	\begin{equation*}
	C_n=\{v_1\ldots v_{2n}\in S_{2n} : v_{2n+1-i}=2n+1-v_i \text{ for } i=1,\ldots,n\}.
	\end{equation*}
Then the ring $R$ is the coordinate ring of a type~C opposite Schubert cell associated to some $123$-avoiding permutation $v\in C_n$, with an appropriate choice of coordinates (see \Cref{prop: partial symm matrices}). 
From the point of view of Schubert cells in $G/B^+_G$, our choices of symplectic form and coordinates have
a long history.  These coordinates were used extensively to study Schubert cells by W.~Fulton and P.~Pragacz~\cite{FultonPragacz}, who were likely aware of the connection to symmetric matrices at least in some special cases.

This choice of coordinates allows us to study a large class of generalized symmetric determinantal ideals from the point of view of Kazhdan-Lusztig varieties in $G/B^+_G$. 
Each ideal we encounter is obtained by imposing southwest rank conditions on $M$, using combinatorial rules encoded by some $w\in C_n$.
Given $v,w\in C_n$, we denote by $ \mathcal N_{v,w}$ the affine variety associated to one of our ideals; these varieties form a subclass of type~C Kazhdan-Lusztig varieties.
A Schubert variety is a $B^+_G$-orbit closure in $G/B^+_G$, and a Kazhdan-Lusztig variety is the intersection of a Schubert variety with an opposite Schubert cell.

In general, Kazhdan-Lusztig varieties provide affine neighborhoods of the $T$-fixed points of Schubert varieties, and they have been used to study singularities of Schubert varieties.
One such instance is \cite{WooYongSings}, in which A.~Woo and A.~Yong introduced Kazhdan-Lusztig ideals of type~A for this purpose. 
Each Kazhdan-Lusztig ideal is the prime defining ideal of a type~A Kazhdan-Lusztig variety, appropriately coordinatized.
In \cite{WooYongGrobner}, Woo and Yong showed that Kazhdan-Lusztig ideals of type~A possess nice Gr\"obner bases for which the corresponding initial ideals are Stanley-Reisner ideals of vertex decomposable balls or spheres. They furthermore proved multiple combinatorial formulas for their multigraded Hilbert series.

Similarly, a type~C Kazhdan-Lusztig ideal is the defining ideal of a type~C Kazhdan-Lusztig variety.
In our main theorem, we use the interpretation of $ \mathcal N_{v,w}$ as a type~C Kazhdan-Lusztig variety to give a Gr\"obner basis with squarefree initial terms for the ideals we encounter.
We give $R$ a term order which is \emph{diagonal}. Roughly this means that the leading term of any minor is the product of the diagonal entries of the submatrix it arises from. 
One example of a diagonal term order is the lexicographic term order where
$z_{ij}>z_{i'j'}$ if and only if either $i>i'$, or $i=i'$ and $j>j'$. 
Our main result is as follows. It is stated precisely as \Cref{thm: main} in the main body of the paper:
\begin{thm*} 
Let $v,w\in C_n$ and $v$ be $123$-avoiding.
The determinants defining the ideal of $ \mathcal N_{v,w}$ form a Gr\"obner basis with respect to any diagonal term order.
\end{thm*}
This result is proved in \Cref{sec: proof} using $K$-polynomials and the subword complexes of A.~Knutson and E.~Miller \cite{KnutsonMillerAdvances,KnutsonMillerAnnals}.  Note that our conventions are upside-down from those of Knutson and Miller, so our diagonal term orders are indeed analogues of their antidiagonal term orders.

In \cite{KnutsonFrobenius} Knutson showed the defining ideal of any Kazhdan-Lusztig variety has a Gr\"obner basis whose leading terms are squarefree, and, in \cite{Knutson-degenerate}, he determined that the resulting initial ideal is the Stanley-Reisner ideal of a certain subword complex.
However, he did not provide a Gr\"obner basis.
Up to sign, our coordinates agree with the Bott-Samelson coordinates of~\cite{KnutsonFrobenius}.
Hence, our results make \cite[Theorem 7]{KnutsonFrobenius} more explicit by
describing the coordinates and stating which minors in the Gr\"obner basis arise from
each element of the essential set (which corresponds to Knutson's ``basic elements'').

In \Cref{sec: formulas} we define type~C pipe dreams and use them to give consequences to \Cref{thm: main}.
Namely, in \Cref{cor:initIndex} we give prime decompositions of the initial ideals, in \Cref{prop:multidegree formula} we give combinatorial formulas for their multigraded Hilbert series, and in \Cref{prop: K-poly pipe} we give combinatorial formulas for their $K$-polynomials.
Up to a change in convention, these formulas give, in the case where $v$ is $123$-avoiding, combinatorial models of S.~Billey's formula~\cite{Billey}
and its extension to $K$-theory by W.~Graham~\cite{Graham} and M.~Willems~\cite{Willems} for a particular choice of reduced word.  (We note Billey's formula was first stated by H.~Andersen, J.~Jantzen, and W.~Soergel~\cite{AJS} and independently
rediscovered in different but related context by Billey~\cite{Billey}; see J.~Tymoczko's survey
paper~\cite{TymBilley} for details and more recent developments.)
We note that in recent work E.~Smirnov and A.~Tutubalina \cite{SmirnovTutubalina} have studied pipe dreams in all classical groups; these differ from ours even in the special case we describe.

The polynomials given in these formulas also have an interpretation as a particular
specialization of type~C double Schubert and double Grothendieck polynomials, which are
\emph{stable} equivariant Chow \cite[Theorem~10.8]{Ikeda-Mihalcea-Naruse} and K-theory \cite[Theorem~2]{KirillovNaruse} classes of type~C Schubert varieties.  Here, being stable classes means they are lifts of these classes, independent of the rank of the ambient flag variety,
that satisfy certain recurrences and boundary conditions parallel to those holding in the type~A case. 
From these polynomials the multidegrees and K-polynomials
of Kazhdan-Lusztig varieties are obtained geometrically in either of two equivalent ways, 
restricting to affine patches or localization at torus fixed points, or algebraically according to particular restriction maps.
T.~Ikeda, L.~Mihalcea and H.~Naruse \cite{Ikeda-Mihalcea-Naruse} were the first to define type~C double Schubert polynomials,
and they gave several formulas including two using divided difference operators as well as an algebraic restriction map for recovering local classes. 
Type~C double Grothendieck polynomials are due to A.~Kirillov in \cite{kirillov2015double}, which again gives formulas,
though the connection to geometry is made not in that preprint but in Kirillov and Naruse \cite{KirillovNaruse}.  In this interpretation, one can consider our formulas as type~C analogues of
the type~A specialization formulas of A.~Buch and R.~Rimanyi~\cite{Buch-Rimanyi}.

We note that the \emph{symmetric matrix Schubert varieties} defined and studied by Z.~Hamaker, E.~Marberg, and B.~Pawlowski in \cite{MarbergPawlowski,HamakerMarbergPawlowski} are not special cases of the varieties that we consider in the present paper. The varieties studied by Hamaker, Marberg, Pawlowski are defined by imposing northwest rank conditions on symmetric matrices, while we impose southwest (and northeast) conditions. Correspondingly, the pipe dreams they introduce are symmetric across an axis perpendicular to our axis of symmetry. The varieties in \cite{MarbergPawlowski,HamakerMarbergPawlowski} are related to Borel group orbit closures in $G/K$ where $G = GL_n$ is a general linear group and $K = O_n$ is an orthogonal subgroup of $G$.

\subsubsection*{Outline of this paper}
In \Cref{sec: CA background} we give the commutative algebra background for the paper.
In \Cref{sec:KL background} we establish the notation and setup for type~C Kazhdan-Lusztig varieties.
In \Cref{sect:smallPatches} we introduce the coordinates of \Cref{prop: partial symm matrices} which we use for opposite Schubert cells associated to $123$-avoiding permutations.
Then in \Cref{prop: KL coord ring} we use these coordinates to describe the defining ideal of $\mathcal N_{v,w}$ as a generalized symmetric determinantal ideal.
In preparation for computing the $K$-polynomials of our ideals, we describe  in \Cref{sect:TActionSmall} the weights of the coordinates with respect to the action of the torus of diagonal matrices.
In \Cref{sec: type C swc} we give background on subword complexes, and, for the complexes related to our ideals, we describe how to label their vertices using our coordinates. 
We also state the vertex decomposition of subword complexes, which we will use to compare the $K$-polynomials of our ideals with the $K$-polynomials of the Stanley-Reisner ideals for subword complexes.
We then proceed to prove \Cref{thm: main} in \Cref{sec: proof}.
In \Cref{sec: formulas} we introduce type~C pipe dreams for small patches and give various consequences to \Cref{thm: main}. 
Lastly, in \Cref{sec:beyondSmall}, we show that our Gr\"obner basis result (\Cref{thm: main}) does not naturally extend beyond the small patch setting to general type~C Kazhdan-Lusztig ideals.

\subsection*{Acknowledgements}

We thank Patricia Klein, Allen Knutson, and Alex Yong for helpful discussions, and the Fields Institute for hosting us at the Thematic Program on Combinatorial Algebraic Geometry in Fall 2016, where this project was conceived. We also thank an anonymous referee and Emmanuel Neye for helpful feedback on the first version of this paper. 
LE was partially supported by NSF Grants DMS 1855598 and DMS 2142656. JR was partially supported by NSERC Grant RGPIN-2017-05732.  AW was partially supported by Simons Collaboration Grant 359792.
AF received funding from the European Union's Horizon 2020 research and innovation programme under the Marie Sk\l odowska-Curie grant agreement No~792432.

\section{Background}
\label{sec: CA background}

\subsection{Gr\"obner bases and initial ideals}
A \textbf{term order} $<$ is a total order on the monomials in a polynomial ring $R$
with respect to which 1 is minimal and such that if $m,m',m''$ are monomials such that $m'<m''$ then $mm'<mm''$.
One class of term orders we will use are \textbf{lexicographic} term orders.
Given a total ordering $x_1>\cdots>x_k$ on the variables of~$R$,
an exponent vector $(a_1,\ldots,a_k)\in\mathbb N^k$ can be assigned to any monomial $m=x_1^{a_1}\cdots x_k^{a_k}$;
then two monomials compare in the lexicographic term order just as their exponent vectors compare in the lexicographic order on $\mathbb N^k$.
More precisely, $x_1^{a_1}\cdots x_k^{a_k}<x_1^{b_1}\cdots x_k^{b_k}$ if and only if there is some $1\le i\le k$ so that
$a_1=b_1$, \ldots, $a_{i-1}=b_{i-1}$, and $a_i<b_i$.
In particular, the variables themselves still compare as $x_1>\cdots>x_k$ in the lexicographic term order.

The \textbf{initial term} of a polynomial in~$R$, with respect to a fixed term order $<$,
is the maximum of the monomials in its support.
If $I$ is an ideal of~$R$, then its \textbf{initial ideal}, denoted $\init_< I$, is the ideal generated by all initial terms of elements of~$I$.
A \textbf{Gr\"obner basis} for~$I$ is a generating set for $I$ whose initial terms generate $\init_< I$.

\subsection{Torus actions, multigradings, and $K$-polynomials}

One reference for the material in this section is~\cite[Chapter 8]{MillerSturmfels}.

Suppose a torus $T=(\mathbb{K}^*)^n$ acts on affine space $\mathbb{K}^k=\Spec \mathbb{K}[z_1,\ldots,z_k]$ with weights $-a_1,\ldots, -a_k\in\mathbb{Z}^n$.
This means that, given $x=(x_1,\ldots,x_n)\in T$ and $p=\sum_{i=1}^k z_i(p)\,\mathbf{f}_i\in \mathbb{K}^k$ (where $\{\mathbf{f}_i\}$ denotes the dual basis to $\{z_i\}$),
$$t\cdot p=\sum_{i=1}^k x^{-a_i}z_i(p)\,\mathbf{f}_i,$$
where $$x^{-a_i}=\prod_{j=1}^n x_j^{-a_{ij}}.$$
Then $T$ acts on the coordinate functions $z_1,\ldots, z_k$ with weights $a_1,\ldots,a_k$ respectively.
This action induces a $\mathbb{Z}^n$-grading on the ring $R=\mathbb{K}[z_1,\ldots,z_k]$ given by setting the degree of $z_i$ as $a_i$, so that $\deg(z_1^{b_1}\cdots z_k^{b_k})=\sum_{i=1}^k b_ia_i$.

Given $\mathbf{a}\in\mathbb{Z}^n$ and a graded $R$-module $M$, let $M_{\mathbf{a}}$ denote the $\mathbf{a}$-th graded piece of $M$.
Suppose $\dim_{\mathbb{K}}(M_{\mathbf{a}})$ is finite for all $\mathbf{a}$, which will be the case if $a_1,\ldots, a_k$ generate a pointed cone in $\mathbb{Z}^n\otimes\mathbb{R}$ and $M$ is finitely generated.
Then define the {\bf Hilbert series} of $M$ to be
$$\mathcal{H}(M; \mathbf{t}) = \sum_{\mathbf{a}\in\mathbb{Z}^n} \dim_{\mathbb{K}} (M_{\mathbf{a}})\, \mathbf{t}^{\mathbf{a}}.$$
Furthermore define the {\bf $K$-polynomial} of $M$ as
$$\mathcal{K}(M;\mathbf{t})=\mathcal{H}(M;\mathbf{t})\prod_{i=1}^k (1-\mathbf{t}^{\deg(z_i)}).$$  This is a Laurent polynomial in the variables $t_i$.
Finally, the {\bf multidegree} of~$M$ with its given multigrading, denoted $\mathcal{C}(M;\mathbf{t})$, is the sum of all lowest degree terms in $\mathcal{K}(M;1-t_1,\ldots,1-t_k)$.

Note that an ideal and its initial ideal have equal $K$-polynomials and equal multidegrees.  (This is called the \textbf{degenerative} property in \cite{MillerSturmfels}.)
Furthermore, if $N\subseteq M$ and $\mathcal{K}(N;\mathbf{t})=\mathcal{K}(M;\mathbf{t})$, then $N=M$.
Also, multidegrees are \textbf{additive} in the sense that the multidegree $\mathcal{C}(R/I;\mathbf{t})$ is the sum $\sum_{J}\mathcal{C}(R/J;\mathbf{t})$ where the sum is over those $J$ in a primary decomposition of $I$ such that $\sqrt{J}$ is a minimal prime of $I$ that has the same height as $I$ (see \cite[\S 8.5]{MillerSturmfels}).

\subsection{Simplicial complexes and Stanley-Reisner ideals}

A \textbf{simplicial complex} $\Delta$ on the vertex set $V$ is a set of subsets of~$V$, called \textbf{faces},
such that if $F\in \Delta$ then all subsets of $F$ are in~$\Delta$.
A \textbf{facet} of $\Delta$ is a maximal face under containment.
If $\Delta$ is a simplicial complex on~$V$, and $z\not\in V$,
then the \textbf{cone} $\cone_z \Delta$ is the simplicial complex
\[\{F\subseteq V\cup\{z\} : F\cap V\in \Delta\}\]
on vertex set $V\cup\{z\}$.

The \textbf{Stanley-Reisner ideal} of $\Delta$
is the ideal $I_\Delta$ of the polynomial ring $R=\mathbb K[V]$ generated by products of variables that index non-faces of $\Delta$, that is,
\[I_\Delta := \Big\langle\prod_{z\in Z}z : Z\subseteq V, Z\not\in\Delta\Big\rangle.\]

\section{Kazhdan-Lusztig varieties}
\label{sec:KL background}

In this section, we recall background on Schubert varieties in flag varieties of types $A$ and~$C$. 
In particular, we discuss Kazhdan-Lusztig varieties, which we define (following \cite{WooYongSings}) to be the intersection of a Schubert variety with an opposite Schubert cell. 
These are affine varieties.

\subsection{Schubert cells and varieties}

Fix an integer $n\geq 1$, and let $E$ be the $2n\times 2n$ matrix
\[E := \begin{bmatrix}0 &J_n \\ -J_n &0 \end{bmatrix}, \] 
where $J_n$ is the $n\times n$ antidiagonal matrix with antidiagonal entries $1$.
The matrix $E$ determines a non-degenerate, skew-symmetric bilinear form on $\mathbb{K}^{2n}$.
The \textbf{symplectic group} $Sp_{2n}(\mathbb{K})$ is
\[
Sp_{2n}(\mathbb{K}) := \{M\in GL_{2n}(\mathbb{K}) : E(M^{\rm t})^{-1}E^{-1} = M  \},  
\]
or, equivalently, it is the fixed point set of the involution
\[
\sigma: GL_{2n}(\mathbb{K})\rightarrow GL_{2n}(\mathbb{K}),\quad
\sigma(M) = E(M^{\rm t})^{-1}E^{-1}.
\]
Following \cite[Chapter 6]{Lakshmibai-Raghavan} we let $H:=GL_{2n}(\kk)$ and $G:=Sp_{2n}(\kk)$.
We will work with the following Borel subgroups of $H$\/:
	$$
	B^+_H:=\{\text{upper triangular matrices in }H\}
	\quad
	\text{and}
	\quad
	B^-_H:=\{\text{lower triangular matrices in }H\}.
	$$
These give rise to the following Borel subgroups of $G$:
	$$
	B^+_G=(B_H^+)^\sigma
	\qquad
	\text{and}
	\qquad
	B^-_G=(B_H^-)^\sigma.
	$$

Consider the type~A flag variety $H/B^+_H$. 
A \textbf{Schubert cell} in this flag variety is a $B^+_H$-orbit for the left action of $B^+_H$ on $H/B^+_H$ by multiplication, and a \textbf{Schubert variety} is its closure. 
An \textbf{opposite Schubert cell} is a $B^-_H$-orbit in $H/B^+_H$ and an \textbf{opposite Schubert variety} is its closure.
In the type~C flag variety $G/B^+_G$, Schubert cells and varieties are defined analogously by replacing appearances of $H$ and $B^{\pm}_H$ in the above definitions by $G$ and $B^{\pm}_G$ respectively.

Denote by $S_{2n}$ the Weyl group of $H$.
Given $w\in S_{2n}$ we denote by $P(w)$ the permutation matrix having its nonzero entries in positions $(w(j),j)$ for $j=1,\ldots,n$. We use this convention to be consistent with \cite{WooYongGrobner}. 
Each Schubert cell in the type~A flag variety $H/B^+_H$ is equal to some orbit $B^+_H\cdot P(w) B^+_H/B^+_H$ where $w\in S_{2n}$. 
The analogue is true for opposite Schubert cells:
every opposite Schubert cell in $H/B^+_H$ equals
\[\Omega^{A\circ}_v := B^-_H\cdot P(v) B^+_H/B^+_H\]
for some $v\in S_{2n}$. 
We remark that here, and throughout the remainder of the paper, we use the letter ``$w$'' for permutations indexing Schubert cells or varieties, and we use the letter ``$v$'' for permutations indexing opposite Schubert cells.

The Weyl group $C_n$ of $G$ can be identified with the set of permutations
	\begin{equation}\label{eq: symmetry C}
	C_n=\{v_1\ldots v_{2n}\in S_{2n} : v_i=2n+1-v_{2n+1-i} \text{ for } i=1,\ldots,n\}.
	\end{equation}
Equivalently, $C_n$ consists of the $v\in S_{2n}$ such that $w_0vw_0=v$, where $w_0$ is the longest element of $S_{2n}$.
In the type~C flag variety $G/B^+_G$, Schubert and opposite Schubert cells and varieties are indexed by elements of $C_n$.
Concretely, given $w\in C_n$, the permutation matrix $P(w)$ is an element of $G$, so $B^+_G\cdot P(w) B^+_G/B^+_G$ is a Schubert cell, and every Schubert cell is of this form.
The analogous statements hold for Schubert varieties and opposite Schubert cells and varieties.
We denote the type~C opposite Schubert cells by
\[\Omega^\circ_v := B^-_G\cdot P(v) B^+_G/B^+_G\]
for $v\in C_n$.

It is useful to note that type~C Schubert cells and varieties are the $\sigma$-fixed point sets of type~A Schubert cells and varieties. 
See also the treatment in \cite[Theorem 2.5]{Fink-Rajchgot-Sullivant}.

\begin{thm}\cite[Proposition 6.1.1.1]{Lakshmibai-Raghavan}\label{thm: LR} 
The involution $\sigma$ induces a natural involution $\sigma:H/B_H^+\to H/B_H^+$.\footnote{We abuse notation and use $\sigma$ for both maps.}
For $v\in C_n$, the opposite Schubert cell $\Omega^\circ_v$ is stable under $\sigma$ and 	
	$$
	\Omega^\circ_v=(\Omega^{A\circ}_v)^\sigma.
	$$
In other words, $\Omega^\circ_v$ consists of the $\sigma$-fixed points of the type~A opposite Schubert cell $\Omega^{A\circ}_v$.
\end{thm}

Let $X^A_w$ denote the type $A$ Schubert variety $\overline{B_H^+\cdot P(w)B_H^+/B_H^+}$ and let $X_w$ denote the type $C$ Schubert variety $\overline{B_G^+\cdot P(w)B_G^+/B_G^+}$. Following \cite{WooYongSings} (see also \cite{WooYongGrobner}), we refer to the intersection of a Schubert variety with an opposite Schubert cell as a \textbf{Kazhdan-Lusztig variety}. 
We denote the type~A Kazhdan-Lusztig variety as
\[\mathcal N^A_{v,w} = X^A_w\cap\Omega^{A\circ}_v,\]
and the type~C Kazhdan-Lusztig variety as
\begin{equation}\label{eq: type C KL}
 \mathcal N_{v,w} = X^A_w\cap \Omega^\circ_v.
 \end{equation}
Despite the appearances of $H = GL_{2n}(\mathbb{K})$ in the latter intersection above, 
$\mathcal{N}_{v,w}$ is indeed equal to the intersection of a type~C Schubert variety with a type~C opposite Schubert cell. This follows immediately 
from \cite[Proposition 6.1.1.2]{Lakshmibai-Raghavan}, which says that 
\[
X_w = X^A_w\cap G/B_G^+,
\]
as schemes, under the natural inclusion $G/B^+_G\hookrightarrow H/B^+_H$.

We remark that Kazhdan-Lusztig varieties are useful for studying singularities of Schubert varieties using computational algebraic methods. 
This is because a neighborhood of a torus fixed point in a Schubert variety is isomorphic, up to a factor of an affine space, to a Kazhdan-Lusztig variety, which is an affine variety.
This isomorphism is due to D.~Kazhdan and G.~Lusztig \cite[Lemma A.4]{KazhdanLusztig}, and explained in \cite[Section 3]{WooYongSings}. We will describe the prime defining ideals of Kazhdan-Lusztig varieties in Section \ref{sect:KLrank}.

\subsection{Permutations and left-right weak order}\label{sec: weak order}
The \textbf{simple reflections} in $S_m$ are the permutations $s_1, \ldots, s_{m-1}$,
where $s_i$ transposes $i$ and $i+1$.
In $C_n$, define the simple reflections to be $c_0, c_1,\ldots,c_{n-1}$, where 
$c_0\in C_n$ is the permutation that transposes $n$ and $n+1$, 
and for $i=1,\ldots,n-1$, $c_i\in C_n$ is the permutation that transposes $n+i$ with $n+i+1$ (so it must also transpose $n-i$ and $n-i+1$).
We warn the reader that these indexing conventions for $S_m$ and~$C_n$ are different:
under the defining embedding $C_n\subseteq S_{2n}$, the simple reflection $c_i$ is identified with $s_n$ if $i=0$ or $s_{n-i}s_{n+i}$ otherwise,
not with something built from $s_i$.
Both $S_m$ and $C_n$ are generated by their sets of simple reflections.

If $W=S_m$ (resp.\ $C_n$), a \textbf{reduced word} for $v\in W$ is a sequence $Q=({\alpha_1},\ldots,{\alpha_\ell})$
such that $v=s_{\alpha_1}\cdots s_{\alpha_\ell}$ (resp.\ $v=c_{\alpha_1}\cdots c_{\alpha_\ell}$)  and $\ell$ is minimized.
We denote by $\ell_W(v)$ the length of any reduced word for $v\in W$. When there is no chance for confusion, we omit the subscript $W$ from our notation for length. 

Throughout the paper we let $<_{\rm R}$ denote the \textbf{right weak order} on $S_m$;
namely, $u\le_{\rm R} v$ if some prefix of some reduced word for $v$ is a reduced word for $u$.
Similarly, $<_{\rm L}$ denotes the \textbf{left weak order} on $S_m$, which is defined by declaring that $u\le_{\rm L} v$ if some suffix of some reduced word for $v$ is a reduced word for $u$.
The \textbf{left-right weak order} on $S_m$ is denoted throughout the paper by $<$ and defined by $u\le v$ if $v=s_{\alpha_1}\cdots s_{\alpha_a}us_{\beta_1}\cdots s_{\beta_b}$ and $\ell(v)=\ell(u)+a+b$.
We write $u\lessdot v$ if $v$ covers $u$ in left-right weak order.
Both weak orders as well as the two sided weak order on $C_n$ are induced by that on~$S_{2n}$. Thus we use the same notation for them.

We let $\Bruhatl$ denote the Bruhat order; namely,
$v\Bruhatge w$ if the reduced word $Q$ for~$v$ has as a subword a reduced word for~$w$.
Whether this is the case depends only on~$v$, not on the choice of~$Q$.

A simple reflection $c_k$ is a (right) \textbf{ascent} of $v\in C_n$ if $vc_k\Bruhatg v$
and a (right) \textbf{descent} of~$v$ otherwise, namely if $vc_k\Bruhatl v$.
The \textbf{last ascent} of~$v$ is the ascent $c_k$ where $k$ is maximized.
Note that $vc_k$ and $v$ compare the same way in the Bruhat, right weak, and left-right weak orders:
$vc_k$ is either greater than~$v$ in all three or less than~$v$ in all three.

Our convention for the (Rothe) \textbf{diagram} of a permutation $w\in S_m$ is the set
	$$
	D(w)=\{(w(j),i) :i<j \text{ and } w(i)<w(j)\}.
	$$
It is drawn by placing boxes in an $n\times n$ matrix in the positions given by elements of~$D(w)$.
There is a familiar pictorial procedure to obtain $D(w)$ from $P(w)$\/:
one replaces each $1$ by a $\bullet$, deletes all $0$s, and draws at each $\bullet$  the ``hook"  that extends to the east and north of the $\bullet$.
The entries of the matrix that no hook passes through are the elements of $D(w)$.

\begin{ex} The diagram of $w=365124$ is 
	$$
	D(w)=\{(4,1),(5,1),(6,1),(2,4),(4,4),(4,5)\}
	$$
and it is drawn
	$$
	D(w) = 
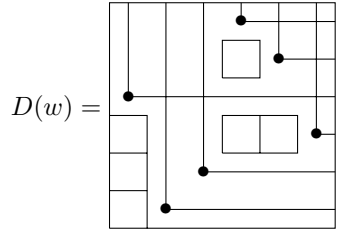
	\begin{tikzpicture}[baseline=(O.base)]
	\node(O) at (1,1.5) {};
	\draw (0,0)--(3,0)--(3,3)--(0,3)--(0,0)
		(.5,0)--(.5,1.5)--(0,1.5)
		(0,.5)--(.5,.5)
		(0,1)--(.5,1)
		(1.5,2)--(2,2)--(2,2.5)--(1.5,2.5)--(1.5,2)
		(1.5,1)--(2.5,1)--(2.5,1.5)--(1.5,1.5)--(1.5,1)
		(2,1)--(2,1.5); 
	\draw (.75,3)--(.75,.25) node {$\bullet$}--(3,.25)
		(.25,3)--(.25,1.75) node {$\bullet$}--(3,1.75)
		(1.25,3)--(1.25,.75) node {$\bullet$}--(3,.75)
		(1.75,3)--(1.75,2.75) node {$\bullet$}--(3,2.75)
		(2.25,3)--(2.25,2.25) node {$\bullet$}--(3,2.25)
		(2.75,3)--(2.75,1.25) node {$\bullet$}--(3,1.25);
	\end{tikzpicture}
	\:.$$
\end{ex}

Rothe diagrams are important to us for providing coordinates for opposite Schubert cells and Kazhdan-Lusztig varieties, as we see in the next subsection.

\subsection{Opposite Schubert cells as spaces of matrices}\label{subsec: opp cells}

Let $H = GL_m(\mathbb{K})$ and $B_H^+\subseteq H$ be the Borel subgroup of upper triangular matrices. 
For $v\in S_m$, let $\Sigma^A_v\subseteq H$ be the set of matrices $M$
such that, if $P(v)_{ij}=1$, then $M_{ij}=1$,
and, otherwise, if $(i,j)\not\in D(v)$, then $M_{ij}=0$.

\begin{prop}\cite[Section 10.2]{FultonYoung}\label{prop: matrix Schuberts}
The map $\pi_H: H\to H/B^+_H$ 
sending a matrix $M$ to its coset $MB^+_H/B^+_H$ induces a (scheme-theoretic) isomorphism from the space of matrices $\Sigma_v^{A}$ to the opposite Schubert cell $\Omega_v^{A\circ}$. 
\end{prop}

We can similarly identify each type~C opposite Schubert cell with a space of matrices using the map $\pi_H$.
We now do this explicitly, to prepare for the explicit coordinate-dependent presentation needed in our main theorem.
The material discussed in this section follows from general theory on algebraic groups and flag varieties, e.g.~\cite[Chapter~13]{Jantzen}, and this particular presentation features in \cite{Billey-Haiman}.

Let $v\in C_n$. Identifying $\Omega_v^\circ$ as a closed subvariety of $\Omega_v^{A\circ}$ by Theorem \ref{thm: LR}, define the space of matrices
	\begin{equation}\label{eq: sigma v}
	\Sigma_v:= \pi_H^{-1}(\Omega_v^\circ),
	\end{equation}
and note that $\Sigma_v$, which is a closed subvariety of $\Sigma^A_v$, is isomorphic to $\Omega_v^\circ$.

Furthermore, we identify Kazhdan--Lusztig varieties with spaces of matrices by letting
$$\Sigma^A_{v,w}:=\pi_H^{-1}(\mathcal{N}^A_{v,w})$$
and
$$\Sigma_{v,w}:=\pi_H^{-1}(\mathcal{N}_{v,w}).$$

We now wish to describe $\Sigma_v$ as the set of $\sigma$-fixed points of $\Sigma_v^A$.
This description will follow from the containment $\sigma(\Sigma_v^A)\subseteq\Sigma_v^A$ for $v\in C_n$.
In order to prove this containment, the following factorization of the matrices in $\Sigma_v^A$ is useful. 
Let $U^-_{i}$ be the unipotent subgroup of $H$ consisting of matrices with 
$1$s along the diagonal, $0$s in all off-diagonal positions except for $(i+1,i)$, and an arbitrary element of $\mathbb{K}$ in position $(i+1,i)$. 

\begin{prop}\label{prop:matrixFactor}
Define $\tilde{\Sigma}_\alpha:=P(s_\alpha)U_\alpha^-$.
Given $v\in S_m$ and $(\alpha_1,\ldots,\alpha_\ell)$ a reduced word for $w_0v$, the map 
	\begin{equation*}
	\mathcal{M}:\tilde{\Sigma}_{\alpha_1}\times\cdots\times\tilde{\Sigma}_{\alpha_\ell}\to\Sigma_v^A,
	\qquad
	(a_1,\ldots,a_\ell)\mapsto P(w_0)a_1\cdots a_\ell
\end{equation*}
is an isomorphism.
\end{prop}

\begin{proof}
We proceed by induction on $\ell(w_0v)$.
The base case is when $v=w_0$, and it is clear that the result holds in this case. 
For the inductive case let $vs_i\gtrdot v$ and write a reduced expression $w_0vs_i=s_{\alpha_1}\cdots	s_{\alpha_\ell}$.
By induction, 
	\begin{equation*}
	\mathcal{M}:\tilde{\Sigma}_{\alpha_1}\times\cdots\times\tilde{\Sigma}_{\alpha_\ell}\to\Sigma_{vs_i}^A,
	\qquad
	(a_1,\ldots,a_\ell)\mapsto P(w_0)a_1\cdots a_\ell
\end{equation*}
is an isomorphism. So, it suffices to show that the image of the multiplication map
\begin{equation}\label{eq:multMap}
m: \Sigma^A_{vs_i}\times \tilde{\Sigma}_{i}\rightarrow H, \quad (a,b)\mapsto ab
\end{equation}
is  $ \Sigma^A_v$ and that it
is an isomorphism upon restricting the codomain to $\Sigma^A_v$. 
To see this, let $a\in \Sigma^A_{vs_i}$ and $b\in \tilde{\Sigma}_{i}$. Assume that the $(i,i)$-entry of the matrix $b$ is equal to  $t\in \mathbb{K}$. 
Observe that $ab$ is obtained from $a$ by performing two elementary column operations: first swap columns $i$ and $i+1$, then replace column $i$ by column $i$  plus $t$ times column $i+1$.
Let $a_{[i,i+1]}$ and $(ab)_{[i,i+1]}$ be the submatrices of $a$ and $ab$ respectively consisting of columns $i$ and $i+1$. 
Because $vs_i\gtrdot v$, after removing all rows of $a_{[i,i+1]}$ and $(ab)_{[i,i+1]}$ which don't have pivots, we are left with:
\[
\begin{bmatrix}
0&1\\
1&0\\
\end{bmatrix} \text{(from $a$)}
\quad
\text{and}
\quad
\begin{bmatrix}
1&0\\
t&1\\
\end{bmatrix} \text{(from $ab$)}.
\]
Thus, every matrix in the image of the map $m$ of~\eqref{eq:multMap} can be factored uniquely as $ab$, 
and so $m$ is an isomorphism onto its image. 
Finally, a straightforward check shows that the locations of diagram boxes in rows without pivots of $a_{[i,i+1]}$ and $(ab)_{[i,i+1]}$ coincide.  (Alternatively, see \cite[Lemma 6.5]{WooYongGrobner}.)  Hence, the image of $m$ is contained in $\Sigma^A_v$. 
As $m$ is an isomorphism onto its image, and $\Sigma^A_v$ and the domain of $m$ are affine spaces of the same dimension, the proposition is proved. 
\end{proof}

\begin{cor}\label{cor:SigmaMap}
Let $H = GL_{2n}(\mathbb{K})$. The map $\sigma: H\rightarrow H$ restricts to an isomorphism
\[
\sigma: \Sigma^A_v\rightarrow \Sigma^A_{w_0vw_0}.
\]
In particular, if $v\in C_n\subseteq S_{2n}$, then $\sigma$ maps $\Sigma^A_v$ isomorphically onto itself.
\end{cor}

\begin{proof}

We first observe, by a straightforward direct check, that $\sigma$ maps $\tilde{\Sigma}_\alpha$ isomorphically onto $\tilde{\Sigma}_{2n-\alpha}$. 
Let $v\in S_{2n}$ and let $(\alpha_1,\dots, \alpha_\ell)$ be a reduced word for $w_0v$. By our observation, we have an isomorphism, 
\begin{equation}\label{eq:diagramisom}
\tilde{\Sigma}^A_{\alpha_1}\times \cdots\times \tilde{\Sigma}^A_{\alpha_\ell}\rightarrow \tilde{\Sigma}^A_{2n-\alpha_1}\times \cdots \times \tilde{\Sigma}^A_{2n - \alpha_\ell}
\end{equation}
Noting that $(2n-\alpha_1,\dots, 2n-\alpha_\ell)$ is a reduced word for $w_0(w_0v)w_0 = vw_0$, the first statement of the corollary follows by applying Proposition \ref{prop:matrixFactor}, which states that the domain of \eqref{eq:diagramisom} is isomorphic to $\Sigma^A_v$ and the codomain is isomorphic to $\Sigma^A_{w_0vw_0}$. 

The second statement follows immediately since $w_0vw_0 = v$ for any $v\in C_n\subseteq S_{2n}$. 
\end{proof}

\begin{cor}\label{cor: type C sigma v}
For $v\in C_n$,
$
\Sigma_v = (\Sigma_v^A)^\sigma. 
$
\end{cor}

\begin{proof}
Let $M\in\Sigma_v$. By \eqref{eq: sigma v} and \Cref{thm: LR}, $MB^+_H/B^+_H\in \Omega_v^\circ=(\Omega^{A\circ}_v)^\sigma$.
This implies that $\sigma(M)B^+_H/B^+_H=MB^+_H/B^+_H$.
Since $M\in \Sigma^A_v$, by \Cref{cor:SigmaMap}, $\sigma(M)\in \Sigma^A_v$. 
We can then apply \Cref{prop: matrix Schuberts} to deduce that $\sigma(M)=M$ and conclude that $M\in (\Sigma_v^A)^\sigma$.
Conversely, let $M\in (\Sigma_v^A)^\sigma$. Then $\sigma(MB^+_H/B^+_H)=\sigma(M)B^+_H/B^+_H=MB^+_H/B^+_H$.
Furthermore, by \Cref{prop: matrix Schuberts} we have that $MB^+_H/B^+_H\in(\Omega^{A\circ}_v)^\sigma=\Omega^\circ_v$.
We conclude that $M\in\Sigma_v=\pi_H^{-1}(\Omega^\circ_v)$.
\end{proof}

We end with two examples of computing $\Sigma_v$. The first shows that, in general, the space of matrices $\Sigma_v$ can be complicated. 
The second shows that for particular choices of~$v$, $\Sigma_v$ is easy to describe. 
From Section~\ref{sect:smallPatches} on we will generally restrict to only these nice $\Sigma_v$.

\begin{ex}\label{ex: type C patch} Given $v=231645$, we have that
	$$
	\Sigma^A_v = \pi_H^{-1}(\Omega_v^{A\circ})
	=
	\left\{ 
	\begin{bmatrix}
	0&0&1 &0&0&0 \\
	1&0&0 &0&0&0\\
	a&1&0 &0&0&0\\
	 b & c & d& 0&1&0 \\
	 e& f&g & 0&h&1 \\
	i&j&k &  1&0&0
	\end{bmatrix}
	: a,b,\ldots,k\in\mathbb{K}
	\right\}.
	$$
Since
	\begin{align*}
	\sigma\left(
	\begin{bmatrix}
	0&0&1 &0&0&0 \\
	1&0&0 &0&0&0\\
	a&1&0 &0&0&0\\
	 b & c & d& 0&1&0 \\
	 e& f&g & 0&h&1 \\
	i&j&k &  1&0&0
	\end{bmatrix}
	\right)
	=
	\begin{bmatrix}
	        0  & 0  & 1  & 0  & 0  & 0\\
        1  & 0  & 0  & 0  & 0  & 0\\
       -h  & 1  & 0  & 0  & 0  & 0\\
 -c h + f  & c  & j  & 0  & 1  & 0\\
-a f + (a c - b) h + e  & -a c + b  & -a j + i  & 0  &       -a  & 1\\
 -d h + g  & d  & k  & 1  & 0  & 0	
	\end{bmatrix},
	\end{align*}
we can equate the entries of $\sigma(M)$ with the entries of $M\in \Sigma^A_v$ to obtain the conditions defining $\Sigma_v$.
It is straightforward to verify that 
	$$
	\Sigma_v
	=
	\left\{
	\begin{bmatrix}
	        0  & 0  & 1  & 0  & 0  & 0\\
        1  & 0  & 0  & 0  & 0  & 0\\
       a  & 1  & 0  & 0  & 0  & 0\\
 	b  & c  & d  & 0  & 1  & 0\\
	e  & -a c + b  & -a d + i  & 0  &       -a  & 1\\
	i& d  & k  & 1  & 0  & 0	
	\end{bmatrix}
	: a,b,c,d,e,i,k\in\mathbb{K}
	\right\}.
	$$
\end{ex}

\begin{ex}\label{eg:symmetricPatch} By a similar computation to the one in the previous example, one can check that the space of matrices $\Sigma_{321654}$ is naturally identified with the space of $3\times 3$ symmetric matrices. That is, 
	\[
	\Sigma_{321654}
	=
	\left\{ 
	\begin{bmatrix}
	0&0&1 &0&0&0 \\
	0&1&0 &0&0&0\\
	1&0&0 &0&0&0\\
	 z_{11} & z_{12} &z_{13} & 0&0&1 \\
	 z_{12} & z_{22} &z_{23} & 0&1&0 \\
	  z_{13} & z_{23} &z_{33} &  1&0&0
	\end{bmatrix}
	: z_{ij}\in\mathbb{K}
	\right\}.
	\]
\end{ex}

\subsection{Rank conditions on type~C Kazhdan-Lusztig varieties}\label{sect:KLrank}
Given $w\in S_{2n}$, let $r_w:\{1,\ldots,2n\}\times\{1,\ldots,2n\}\rightarrow \{1,\ldots,2n\}$ be the {\bf rank function} of $w$, defined by
$$r_w(p,q)=|\{i\leq q: w(i)\geq p\}|,$$
so that $r_w(p,q)$ is the number of entries of $w$ weakly southwest of $(p,q)$.

Given a matrix $M$, let $\tau_{p,q}(M)$ denote the submatrix of entries of~$M$ weakly southwest of position $(p,q)$.
A matrix $M\in\Sigma^A_v$ is in $\Sigma^A_{v,w}$ if and only if, for all $p,q\in[2n]$, 
$\tau_{p,q}(M)$ has rank at most $r_w(p,q)$.
Not all of these rank conditions are necessary to determine $\Sigma^A_{v,w}$.  
In type A, Fulton~\cite{FultonDuke} defined the {\bf essential set}, which gives a smaller set of sufficient conditions, 
as the set of boxes on the northeast%
\footnote{Note that Fulton uses different conventions to ours. 
His hooks emanate east and south rather than east and north,
and he works in $B^-\mathbin\setminus G$ rather than $G/B^-$, so his permutation matrices are the transpose of ours.}
corners of the connected components of $D(w)$.
To be precise, let
$$E^A(w):=\{(p,q)\in D(w) : (p-1,q), (p,q+1) \not\in D(w)\}.$$
Equivalently, one can also define
$$E^A(w)=\{(p,q) : w(q)<p\leq w(q+1), w^{-1}(p-1)\leq q<w^{-1}(p)\}.$$  
Then $M=\Sigma^A_{v,w}$ if and only if
the size $r_w(p,q)+1$ minors of $\tau_{p,q}(M)$ vanish for all $(p,q)\in E^A(w)$, and in fact these equations define $\Sigma^A_{v,w}$ as a subvariety of $\Sigma^A_v$ scheme-theoretically~\cite[Proposition 3.1]{WooYongSings}.

\begin{ex}\label{ex: type A rank conds}
Let $w = 465213$. We have 
\[
	D(w) =
	\begin{tikzpicture}[baseline=(O.base)]
	\node(O) at (1,1.5) {};
	\draw (0,0)--(0.5,0)--(0.5,1.0)--(0,1.0)--(0,0)
		(0,0.5)--(.5,0.5)
		(1.5,1.5)--(2.5,1.5)--(2.5,2)--(1.5,2)--(1.5,1.5)
		(2,2)--(2,1.5); 
	\draw (.25, 3)--(.25,1.25) node {$\bullet$}--(3,1.25)
		(.75,3)--(.75,.25) node {$\bullet$}--(3,.25)
		(1.25,3)--(1.25,.75) node {$\bullet$}--(3,.75)
		(1.75,3)--(1.75,2.25) node {$\bullet$}--(3,2.25)
		(2.25,3)--(2.25,2.75) node {$\bullet$}--(3,2.75)
		(2.75,3)--(2.75,1.75) node {$\bullet$}--(3,1.75);
	\end{tikzpicture}
	\:,
\]
so, the (type~A) essential set of $w$ is $E^A(w) = \{(5,1), (3,5)\}$. 
Furthermore, $r_w(5,1)=0$ and $r_w(3,5)=3$.

Suppose that $v=231645$, as featured in Example~\ref{ex: type C patch}.
Then $M\in\Sigma^A_{v,w}$ if and only if $M\in\Sigma_v^A$ and the size $r_w(p,q)+1$ minors of $\tau_{p,q}(M)$ vanish for all $(p,q)\in E^A(w)$.
In particular, $e=i=0$, and we have 5 additional equations coming from the $4\times 4$ minors of
	$$
  \begin{bmatrix}
	a&1&0 &0&0\\
	 b & c & d& 0&1 \\
	 e& f&g & 0&h \\
	i&j&k &  1&0
	\end{bmatrix}.
  $$
\end{ex}

Recall from \eqref{eq: type C KL} that $\mathcal{N}_{v,w}$ is the intersection of $\Omega_v$ with a type~A Schubert variety. 
Hence, the rank conditions defining $\Sigma_{v,w}$ are the same as those defining $\Sigma_{v,w}^A$, but now applied to $\Sigma_v$ instead of $\Sigma_v^A$. 
In type C, Anderson~\cite{AndersonEJC} showed that a smaller set suffices.  (Some details were made more explicit in~\cite[Section 4]{WooEJC}.)
First, for a permutation $w\in C_n$, boxes of
$E^A(w)$ always come in pairs.  If $(p,q)\in E^A(w)$, then $(2n+2-p,2n-q)\in E^A(w)$, and, furthermore,
$$r_w(2n+2-p,2n-q)=p-q-1+r_w(p,q).\footnote{AW regrets his earlier failure in~\cite{WooEJC} in the perpetual quest to make an even number of sign errors.}$$   
We will choose one box out of each pair by requiring that $p\geq n+1$, and, if $p=n+1$, $q\leq n$.  
Furthermore, if $(p,q)$ and $(p, 2n-q)$ are both in $E^A(w)$ with $p> n+1$ and $q< n$, and $r_w(p,q)=r_w(p,2n-q)-(n-q)$, then $(p, 2n-q)$ is redundant.%

\begin{defn}\label{def: type C essential}
Define $E(w)$ as the subset of $E^A(w)$ consisting of $(p,q)\in E^A(w)$ that satisfy the following conditions%
\footnote{Anderson in \cite[Definition 1.2]{AndersonEJC} and AW in \cite[Section 4]{WooEJC} choose the leftmost box in each pair,
rather than the lower box as we do.}:
\begin{itemize}
\item $p\geq n+1$
\item If $q\geq n+1$ and $(p,2n-q)\in E^A(w)$, then $r_w(p,2n-q)>r_w(p,q)+n-q$.
\end{itemize}
\end{defn}
The second condition subsumes the redundancy condition for $p=n+1$; we always will get equality instead of the desired inequality in that case.

\begin{ex}
Let $w = 465213$ as in Example~\ref{ex: type A rank conds}. 
The (type~C) essential set of $w$ is $E(w) = \{(5,1)\}$. 
Suppose that $v=321654$ as in Example~\ref{eg:symmetricPatch} .
Then $M\in \Sigma_{v,w}$ if and only if $M\in\Sigma_v$ and the size $1$ minors of $\tau_{5,1}(M)$ vanish. Thus,
	\[
	\Sigma_{v,w}
	=
	\left\{ 
	\begin{bmatrix}
	0&0&1 &0&0&0 \\
	0&1&0 &0&0&0\\
	1&0&0 &0&0&0\\
	 z_{11} & 0 &0 & 0&0&1 \\
	 0 & z_{22} &z_{23} & 0&1&0 \\
	  0 & z_{23} &z_{33} &  1&0&0
	\end{bmatrix}
	: z_{ij}\in\mathbb{K}
	\right\}.
	\]
\end{ex}

\begin{ex}
Let $w=426153$.  The type~A essential set of $w$ is $E^A(w)=\{(3,2), (3,4), (5,2), (5,4)\}$.  The first condition that $p\geq n+1$ eliminates $(3,2)$ and $(3,4)$ (whose conditions are equivalent to those given by $(5,4)$ and respectively $(5,2)$).  Note that $(5,4)$ does not satisfy the second condition, since $q=4\geq n+1$, $(p,2n-q)=(5,2)\in E^A(w)$, and $r_w(5,2)=0=r_w(5,4)+n-q=1+3-4$.  Hence $E(w)=\{(5,2)\}$.

If we let $v=321654$ as in Example~\ref{eg:symmetricPatch}, we see that
the condition $r_w(5,2)=0$ forces $z_{12}=z_{22}=z_{13}=z_{23}=0$, and this automatically forces $r_w(5,4)=1$ (noting
that $z_{23}$ appears in two places in the matrix), indicating that the condition from $(5,4)$ is redundant.  In particular,
	\[
	\Sigma_{v,w}
	=
	\left\{ 
	\begin{bmatrix}
	0&0&1 &0&0&0 \\
	0&1&0 &0&0&0\\
	1&0&0 &0&0&0\\
	 z_{11} & 0 &0 & 0&0&1 \\
	 0 & 0 & 0 & 0&1&0 \\
	  0 & 0 &z_{33} &  1&0&0
	\end{bmatrix}
	: z_{ij}\in\mathbb{K}
	\right\}.
	\]
\end{ex}

\section{Small patches}\label{sect:smallPatches}

Let $v_\square\in C_n$ denote the \textbf{square word} permutation, whose permutation matrix is
	\[
	P(v_\square)=\begin{bmatrix}J_n  & 0 \\ 0 & J_n \end{bmatrix}.
	\]
In this section,
we discuss various properties of type~C opposite Schubert cells $\Omega^\circ_v$ where $v \ge v_\square$ in left-right weak order. We refer to such opposite Schubert cells as \textbf{small patches}. 

The purpose of this section is for us to fix explicit coordinates and conventions. In addition to being crucial in our main theorem on Gr\"obner bases, our choice of coordinates yields a natural identification between small patches and symmetric ladders from the commutative algebra literature \cite{Gorla,Gorla-Migliore-Nagel}.

\subsection{Small patches and symmetric matrices}
To choose coordinates on type~A opposite Schubert cells, 
it is enough to take a distinct indeterminate for each element of~$D(v)$
(see \cite[Section~2.2]{WooYongGrobner}). 
In this section, we put specific coordinates on type~C opposite Schubert cells $\Omega_v^\circ$ when $v\geq v_{\square}$.

By Theorem~\ref{thm: LR}, the type C opposite Schubert cell $\Omega_{v}^\circ$ of $v\in C_n$ consists of the $\sigma$-fixed points of the type~A cell $\Omega^{A\circ}_{v}$.
For $v_\square$, the type~A cell $\Omega^{A\circ}_{v_\square}$ is isomorphic to the set of matrices
	\[
	\Sigma^A_{v_\square}
	=\left\{ 
	\begin{bmatrix}
	J_n & 0 \\ M & J_n
	\end{bmatrix}
	: M \text{ is any $n\times n$ matrix}
	\right\}.
	\]
Applying Corollary~\ref{cor: type C sigma v} we can directly compute the $\sigma$-fixed points of $\Sigma^A_{v_\square}$ to show the following result.  (See Example \ref{eg:symmetricPatch} for the $n=3$ case).

\begin{prop}\label{prop: square patch}
	$
	\Sigma_{v_\square}
	=
	\left\{ 
	\begin{bmatrix}
	J_n & 0 \\ Z & J_n
	\end{bmatrix}
	: Z \text{ is a symmetric } n\times n \text{ matrix}
	\right\}.
	$
\end{prop}

Our next goals is to give explicit coordinates for $\Sigma_v$ whenever $v\ge v_\square$.  This will show that, by deleting certain rows and columns with no variables, matrices
in $\Sigma_v$ can be identified with partial symmetric matrices, and our coordinates will be entries of these matrices.  For the rest of this paper we restrict to such~$v$. 
A permutation $v\in S_m$ is {\bf $123$-avoiding} if there do not exist $i<j<k$ such that $v(i)<v(j)<v(k)$.  A {\bf left-to-right minimum} of a permutation is an index $a$ such that $v(i)>v(a)$ for all $i<a$; a {\bf right-to-left maximum} is an index $b$ such that $v(k)<v(b)$ for all $k>b$.

\begin{lemma}\label{lem: 123 shuffle}
If $v\in C_n$ then $v$ is $123$-avoiding if and only if there exist left-to-right minima $a_1<\cdots<a_n$ such that $a_i\neq 2n+1-a_j$ for any $i$ and $j$.  If we let $b_i=2n+1-a_{n+1-i}$ for all $i$, then $b_1<\cdots<b_n$ are right-to-left maxima.
\end{lemma}

Note that, by definition of left-to-right minima and right-to-left maxima, $v(a_1)>\cdots >v(a_n)$ and $v(b_1)>\cdots >v(b_n)$.

\begin{proof}
It is a classical result that every index in a $123$-avoiding permutation $v\in S_m$ is a left-to-right minimum or a right-to-left maximum.  Indeed, if $j$ is neither, then by definition there exists $i<j$ with $v(i)<v(j)$ and there exists $k>j$ with $v(k)>v(j)$, so $v$ is not $123$-avoiding.  For $v\in C_n$, whenever $a$ is a left-to-right minimum, $2n+1-a$ is a right-to-left maximum by definition.  Hence, for all $j$ with $1\leq j\leq n$, if only one of $j$ and $2n+1-j$ is a left-to-right minimum, we can let the left-to-right minimum be $a_i$ for some $i$, and if both $j$ and $2n+1-j$ are left-to-right minima, we can arbitrarily choose one to be one of the $a_i$.
\end{proof}

\begin{prop}\label{lem: shuffle}
If $v\in C_n$ then $v\ge v_\square$ if and only if $v$ is $123$-avoiding.
Moreover, if  $v=u_lv_\square u_r$, $\ell(v)=\ell(u_\ell)+\ell(v_\square)+\ell(u_r)$, and we set $a_i=u_r^{-1}(i)$ and $b_i=u_r^{-1}(n+i)$, then the $a$'s and $b$'s are as in \Cref{lem: 123 shuffle}.
\end{prop}

\begin{proof} 
Supposing that $v\ge v_\square$, we will prove that $v$ is $123$-avoiding by induction on $\ell(v)-\ell(v_\square)$.
Notice that the statement is true for $v_\square$.
For the inductive step, suppose $v=u_lv_\square u_r$ is $123$-avoiding and $w\gtrdot v\ge v_\square$. Then there exists $c_d$ such that $w=vc_d$ or $w=c_dv$.
First, suppose that $w=vc_d$.
Since $vc_d>v$ we must have picked $d$ such that $v(n-d)<v(n-d+1)$.  Note $n-d$ cannot be a right-to-left maximum and $n-d+1$ cannot be a left-to-right minimum, so $n-d=a_j$ and $n-d+1=b_{n+1-k}$ for some $j,k$ (which implies that $n+d=a_k$ and $n+d+1=b_{n+1-j}$).
Notice that we obtain the desired sequences $a^\prime_1<\cdots < a^\prime_n$ and $b^\prime_1<\cdots <b^\prime_n$ for $w=vc_d$ by taking $a^\prime_i=a_i$ and $b^\prime_i=b_i$ for all $i$, except that $a^\prime_j=n-d+1$, $b^\prime_{n+1-k}=n-d$, $a^\prime_k=n+d+1$, and $b^\prime_{n+1-j}=n+d$.
We conclude that $w=vc_d$ is $123$-avoiding.  Furthermore, $w=u_lv_\square u^\prime_r$ where $u^\prime_r=u_rc_d$, and we have $a^\prime_i=a_i=u_r^{-1}(i)=(u^\prime_r)^{-1}(i)$ for $i\neq j,k$, while $a^\prime_j=n-d+1=c_d(a_j)=c_d(u_r^{-1}(j))=(u^\prime_r)^{-1}(j)$, and similarly $a^\prime_k=(u^\prime_r)^{-1}(k)$.

Continuing the inductive step, suppose that $w=c_dv$. 
Since $c_dv>v$ we must have picked $d$ such that $v^{-1}(n-d)<v^{-1}(n-d+1)$.
Since $v(v^{-1}(n-d))<v(v^{-1}(n-d+1))$, then $v^{-1}(n-d)$ is not a right-to-left maximum and $v^{-1}(n-d+1)$ is not a left-to-right minimum, so $v^{-1}(n-d)=a_j$ and $v^{-1}(n-d+1)=b_{n+1-k}$ for some $j,k$ (which implies that $v(a_k)=n+d$ and $v(b_{n+1-j})=n+d+1$).
Observe then that for all $i$, $c_dv(a_i)>c_dv(a_{i+1})$ and $c_dv(b_i)>c_dv(b_{i+1})$. 
In this case we obtain the desired sequences for $w=c_dv$ by keeping the $a$'s and $b$'s for $v$.  Note that $u_r$ is unchanged. 
This proves the second statement and the forward direction of the first statement.

Now suppose that $v\in C_n$ is $123$-avoiding.
Choose $a$'s and $b$'s for $v$, as in \Cref{lem: 123 shuffle}.
We will provide an algorithm that produces $\alpha_1,\ldots,\alpha_s$ and $\beta_1,\ldots,\beta_t$ such that $v_\square=c_{\alpha_s}\cdots c_{\alpha_1}v c_{\beta_1}\cdots c_{\beta_t}$ and $\ell(v)=\ell(v_\square)+s+t$, thus proving that $v\ge v_\square$.

We start by finding the $\alpha$'s.  If $v(a_1)=n$, we will have no $c_{\alpha}$'s.  Otherwise, since $v(a_1)>\cdots>v(a_n)\geq 1$, we must have $v(a_1)>n$.
Let $j=\max\{k\in[n] : v(a_k)>n+1-k\}$, and let $c
\in C_n$ be the generator that transposes $v(a_j)-1$ with $v(a_j)$ (so
it transposes $v(2n+1-a_j)$ with $v(2n+1-a_j)+1$).  
Since $a_j$ is a left-to-right minimum, $v^{-1}(v(a_j)+1)>a_j$, so $cv<v$.  
By the construction of $j$, we have $v(a_j)-1\neq v(a_{j+1})$ and hence $v^{-1}(v(a_j)-1)$ is not a left-to-right minimum for $v$.  
Hence $cv$ is also $123$-avoiding as witnessed by the same indices $a_i$ and $b_i$.
Set $\alpha_1$ so that $c=c_{\alpha_1}$.
Iterate this process, without changing the $a_i$ and $b_i$, until $v(a_k)=n+1-k$ for all $k\in[n]$.

Starting with the output $v$ of the previous paragraph we now find the $\beta$'s.
Note that this output has the property that $v(a_i)=n+1-i$ (and $v(b_i)=2n+1-i$) for all $i\in[n]$, and this property will be maintained throughout the process.
Let $j=\min\{k\in[n] : a_k>k\}$, and let $c
\in C_n$ be the generator that transposes $a_j$ with $a_j-1$ (so it must also transpose $b_{n+1-j}$ with $b_{n+1-j}+1$).
First, note that since $a_j$ is a left-to-right minimum, $v(a_j-1)>v(a_j)$, so $vc<v$.
Second, note that $a_j-1$ is not a left-to-right minimum by definition of~$j$, so $a_j-1=b_k$ for some $k\in[n]$ and $v(a_j-1)>n$.  
Take $a^\prime_i=a_i$ and $b^\prime_i=b_i$ for all $i$, except that $a^\prime_j=b_k$, $b^\prime_{k}=a_j$, $a^\prime_{n+1-k}=b_{n+1-j}$, and $b^\prime_{n+1-j}=a_{n+1-k}$.
Then $vc$ is $123$-avoiding and $vc(a^\prime_j)=n+1-j$ (and $vc(b^\prime_{n+1-j})=2n+1-j$).
Set $\beta_1
$ to be the index such that $c=c_{\beta_1}$. Iterate the process starting with $vc_{\beta_i}\cdots c_{\beta_1}$ until no longer possible, so until $a_j=j$ for all $j\in[n]$. 
Note then that the algorithm terminates with $v_\square$ and the proposition follows.
\end{proof}

One can consider Proposition~\ref{lem: shuffle} as a type C analogue of \cite[Theorem 2.1]{Billey-Jockusch-Stanley}, which characterizes 321-avoiding permutations as those smaller (in left-right weak order) than the maximal grassmannian permutation for some descent. This directly implies the ``only if'' direction of our first sentence, but in type~A the choice of maximal grassmannian permutation can depend on~$v$, whereas in type~C, $v_\square$ is the only choice.
Also unlike type~A, the conditions of \Cref{lem: shuffle} are not equivalent to $w_0v$ being fully commutative: $v=1324\in C_2$ does not satisfy the proposition although $w_0v=4231=c_1c_0c_1$ is fully commutative.

The following corollary describes the diagram for $v\ge v_\square$; see \Cref{fig: skew partition} for an example.

\begin{cor}	\label{cor: skew partition} If $v\in C_n$ then $v\ge v_\square$ if and only if $D(v)$ becomes a skew partition after deleting all the rows and columns that do not contain boxes of $D(v)$. 
\end{cor}

\begin{proof}
This is a direct consequence of the previous proposition together with the note following \cite[Theorem 2.1]{Billey-Jockusch-Stanley}. We remark that we obtain $123$-avoiding permutations instead of $321$-avoiding ones due to the difference in our conventions for $D(v)$.
\end{proof}

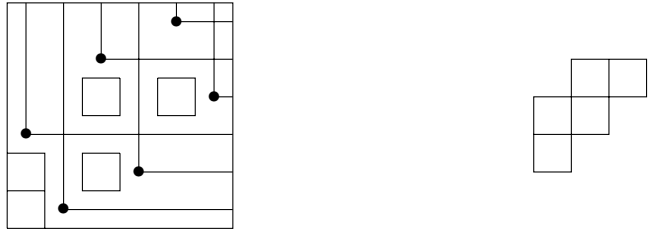
\begin{figure}[h]
	\begin{tikzpicture}[baseline=(O.base)]
	\node(O) at (1,1.5) {};
	\draw (0,0)--(3,0)--(3,3)--(0,3)--(0,0)
		(.5,0)--(.5,1)--(0,1)
		(0,.5)--(.5,.5)
		(1.5,1)--(1,1)--(1,.5)--(1.5,.5)--(1.5,1)
		(1.5,2)--(1,2)--(1,1.5)--(1.5,1.5)--(1.5,2)
		(2,2)--(2.5,2)--(2.5,1.5)--(2,1.5)--(2,2)
		; 
	\draw 
		(.25,3)--(.25,1.25) node {$\bullet$}--(3,1.25)
		(.75,3)--(.75,.25) node {$\bullet$}--(3,.25)
		(1.25,3)--(1.25,2.25) node {$\bullet$}--(3,2.25)
		(1.75,3)--(1.75,.75) node {$\bullet$}--(3,.75)
		(2.25,3)--(2.25,2.75) node {$\bullet$}--(3,2.75)
		(2.75,3)--(2.75,1.75) node {$\bullet$}--(3,1.75);	
	\def\l{7}
	\def\u{.75}
	\draw (\l,\u+.5)--(\l,\u+1)--(\l+1.5,\u+1)--(\l+1.5,\u+1.5)--(\l+.5,\u+1.5)--(\l+.5,\u)--(\l,\u)--(\l,\u+.5)--(\l+1,\u+.5)--(\l+1,\u+1.5)
		;
	\end{tikzpicture}
	\caption{On the left we have the diagram of $v=462513>v_\square$ and on the right its associated skew partition.}
	\label{fig: skew partition}
\end{figure}

Throughout the paper we denote by $\overline{v}$ a factorization $v=u_lv_\square u_r$ such that $\ell(v)=\ell(u_l)+\ell(v_\square)+\ell(u_r)$.
 We let
  $$R_{\overline{v}}=\mathbb{K}[z_{ij} : i\leq j,\ u_r^{-1}(i)<u_r^{-1}(2n+1-j),\ u_l(n+1-i)<u_l(n+j)].$$
Furthermore, let $M_{\overline{v}}$ be the matrix with $z_{ij}$ as the entries at $(u_l(n+j), u_r^{-1}(i))$ and $(u_l(n+i),u_r^{-1}(j))$ whenever $u_r^{-1}(i)<u_r^{-1}(2n+1-j)$ and $u_l(n+1-i)<u_l(n+j)$, 
1s at $(v(i),i)$ for all $i$, and 0s at all other positions. 
Note that the $z_{ij}$ only appear within $D(v)$.
For some examples, the general element of $\Sigma_{321654}$ in \Cref{eg:symmetricPatch} and the matrix on the left hand side of the equality in \Cref{ex:furthermore} are both of the form $M_{\overline{v}}$.

Note that given $v\ge v_\square$, a choice of $a_1,\ldots,a_n$ and $b_1,\ldots, b_n$ as in \Cref{lem: 123 shuffle} is equivalent to choosing a factorization $\overline{v}$. 
In this language
 $$R_{\overline{v}}=\mathbb{K}[z_{ij} : i\leq j,\ a_i<b_{n+1-j},\ v(a_i)<v(b_{n+1-j})]$$
and $M_{\overline{v}}$ is the matrix with $z_{ij}$ as the entries at $( v(b_{n+1-j}),a_i)$ and $( v(b_{n+1-i}),a_j)$ whenever $a_i<b_{n+1-j}$ and $v(a_i)<v(b_{n+1-j})$, 
1s at $(v(i),i)$ for all $i$, and 0s at all other positions. 

\begin{ex}\label{ex: variable labels change}
This example shows how the labeling of the coordinates in $R_{\overline{v}}$ depends on the choice of our factorization $\overline{v}$. 
The factorization $\overline{v}^{(1)}$ corresponding to $v=642531>v_\square$, $(a_{\bullet})=(1,2,3)$, and $(b_{\bullet})=(4,5,6)$ is $\overline{v}^{(1)}=u_l^{(1)} v_\square u_r^{(1)}$ where $u_l^{(1)}=246135$ and $u_r^{(1)}=123456$.  Then
	$$M_{\overline{v}^{(1)}}=\begin{bmatrix}
	0&0&0 &0&0&1 \\
	0&0&1 &0&0&0\\
	0&0&z_{23} &0&1&0\\
	 0 & 1 &0 & 0&0&0 \\
	 0 & z_{23} &z_{33} & 1&0&0 \\
	  1 & 0 &0 &  0&0&0
	\end{bmatrix}.$$
By comparison, for $(a_{\bullet})=(2,3,6)$ and $(b_{\bullet})=(1,4,5)$ we have that $\overline{v}^{(2)}=u_l^{(2)}v_\square u_r^{(2)}$ where $u_l^{(2)}=124356$ and $u_r^{(2)}=412563$.  Then
	$$M_{\overline{v}^{(2)}}=\begin{bmatrix}
	0&0&0 &0&0&1 \\
	0&0&1 &0&0&0\\
	0&0&z_{12} &0&1&0\\
	 0 & 1 &0 & 0&0&0 \\
	 0 & z_{12} &z_{22} & 1&0&0 \\
	  1 & 0 &0 &  0&0&0
	\end{bmatrix}.$$
\end{ex}

One can relate the variables to the self-conjugate skew partition associated to $v$ by Corollary~\ref{cor: skew partition} as follows.
If the skew partition has $n$ rows (or equivalently $n$ columns), then there is only one choice for $\overline{v}$, and $z_{ij}$ is a variable if and only if $(i,j)$ (equivalently $(j,i)$) is a box of the skew partition.
If the skew partition has fewer rows, then different choices of $\overline{v}$ will give rise to different (but always self-conjugate) placements of the self-conjugate skew partition in an $n\times n$ box corresponding to different coordinates.

The matrix $M_{\overline{v}}$ satisfies the following property.

\begin{prop}\label{rows and columns with ones}
If $1\le\delta\le 2n$ is a left-to-right minimum of $v$ then the only nonzero entry of~$M_{\overline{v}}$ in row~$v(\delta)$ is the $1$ at position $(v(\delta), \delta)$.
Similarly, if $1\le\delta\le 2n$ is a right-to-left maximum of $v$, then the only nonzero entry of~$M_{\overline{v}}$ in column~$\delta$ is the $1$ at position $(v(\delta), \delta)$.
\end{prop}

\begin{proof} 
Suppose that $1\le\delta\le 2n$ is a left-to-right minimum of $v$.
Entries to the right of this position lie on its hook, while an entry $(\epsilon,\delta)$ to its left lies on the hook extending from~$(\epsilon,v^{-1}(\epsilon))$ by the left-to-right minimum condition.
An analogous argument proves the second part of the lemma.
\end{proof}

We now give concrete coordinates for the coordinate ring of $\Sigma_v$.

\begin{prop}\label{prop: partial symm matrices} If $v\ge v_\square$, then $R_{\overline{v}}$ is a coordinate ring of $\Sigma_v$ and $M_{\overline{v}}$ is the generic matrix in $\Sigma_v$.  (In other words, a matrix is in $\Sigma_v$ if and only if it can be obtained by setting each variable in $M_{\overline{v}}$ to some element of $\mathbb{K}$.)  
Furthermore, if $v=u_lv_\square u_r$ and $\ell(v)=\ell(u_l)+\ell(v_\square)+\ell(u_r)$, then the rule $M\mapsto P(u_l^{-1})MP(u_r^{-1})$ 
induces the injective map from $\Sigma_v$ to $\Sigma_{v_\square}$ which identifies the entry named $z_{ij}$ in $M_{\overline{v}}$ with the entry named $z_{ij}$ in $M_{\overline{v_{\square}}}$.
\end{prop}

Notice that $v\not\ge v_\square$ in Example~\ref{ex: type C patch},
and indeed the entries of the general matrix in~$\Sigma_v$ in that example cannot all be made to be variables.
An interesting question is to describe the entries of $\Sigma_v$ for general $v\in C_n$ and give a Gr\"obner basis for Kazhdan-Lusztig varieties arising from these cells.

Before proving this result, let us give an example and some necessary lemmas.

\begin{ex}\label{ex:furthermore}
Let $v = {462513}$. Then $c_0v_{\square}c_0c_1 = v$. The following equality illustrates the ``furthermore'' part of \Cref{prop: partial symm matrices}:
\[
 P(c_0) 
\begin{bmatrix}
	0&0&0 &0&1&0 \\
	0&0&1 &0&0&0\\
	0&0&z_{12} &0&z_{13}&1\\
	 1 & 0 &0 & 0&0&0 \\
	 z_{12} & 0 &z_{22} & 1&0&0 \\
	  z_{13} & 1 &0 &  0&0&0
	\end{bmatrix} P(c_1c_0) = 
	\begin{bmatrix}
	0&0&1&0&0&0 \\
	0&1&0 &0&0&0\\
	1&0&0 &0&0&0\\
	 0 & z_{12} &z_{13} & 0&0&1 \\
	 z_{12} & z_{22} &0& 0 &1&0 \\
	  z_{13} & 0 &0 &  1&0&0
	\end{bmatrix}
\]
\end{ex}
In general, the effect of right multiplication by $P(u_r^{-1})$ 
is to collect at the left side all columns containing any variable $z_{ij}$,
and similarly left multiplication by $P(u_l^{-1})$ collects rows with variables at the bottom.

The following lemmas will be used to prove \Cref{lem: before lunch}. They are adaptations of \cite[Lemma 6.5]{WooYongGrobner} to type~C.
\begin{lemma}\label{lem: weak cover diagram}
Let $v\in C_n$ and $k$ be such that  $vc_k\gtrdot v$ in right weak order.
The diagram $D(vc_k)$ is obtained from $D(v)$ as follows:
	$D(vc_k)$ agrees with $D(v)$ except in columns $n\pm k$ and $n\pm k+1$. 
	To obtain columns $n-k$ and $n-k+1$ of $D(vc_k)$, move all the boxes of $D(v)$ in column $n-k$ strictly above row $v(n-k+1)$ one unit to the right and delete the box in position $(v(n-k+1),n-k)$. 
	Repeat the analogous process in columns $n+k$ and $n+k+1$ (if $k\neq0$).
\end{lemma}

We need the analogous lemma for left weak order as well:
\begin{lemma}\label{lem: left weak cover diagram}
Let $v\in C_n$ and $k$ be such that  $c_kv\gtrdot v$ in left weak order.
The diagram $D(c_kv)$ is obtained from $D(v)$ as follows:
	$D(c_kv)$ agrees with $D(v)$ except in rows $n\pm k$ and $n\pm k+1$. 
	To obtain rows $n-k$ and $n-k+1$ of $D(c_kv)$, move all the boxes of $D(v)$ in row $n+1-k$ strictly right of column $v^{-1}(n-k)$ one unit up and delete the box in position $(n+1-k,v^{-1}(n-k))$. 
	Repeat the analogous process in rows $n+k$ and $n+k+1$ (if $k\neq 0$).
\end{lemma}

\begin{proof}[Proof of \Cref{prop: partial symm matrices}]
To prove that $M_{\overline{v}}$ is the generic matrix in $\Sigma_v$ it suffices to show that $\sigma(M_{\overline{v}})=M_{\overline{v}}$ and the entries of $M_{\overline{v}}$ are as in the beginning of \Cref{subsec: opp cells}, 
implying that, regardless how we evaluate the $z_{ij}$ in $M_{\overline{v}}$, we get a matrix in~$\Sigma^A_v$.
The second statement follows by definition, since $(M_{\overline{v}})_{ij}=1$ whenever $P(v)_{ij}=1$ and $(M_{\overline{v}})_{ij}=0$ whenever $(i,j)\notin D(v)$.
For the first statement we proceed by induction on $\ell(v)$. 
The base case, when $v=v_\square$, is trivial.
For the inductive case, consider some $v> v_\square$, and let $c_k$ be a simple reflection such that either $vc_k\gtrdot v$ or $c_kv\gtrdot v$.  
By the inductive hypothesis, $M_{\overline{v}}$ is the generic matrix in $\Sigma_{v}$.
Throughout the proof we will fix a factorization $\overline{v}=u_lv_\square u_r$ with corresponding sequences $(a_{\bullet})$ and $(b_{\bullet})$.

First suppose that $vc_k\gtrdot v$ so that $v(n- k)<v(n- k+1)$ and $v(n+ k)<v(n+ k+1)$.
Let $\overline{vc_k}$ be the factorization $(u_l)v_\square(u_rc_k)$.
Since $v$ is $123$-avoiding, $n\pm k$ are left-to-right minima, and $n\pm k +1$ are right-to-left maxima.
Therefore, $(v(n- k+1),n- k)=(v(b_{n+1-j}),a_i)$ and $(v(n+ k+1),n+ k)=(v(b_{n+1-i}),a_j)$ for some $i,j$, and the corresponding entries in $(M_{\overline{v}})$ are the variable $z_{ij}$ or $z_{ji}$.
Without loss of generality we assume that $i\le j$.
Our goal is to show that $M_{\overline{vc_k}}P(c_k)$ is obtained from $M_{\overline{v}}$ by setting $z_{ij}=0$.
Right multiplication by $P(c_k)$ swaps column $n+k$ with $n+k+1$ and column $n-k$ with $n-k+1$.  
Since $P(vc_k)P(c_k)=P(v)$ then the positions of the $1$s in $M_{\overline{vc_k}}P(c_k)$ and $M_{\overline{v}}$ agree.
By \Cref{lem: weak cover diagram}, 
	$$
	D(vc_k)\setminus [n]\times\{n\pm k, n\pm k+1\}
	=
	D(v)\setminus [n]\times\{n\pm k, n\pm k+1\},
	$$
and therefore the positions of $0$ entries in $M_{\overline{vc_k}}P(c_k)$ and $M_{\overline{v}}$ agree on all columns, except possibly columns $n\pm k, n\pm k+1$.
This lemma also implies that, for $\delta\in\{n\pm k,n\pm k+1\}$, if $(M_{\overline{v}})_{\epsilon\delta}=0$, then $(M_{\overline{vc_k}}P(c_k))_{\epsilon\delta}=0$.

It remains to analyze the variable entries of $M_{\overline{v}}$.
Since $n\pm k+1$ are right-to-left maxima for $v$, columns $n\pm k+1$ of $M_{\overline{v}}$ do not contain any variables. 
Similarly, since $n\pm k$ are right-to-left maxima for $vc_k$, columns $n\pm k+1$ of $M_{\overline{v}}P(c_k)$ do not contain any variables and are therefore equal to columns $n\pm k+1$ of $M_{\overline{v}}$.
Note that the sequences $(a^\prime_{\bullet})$ and $(b^\prime_{\bullet})$ from \Cref{lem: shuffle} for $\overline{vc_k}$ agree with the sequences  $(a_{\bullet})$ and $(b_{\bullet})$ everywhere except
	$$a^\prime_i=b_{n+1-j}=n-k+1,\quad b^\prime_{n+1-j}=a_i=n-k,\quad a^\prime_{j}=b_{n+1-i}=n+k+1, \text{ and }\quad b^\prime_{n+1-i}=a_{j}=n+k.$$
If $\delta\neq n\pm k,n+1\pm k$ and $(M_{\overline{v}})_{\epsilon\delta}=z_{i'j'}$, then $(\epsilon,\delta)\in D(v)$ and $(\epsilon,\delta)\in\{(v(b_{n+1-j'}),a_{i'}),(v(b_{n+1-i'}),a_{j'})\}$.
Combining $(\epsilon,\delta)\in D(vc_k)$, which follows from \Cref{lem: weak cover diagram}, with $vc_k(b^\prime_{m})=v(b_{m})$ for all $m$, we have that $(\epsilon,\delta)\in\{(vc_k(b^\prime_{n+1-j'}),a^\prime_{i'}),(vc_k(b^\prime_{n+1-i'}),a^\prime_{j'})\}$.
It follows that $(M_{\overline{vc_k}})_{\epsilon\delta}=z_{i'j'}$.
Finally, assume $\delta=n\pm k$ and $(M_{\overline{v}})_{\epsilon\delta}=z_{i'j'}$ so that $(\epsilon,\delta)\in D(v)$ and by \Cref{lem: weak cover diagram} $(\epsilon,\delta+1)\in D(vc_k)$, except if $\epsilon=v(\delta)$.
If $\epsilon=v(\delta)$ then $z_{i'j^\prime}=z_{ij}$ and we obtain $(M_{\overline{vc_k}}P(c_k))_{\epsilon\delta}$ by setting $z_{ij}=0$, as desired.
If $\epsilon\neq v(\delta)$ then $\epsilon\in\{v(b_{n+1-j'}),v(b_{n+1-i'})\}=\{vc_k(b'_{n+1-j'}),vc_k(b'_{n+1-i'})\}$ and therefore $(M_{\overline{vc_k}})_{\epsilon\delta}=z_{i'j'}$.
We conclude that $M_{\overline{vc_k}}P(c_k)$ is obtained from $M_{\overline{v}}$ by setting $z_{ij}=0$ if it lies in positions $(v(n\pm k+1),n\pm k)$.

We are left with proving that $\sigma(M_{\overline{vc_k}})=M_{\overline{vc_k}}$.  However, by induction, $\sigma(M_{\overline{v}})=M_{\overline{v}}$, and therefore $\sigma(M_{\overline{vc_k}}P(c_k))=M_{\overline{vc_k}}P(c_k)$ by the argument above.  
Hence, since $\sigma$ is a group homomorphism and $\sigma(P(c_k))=P(c_k)$, we conclude $\sigma(M_{\overline{vc_k}})=M_{\overline{vc_k}}$.
Finally, returning to the case $c_kv\gtrdot v$,
using \Cref{lem: left weak cover diagram} and other similar arguments, one can show that in this case
$P(c_k)M_{\overline{c_kv}}$ is obtained from $M_{\overline{v}}$ by setting $z_{ij}=0$ if it lies in positions $(n\pm k+1,v^{-1}(n\pm k))$, and $\sigma(M_{\overline{c_kv}})=M_{\overline{c_kv}}$.
\end{proof}

Given $v\in C_n$, $v \geq v_\square$,
let $V_{\overline{v}}$ be the set of variables of $R_{\overline{v}}$, i.e.\ the set of variables that appear as entries of $M_{\overline{v}}$.

\begin{cor}\label{cor: deleted variable}
Fix a factorization $\overline{v}=u_lv_\square u_r$.
If $c_k$ is an ascent of~$v$ and we set $\overline{vc_k}=(u_l)v_\square (u_rc_k)$, then
$V_{\overline{vc_k}}\subseteq V_{\overline{v}}$ and $V_{\overline{v}}\setminus V_{\overline{vc_k}}=\{z_{ij}\}$, where $z_{ij}$ is the entry of $M_{\overline{v}}$ in positions $(v(n\pm k+1),n\pm k)$.
\end{cor}

\begin{proof}This follows from the inductive step in \Cref{prop: partial symm matrices}.
\end{proof}

\subsection{Equations for type~C Kazhdan-Lusztig varieties}\label{sect:KLeq}
Let $R_{\overline{v}}$ and $R_{v}^A$ denote respectively the coordinate rings $\Bbbk[\Sigma_v]$ and $\Bbbk[\Sigma_v^A]$.
Similarly, let $M_{\overline{v}}$ and $M_{v}^A$ denote respectively the generic matrices in $\Bbbk[\Sigma_v]$ and $\Bbbk[\Sigma_v^A]$.
Let $I_{\overline{v},w}$ be the ideal of $R_{\overline{v}}$ generated by the size $r_w(p,q)+1$ minors of $\tau_{p,q}(M_{\overline{v}})$ over all $(p,q)$ in $E(w)$.  
We call $I_{\overline{v},w}$ a {\bf Kazhdan--Lusztig ideal}.  

\begin{prop}\label{prop: KL coord ring}
We have $\Sigma_{v,w} = \operatorname{Spec}(R_{\overline{v}}/I_{\overline{v},w})$.
\end{prop}

Theorem~\ref{thm: main} will give an independent proof of this proposition in the case $v\ge v_\square$.

\begin{proof}
The discussion of rank conditions in Section~\ref{sect:KLrank} proves equality as sets.
Equality as schemes follows from \cite[Proposition 3.1]{WooYongSings},
which gives the analogous scheme-theoretic equality in type~A,
and Theorem~\ref{thm: LR} along with Equation~\eqref{eq: type C KL} (which follows from \cite[Proposition 6.1.1.2]{Lakshmibai-Raghavan}).
\end{proof}

We now define the term orders we use in this paper.  A {\bf diagonal} term order on $R_{\overline{v}}$
is one where, given any minor in $M_{\overline{v}}$ where the diagonal term is nonzero,
the diagonal term is the leading term.
In notation, this means that, if $\epsilon_1<\cdots<\epsilon_r$, $\delta_1<\cdots<\delta_r$,
$D$ is the minor of $M_{\overline{v}}$ using rows $\{\epsilon_1,\ldots,\epsilon_r\}$
and columns $\{\delta_1,\ldots,\delta_r\}$, and $\prod_{i=1}^r (M_{\overline{v}})_{\epsilon_i \delta_i}$ is nonzero, then this product is the leading term of $D$.
Note that there can be multiple distinct diagonal term orders.
However, even if the diagonal term is zero for a given minor, there are restrictions on which term can be its leading
term under a diagonal term order, since the condition applies to every subminor of the minor in question.
Because we are taking southwest minors rather than northwest minors in defining $I_{\overline{v},w}$, our diagonal term orders are equivalent to the antidiagonal term orders of~\cite{KnutsonMillerAnnals}.

We note that our diagonal term orders on different sets of variables are compatible with each other.

\begin{prop}\label{prop:diagonalcompatibility}
Let $c_k$ be an ascent of $v$, let $\prec$ be a diagonal term order on $R_{\overline{v}}$, and let
$\mathord{\prec'}=\mathord{\prec\mid_{R_{\overline{vc_k}}}}$ be the restriction of $\prec$ to $R_{\overline{vc_k}}$.
Then $\prec'$ is a diagonal term order on $R_{\overline{vc_k}}$.
\end{prop}

\begin{proof}
Suppose $D'$ is a minor of $M_{\overline{vc_k}}$ with nonzero diagonal term
using rows $\epsilon_1<\cdots<\epsilon_s$ and columns $\delta_1<\cdots<\delta_s$.  Let $D$ be the minor of
$M_{\overline{v}}$ using rows $\epsilon_1<\cdots<\epsilon_s$ and columns $c_k(\delta_1)<\cdots<c_k(\delta_s)$.
If $D'$ involves only one of the columns $n-k$ and $n-k+1$, and $D'$ involves only one of the columns $n+k$
and $n+k+1$, then the diagonal term of $D$ in $M_{\overline{v}}$ is the same as the diagonal term of $D'$
in $M_{\overline{vc_k}}$, as the rows and columns are ordered in the same way.
Note that the variable $z_{ij}$ in $V_{\overline{v}}\setminus V_{\overline{vc_k}}$ cannot appear in the diagonal term of $D$, as that would imply the diagonal term of $D$ is
zero.  Hence the leading term of $D$ under $\prec'$ must be the diagonal term,
as the leading term of $D'$ under $\prec$ is the diagonal term.

Otherwise, if $D'$ involves both columns $n-k$ and $n-k+1$ in $M_{\overline{v}}$, then, since $c_k$ is an
ascent of $v$, there is a right-to-left maximum in column $n-k+1$ of $M_{\overline{v}}$
and column $n-k$ of $M_{\overline{vc_k}}$. So by \Cref{rows and columns with ones},
$D'=\pm\widetilde{D'}$ where $\widetilde{D'}$ is the minor formed by removing column $n-k$
and row $vc_k(n-k)$ from $D'$.  Similarly $D=\pm\widetilde{D}$ where $\widetilde{D}$ is the minor formed
by removing column $n-k+1$ and row $v(n-k+1)$ from $D$.  Now the argument in the previous paragraph
applies to $\widetilde{D}$ and $\widetilde{D'}$.  A similar argument applies if $D'$ involves both columns $n+k$ and $n+k+1$.
\end{proof}

We show there is at least one diagonal term order, namely the lexicographic term order $\prec_{\mathrm{lex}}$ where
$z_{ij}>z_{i'j'}$ if and only if either $i>i'$, or $i=i'$ and $j>j'$.  One can see from the next section that $\prec_{\mathrm{lex}}$ is the term order used in \cite{Knutson-degenerate}, made explicit for this case.

\begin{prop}\label{prop:lexisdiagonal}
The term order $\prec_{\mathrm{lex}}$ is a diagonal term order.
\end{prop}

\begin{proof}
We prove this by downwards induction in length.  The base case is where $v=w_0$, where $R_{\overline{v}}$ has no variables and hence the statement is vacuously true.
Let $c_k$ be the last ascent of $v$.
By \Cref{cor: deleted variable} $V_{\overline{v}}\setminus V_{\overline{vc_k}}=\{z_{ij}\}$, where $z_{ij}$ is the entry of $M_{\overline{v}}$ in positions $(v(n\pm k+1),n\pm k)$.
Moreover, $z_{ij}$ must appear as the south-most nonzero entry in its column, and the 1 appearing immediately to its right is a right-to-left maximum.  Hence there are no variables southeast of (either appearance, if there are two, of) $z_{ij}$ in $M_{\overline{v}}$.
It follows that $z_{ij}$ is the largest variable in $R_{\overline{v}}$ under $\prec_{\mathrm{lex}}$.

 By induction, $\prec_{\mathrm{lex}}$ restricted to $R_{\overline{vc_k}}$ is a diagonal term order.
By \Cref{prop:diagonalcompatibility}, it suffices to show that 
if $z_{ij}$ appears in a minor of $M_{\overline{v}}$ with a nonzero diagonal term, then it must appear in the diagonal term.
Since there are no variables southeast of $z_{ij}$ in $M_{\overline{v}}$, any minor of $M_{\overline{v}}$ with nonzero diagonal term and such that $z_{ij}$ does not appear on the diagonal term must have as its southeast entry the $1$ directly to the right of $z_{ij}$. 
By \Cref{rows and columns with ones} the only nonzero entry of rightmost column of the minor is this $1$. 
It follows that $z_{ij}$ does not appear in any term of the minor.
\end{proof}

We now state our main theorem.

\begin{thm}\label{thm: main}
Given $v\ge v_\square$, the size $r_w(p,q)+1$ minors of $\tau_{p,q}(M_{\overline{v}})$ over all $(p,q)$ in $E(w)$
form a Gr\"obner basis for $I_{\overline{v},w}$ with respect to any diagonal term order.
\end{thm}

The proof appears in Section~\ref{sec: proof}.
The main technique is to show that $K$-polynomials of subword complexes, suitably weighted, satisfy the Kostant-Kumar recursion.
(This technique follows \cite{Knutson-degenerate}.)

\subsection{Torus action of type~C Kazhdan-Lusztig varieties and the weights for $v\ge v_\square$}\label{sect:TActionSmall}\label{sect:Taction1}

Let $T$ be the torus consisting of the diagonal matrices in  $Sp_{2n}(\mathbb{K})$. 
Since any $(t_{ij})\in T$ is fixed under $\sigma$, it satisfies $t_{ii}=t_{2n+1-i,2n+1-i}^{-1}$ for all $i$.
The torus $T$ acts on $\Omega_v^\circ$ by left multiplication, i.e.\ given matrices $M\in Sp_{2n}(\mathbb{K})$ and $N\in T$,
	$$
	N\bullet (MB^+_G/B^+_G):=(N M)B^+_G/B^+_G.
	$$ 
This action induces the following torus action on $\Sigma_v$: for $M\in\Sigma_v$ and $N\in T$, 
$N\bullet_v M$ is the matrix in $\Sigma_v$ representing $(N M)B^+_G/B^+_G$.
Let us describe the action more concretely.
Notice that in general $NM\notin\Sigma_v$ because the entry of $NM$ in position $(\pi(j),j)$ need not equal $1$.
Thus to obtain an element of $\Sigma_v$ we need to multiply on the right by the appropriate element of $T$ to make these entries $1$.

It will be most convenient for us to denote by $(x_1,\ldots,x_n)$ the element of $T$ where
	\begin{equation}\label{eq:torus}
	(x_1,\ldots,x_n)
	:=\text{diag}(x_n,\ldots,x_1,x_1^{-1}\dots, x_{n}^{-1}),
	\end{equation}
the diagonal matrix with diagonal entries $x_n,\ldots,x_1,x_1^{-1}\dots, x_{n}^{-1}$ from northeast to southwest.

\begin{ex}\label{eg:squarewordweights} We describe the action of $T$ on $\Sigma_{v_\square}$ for $v_\square = 321654$:
 	\begin{align*}
	(x_1,x_2,x_3)
 \bullet_{v}
\begin{bmatrix}
	0&0&1 &0&0&0 \\
	0&1&0 &0&0&0\\
	1&0&0 &0&0&0\\
	 z_{11} & z_{12} &z_{13} & 0&0&1 \\
	 z_{12} & z_{22} &z_{23} & 0&1&0 \\
	  z_{13} & z_{23} &z_{33} &  1&0&0
	\end{bmatrix}
	&=
	(x_1,x_2,x_3)
\begin{bmatrix}
	0&0&1 &0&0&0 \\
	0&1&0 &0&0&0\\
	1&0&0 &0&0&0\\
	 z_{11} & z_{12} &z_{13} & 0&0&1 \\
	 z_{12} & z_{22} &z_{23} & 0&1&0 \\
	  z_{13} & z_{23} &z_{33} &  1&0&0
	\end{bmatrix}
	(x_3^{-1},x_2^{-1},x_1^{-1})\\
	&=
	\begin{bmatrix}
	0&0&1 &0&0&0 \\
	0&1&0 &0&0&0\\
	1&0&0 &0&0&0\\
	x_1^{-2} z_{11} & x_1^{-1}x_2^{-1}z_{12} &x_1^{-1}x_3^{-1}z_{13} & 0&0&1 \\
	x_1^{-1}x_2^{-1} z_{12} & x_2^{-2}z_{22} &x_2^{-1}x_3^{-1}z_{23} & 0&1&0 \\
	 x_1^{-1}x_3^{-1}z_{13} & x_2^{-1}x_3^{-1}z_{23} &x_3^{-2}z_{33} &  1&0&0
	\end{bmatrix}.	
	\end{align*}
\end{ex}

Since $T$ acts by left multiplication on each Schubert variety and on each opposite Schubert cell,  the torus $T$ also acts by left multiplication on each Kazhdan-Lusztig variety $\mathcal{N}_{v,w}$.

Let us now restrict to $v\geq v_{\square}$ and explicitly compute the weights on the coordinates $z_{ij}$ of the action. 
We adopt the convention that the weight $e_i$ denotes the homomorphism in $\text{Hom}(T,\mathbb{K})$ that sends the element $(x_1,\ldots,x_n)$ to $x_i$.  We will write weights additively.  In addition, we will let $t_i=\exp(e_i)$ denote the formal exponential of the weight $e_i$, so that the group operation on the $t_i$ (and monomials therein) is multiplication.

\begin{lemma}\label{lem:squarewordweights}
The coordinate function $z_{ij}$ on $M_{\overline{v_\square}}$ has weight $e_i+e_j$.
\end{lemma}

\begin{proof}
As one can see from Example~\ref{eg:squarewordweights}, acting on $M\in \Sigma_{{v_\square}}$ by $(x_1,\ldots,x_n)$ multiplies the entry $z_{ij}(M)$ by $x_i^{-1}x_j^{-1}$.  Hence, the weight of the action on $z_{ij}$, the coordinate function on this entry, is $e_i+e_j$.
\end{proof}

We next see that the analog of Lemma \ref{lem:squarewordweights} holds for any $v\geq v_{\square}$.
To do so we need to consider the action of $C_n$ on the weights induced from permuting the diagonal entries of $(x_1,\ldots,x_n)$.
Notice that given $u\in C_n$, the induced action is so that for $i\le n$,
	$$
	u\cdot e_i=\begin{cases}
	-e_{n+1-u(n+i)}&  \text{if }\ u(n+i)\le n,\\
	e_{u(n+i)-n}&  \text{if }\ u(n+i)\ge n+1.
\end{cases}
$$

\begin{prop}\label{prop: weights}
Let $v\ge v_\square$ with $\overline{v}=u_lv_\square u_r$.
The weight of a coordinate function of $M_{\overline{v}}$ depends only on its position (and not on $\overline{v}$).
Furthermore, the coordinate function $z_{ij}$ on $M_{\overline{v}}$ has weight $u_l\cdot e_i+u_l\cdot e_j$.
\end{prop}

Note that for $\overline{v}=u_lv_\square u_r$ and $i\le n$, in the notation of Lemma~\ref{lem: 123 shuffle} we have $u_l(i)=v(a_{n+1-i})$ and $u_l(n+i)=v(b_{n+1-i})$.

\begin{proof}
Define $y_{n+i}:=x_i^{-1}$ and $y_{n+1-i} := x_i$, $1\leq i\leq n$, so that $ \text{diag}(y_1,\dots, y_{2n})=(x_1,\ldots,x_n)$.
 Then, the action of $T$ on $\Sigma_v$, $v\in C_n$, is given by
\[
\mathbf{y}\bullet_{v} M = \text{diag}(y_1,\dots, y_{2n})~M~\text{diag}(y_{v(1)}^{-1},\dots, y_{v(2n)}^{-1}), \quad \mathbf{y}\in T, ~M\in \Sigma_v.
\]
Therefore, the weight for the coordinate function in position $(\epsilon,\delta)$ is the weight corresponding to $y_\epsilon y^{-1}_{v(\delta)}$, which depends only on $(\epsilon,\delta)$ and not on $\overline{v}$.

Suppose that $v\ge v_{\square}$ and $\overline{v}=u_lv_\square u_r$.
Let $M\in \Sigma_v$ and $\mathbf{y}\in T$. 
Given $i\le j$, the variable $z_{ij}$ appears as entries $(v(b_{n+1-j}),a_i)$ and $(v(b_{n+1-i}),a_j)$ of $M$.
The entries in positions $(v(b_{n+1-j}),a_i)$ and $(v(b_{n+1-i}),a_j)$ of $\mathbf{y}\bullet_{v} M$ are 
	$$
	y_{v(b_{n+1-j})}y^{-1}_{v(a_i)}z_{ij}=y_{u_l(n+j)}y^{-1}_{u_l(n+1-i)}z_{ij}\ \ \text{ and }\ \ y_{v(b_{n+1-i})}y^{-1}_{v(a_j)}z_{ij}=y_{u_l(n+i)}y^{-1}_{u_l(n+1-j)}z_{ij},
$$
respectively.
By definition of $\mathbf{y}$ and the fact $u_l\in C_n$,
	$$
	y_{u_l(n+j)}y^{-1}_{u_l(n+1-i)}=y_{u_l(n+j)}y_{u_l(n+i)}=y_{u_l(n+i)}y^{-1}_{u_l(n+1-j)}.
$$
Again by definition of $\mathbf{y}$, we conclude that the weight of $z_{ij}$ is $u_l\cdot e_i+u_l\cdot e_j$.
\end{proof}

\begin{ex} Let $v=642531$ and $\overline{v}$ be the factorization associated to $(a_{\bullet})=(1,2,3)$, and $(b_{\bullet})=(4,5,6)$.
We describe the weights of the coordinate functions of $M_{\overline{v}}$ via an explicit computation:
 	\begin{align*}
	(x_1,x_2,x_3)
	 \cdot
	\begin{bmatrix}
	0&0&0 &0&0&1 \\
	0&0&1 &0&0&0\\
	0&0&z_{23} &0&1&0\\
	 0 & 1 &0 & 0&0&0 \\
	 0 & z_{23} &z_{33} & 1&0&0 \\
	  1 & 0 &0 &  0&0&0
	\end{bmatrix}
	&=
	(x_1,x_2,x_3)
	\begin{bmatrix}
	0&0&0 &0&0&1 \\
	0&0&1 &0&0&0\\
	0&0&z_{23} &0&1&0\\
	 0 & 1 &0 & 0&0&0 \\
	 0 & z_{23} &z_{33} & 1&0&0 \\
	  1 & 0 &0 &  0&0&0
	\end{bmatrix}
	(x_2^{-1},x_1,x_3)\\
	&=
	\begin{bmatrix}
	0&0&0 &0&0&x_3 \\
	0&0&x_2 &0&0&0\\
	0&0&x_1z_{23} &0&x_1&0\\
	 0 & x^{-1}_1 &0 & 0&0&0 \\
	 0 & x^{-1}_2z_{23} &x^{-1}_2z_{33} & x^{-1}_2&0&0 \\
	  x^{-1}_3 & 0 &0 &  0&0&0
	\end{bmatrix}
	(x_2^{-1},x_1,x_3)\\
	&=
	\begin{bmatrix}
	0&0&0 &0&0&1 \\
	0&0&1 &0&0&0\\
	0&0&x_2^{-1}x_1z_{23} &0&1&0\\
	 0 & 1&0 & 0&0&0 \\
	 0 & x_1x^{-1}_2z_{23} & x_2^{-1}x^{-1}_2z_{33} & 1&0&0 \\
	 1 & 0 &0 &  0&0&0
	\end{bmatrix}.
	\end{align*}
Thus the weight of $z_{23}$ is $-e_1+e_2$ and the weight of $z_{33}$ is $-e_2-e_2$. 
We verify that for $z_{23}$ this agrees with \Cref{prop: weights}. 
Since $v(b_{n+1-2})\le n$ and $v(b_{n+1-3})\ge n+1$, we have $u_l\cdot e_2=-e_1$ and $u_l\cdot e_3=e_2$.
One can verify that if we now take  $\overline{v}$ to be the factorization associated to $(a_{\bullet})=(2,3,6)$ and $(b_{\bullet})=(1,4,5)$ (as in the second part of \Cref{ex: variable labels change}) the weight of $z_{12}$ is $-e_1+e_2$ and the weight of $z_{22}$ is $-e_2-e_2$.
\end{ex}

We end by noting that this multigrading is positive, so that the only elements in $R_{\overline{v}}$ which have degree ${\mathbf 0}$ are the constants.

\begin{cor}
Let $v\geq v_\square$ with $\overline{v} = u_lv_\square u_r$. 
The multigrading on $R_{\overline{v}}$ that assigns degree $u_l\cdot e_i+u_l\cdot e_j$ to coordinate function $z_{ij}$ is a positive multigrading.
\end{cor}

\begin{proof}
The set of all vectors $e_i+e_j$ generates a pointed cone, so the images of these vectors under the action of a fixed $u_l$ do also.
\end{proof}

\section{Type \texorpdfstring{$C$}{C} subword complexes and vertex decomposition}\label{sec: type C swc}

\subsection{Subword complexes}\label{sec:subword complexes}
In \cite{KnutsonMillerAdvances,KnutsonMillerAnnals} A.~Knutson and E.~Miller defined a family of simplicial complexes, called subword complexes, for arbitrary Coxeter groups.
Let $Q=(\alpha_1,\ldots,\alpha_\ell)$ be a reduced word for $v\in C_n$, as defined in Section~\ref{sec: weak order}.
The \textbf{subword complex} $S(Q,w)$ associated to $Q$ and~$w\in C_n$
is the simplicial complex on the vertex set $[\ell]=\{1,\ldots,\ell\}$
whose maximal faces are the sets $[\ell]\setminus\{i_1,\ldots,i_k\}$
such that the subword $(\alpha_{i_1},\ldots,\alpha_{i_k})$ of~$Q$ is a reduced word for~$w$.
If $v\not\Bruhatge w$,
then $S(Q,w)=\emptyset$ is the simplicial complex with no faces.
This must be distinguished from the complex $\{\emptyset\}$, which is $S(Q,v)$ whenever $Q$ is a reduced word for $v$.

The key fact we use about subword complexes is their vertex decomposition,
first proved as \cite[Theorem E]{KnutsonMillerAnnals} (for every Coxeter group).

\begin{thm}
\label{thm: vertex decomposition}
Let $v,w\in C_n$, and let $Q=(\alpha_1,\ldots,\alpha_\ell)$ be a reduced word for~$v$.
Assume that $v\Bruhatge w$.
If $v=1$, then $w=1$, $Q=()$, and $S(Q,w)=\{\emptyset\}$.
Otherwise $\ell>0$.  Let $Q'=(\alpha_1,\ldots,\alpha_{\ell-1})$ and $i=\alpha_\ell$. Then
\[S(Q,w) = \cone_\ell S(Q',w) \cup S(Q',wc_i). \]
\end{thm}

\subsection{Labeling the vertices with variables}\label{sec: labeling}
Recall that $V_{\overline{v}}$ denotes the set of variables of $R_{\overline{v}}$, i.e.\ the set of variables that appear as entries of $M_{\overline{v}}$.

\begin{lemma}\label{lemma: remove last variable}
Let $c_k$ be an ascent of~$v$, $\overline{v}=u_lv_\square u_r$, $\overline{vc_k}=(u_l)v_\square (u_rc_k)$, and $V_{\overline{v}}\setminus V_{\overline{vc_k}}=\{z_{ij}\}$. Then $k=j-i$.
\end{lemma}

\begin{proof}
By \Cref{cor: deleted variable}, $z_{ij}$ is in positions $(v(n\pm k+1),n\pm k)$ of $M_{\overline{v}}$. We also have that $z_{ij}$ is in position $(v(b_{n+1-i}),a_j)$ and therefore $a_j=n+k$, $b_{n+1-i}=n+k+1$.
Since the number of columns to the right of column $n+k$ is counted by both $|\{a_{i+1},\ldots,a_n\}|+|\{b_{n+1-j},\ldots,b_n\}|$ and $2n-(n+k)=n-k$, we have
	$$
	n-k=|\{a_{i+1},\ldots,a_n\}|+|\{b_{n+1-j},\ldots,b_n\}|=n-i+j.
	$$
We conclude that $k=j-i$.
\end{proof}

\begin{prop}\label{prop: Q is a reduced word}
Let $Q$ be the word $(j_1-i_1, \ldots, j_\ell-i_\ell)$, 
where $z_{i_1j_1}\prec_{\mathrm{lex}}\cdots \prec_{\mathrm{lex}} z_{i_\ell j_\ell}$ are the variables in~$V_{\overline{v}}$. 
The word $Q$ is a reduced word for $w_0v$.
\end{prop}

\begin{proof}
Let $\overline{v}=u_lv_\square u_r$.
We proceed by induction on $\ell(w_0v)$. 
The base case $v=w_0$ is trivial.
For the inductive case, let $Q'$ equal $Q$ without the last letter, which we denote by $\alpha_\ell$.
By construction $\alpha_\ell=j-i$ where $z_{ij}$ is the last variable in $V_{\overline{v}}$ under $\prec_{\mathrm{lex}}$.
Let the lowest box of~$D(v)$ containing $z_{ij}$ be in the $(n+k)$th column.
Since $z_{ij}$ is the last variable in $V_{\overline{v}}$, there are no boxes of $D(v)$ weakly southeast of this box.
This implies that $c_k$ is the last ascent of~$v$ and, by \Cref{cor: deleted variable}, that $V_{\overline{v}}\setminus V_{\overline{vc_k}}=\{z_{ij}\}$, where $k=j-i$ by Lemma~\ref{lemma: remove last variable}.
(As in the proof of Proposition~\ref{prop: partial symm matrices}, $\overline{vc_k}$ denotes the factorization $(u_\ell)v_\square (u_rc_k)$, where $u_\ell v_\square u_r$ is the factorization denoted by $\overline{v}$.)
Therefore $Q'$ is the word constructed from the variables in~$V_{\overline{vc_k}}$.
By the induction hypothesis, $Q'$ is a reduced word for $w_0vc_k$.
Because $c_{\alpha_\ell}$ is an ascent of $vc_{\alpha_\ell}$, that is $\ell((w_0vc_{\alpha_\ell})s_{j-i})=\ell(w_0vc_{\alpha_\ell})+1$,
we can append $\alpha_{\ell}={j-i}$ to a reduced word for $w_0vc_{\alpha_\ell}$ to obtain a reduced word for $w_0vc_{\alpha_\ell}s_{j-i}=w_0v$.
It follows that $Q$ is a reduced word for $w_0v$.
\end{proof}

Define $\zeta:[\ell]\to V_{\overline{v}}$ to be the map that associates to $k$ the $k$-th smallest variable in $V_{\overline{v}}$ under $\prec_{\mathrm{lex}}$.
Let
\[\Delta_{\overline{v},w} = \{\zeta(F) : F\in S(Q, w_0w)\},\]
where $Q$ is the reduced word for $w_0v$ defined in \Cref{prop: Q is a reduced word}.
This is a simplicial complex isomorphic to $S(Q, w_0w)$
but relabeled so its vertex set is~$V_{\overline{v}}$.

Proposition~\ref{thm: vertex decomposition2} translates Theorem~\ref{thm: vertex decomposition} to the notation $\Delta_{\overline{v},w}$,
and breaks it into the cases that will appear in our proofs (which are also parallel to the cases in the statement of Theorem~\ref{thm:KostKum}).

\begin{prop}\label{thm: vertex decomposition2}
Let $v,w\in C_n$ and $\overline{v}=u_lv_\square u_r$.
\begin{itemize}
\item If $v\not\Bruhatle w$, then $\Delta_{\overline{v},w} = \emptyset$.
\item If $v = w_0$, then $w = w_0$ (or we are in the previous case), and $\Delta_{\overline{v},w} = \{\emptyset\}$. 
\item Otherwise, let $k$ be the last ascent of $v$, so $vc_k\Bruhatg v$, and let $\overline{vc_k}=u_lv_\square (u_rc_k)$.
\begin{enumerate}
\item If $k$ is a descent of $w$, so $wc_k\Bruhatl w$, then
\[
\Delta_{\overline{v},w} = \cone_{z_{ij}}\Delta_{\overline{vc_k},w},
\]
where $z_{ij}$ is the largest variable with respect to~$\prec_{\rm lex}$ on $R_{\overline{v}}$.
\item If $k$ is an ascent of $w$, so $wc_k\Bruhatg w$, then
\[
\Delta_{\overline{v},w} = \cone_{z_{ij}}(\Delta_{\overline{vc_k},w})\cup \Delta_{\overline{vc_k}, wc_k}.
\]
\end{enumerate}
\end{itemize}
\end{prop}

\begin{ex}\label{ex: vtx dec 1}
For $v=v_\square=321654$ and $w=635241$ we have that $Q=(0,1,2,0,1,0)$ and $w_0w=c_0c_1$.
The last ascent of $v$ is $c_0$ and this is a descent of $w$. In this case, $\Delta_{\overline{v},w} = \cone_{z_{33}}\Delta_{\overline{vc_0},w}$, and $\Delta_{\overline{vc_0},w}$ is pictured in \Cref{fig: vtx dec 1}.
\begin{figure}[h]
\begin{tikzpicture}
\filldraw[fill=cyan!30] (.5,-1) node {$\bullet$} node[left] {$z_{12}$}--(0,0) node {$\bullet$} node[left] {$z_{11}$}--(1,0) node {$\bullet$}node[above] {$z_{13}$}--(2,0) node {$\bullet$}node[right] {$z_{23}$}--(1.5,-1) node {$\bullet$}node[right] {$z_{22}$}--(.5,-1)--(1,0)--(1.5,-1);
\end{tikzpicture}
\caption{The simplicial complex $\Delta_{\overline{vc_0},w}$ for $v=321654$ and $w=635241$.}
\label{fig: vtx dec 1}
\end{figure}
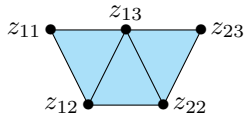
\end{ex}

\begin{ex}\label{ex: vtx dec 2}
For $v=v_\square=321654$ and $w=632541$ we have that $Q=(0,1,2,0,1,0)$ and $w_0w=c_0c_1c_0$.
The last ascent of $v$ is $c_0$ and this is an ascent of $w$. In this case, $\Delta_{\overline{v},w} = \cone_{z_{33}}(\Delta_{\overline{vc_0},w})\cup\Delta_{\overline{vc_k}, wc_k}$ as one can see in \Cref{fig: vtx dec 2}.
\begin{figure}[h]
\begin{tikzpicture}
\filldraw[fill=cyan!30] (0,-1) node {$\bullet$} node[below] {$z_{13}$}--(-1.2,0) node {$\bullet$} node[left] {$z_{11}$}--(-.6,0) node {$\bullet$}node[above] {$z_{12}$}--(0,0) node {$\bullet$}node[above] {$z_{22}$}--(.6,0) node {$\bullet$}node[above] {$z_{23}$}--(1.2,0) node {$\bullet$}node[above] {$z_{33}$}--(0,-1)--(.6,0);
\draw (0,0)--(0,-1)--(-.6,0);
\draw (2,-.5) node {$=$};
\end{tikzpicture}
\begin{tikzpicture}
\filldraw[fill=cyan!30] (0,-1) node {$\bullet$} node[below] {$z_{13}$}--(.6,0) node {$\bullet$}node[above] {$z_{23}$}--(1.2,0) node {$\bullet$}node[above] {$z_{33}$}--(0,-1);
\draw (1.5,-.5) node {$\cup$};
\end{tikzpicture}
\begin{tikzpicture}
\filldraw[fill=cyan!30] (0,-1) node {$\bullet$} node[below] {$z_{13}$}--(-1.2,0) node {$\bullet$} node[left] {$z_{11}$}--(-.6,0) node {$\bullet$}node[above] {$z_{12}$}--(0,0) node {$\bullet$}node[above] {$z_{22}$}--(.6,0) node {$\bullet$}node[above] {$z_{23}$}
--(0,-1);
\draw (0,0)--(0,-1)--(-.6,0);
\end{tikzpicture}
\caption{The simplicial complex $\Delta_{\overline{v},w} = \cone_{z_{33}}\Delta_{\overline{vc_0},w} \cup \Delta_{\overline{vc_0}, wc_0}$ for $v=321654$ and $w=632541$.}
\label{fig: vtx dec 2}
\end{figure}
\end{ex}

Let $K_{\overline{v},w}$ be the Stanley-Reisner ideal of $\Delta_{\overline{v},w}$.
This is the ideal generated by the monomials that are the non-faces of $\Delta_{\overline{v},w}$, so
$$K_{\overline{v},w} = \langle \prod_{z\in Z} z : Z\subseteq V_{\overline{v}}, Z \not\in \Delta_{\overline{v},w} \rangle.$$
Translating Proposition~\ref{thm: vertex decomposition2} to be in terms of $K_{\overline{v},w}$ gives the following.

\begin{prop}\label{thm: vertex decomposition3}
Let $v,w\in C_n$ and $\overline{v}=u_lv_\square u_r$.
\begin{itemize}
\item If $v\not\Bruhatle w$, then $K_{\overline{v},w} = \langle 1\rangle$.
\item If $v = w_0$, then $w = w_0$ (or we are in the previous case), and $K_{\overline{v},w} = \langle 0 \rangle$. 
\item Otherwise, let $k$ be the last ascent of $v$, so $vc_k\Bruhatg v$, and let $\overline{vc_k}=u_lv_\square (u_rc_k)$.
Let $z_{ij}$ be the largest variable with respect to $\prec_{\mathrm{lex}}$ on $R_{\overline{v}}$.
\begin{enumerate}
\item If $k$ is a descent of $w$, so $wc_k\Bruhatl w$, then
\[
K_{\overline{v},w} = K_{\overline{vc_k},w} R_{\overline{v}}.
\]
\item If $k$ is an ascent of $w$, so $wc_k\Bruhatg w$, then
\[
K_{\overline{v},w} = \langle z_{ij}m : m\in K_{\overline{vc_k},w} \rangle + K_{\overline{vc_k}, wc_k}R_{\overline{v}}.
\]
\end{enumerate}
\end{itemize}
\end{prop}

\begin{ex}
In this example we verify (1) and (2) in the proposition above.
First, let $v$ and $w$ be as in \Cref{ex: vtx dec 1}.
We compute that both ideals are generated by $z_{11}z_{22}, z_{11}z_{23}, z_{12}z_{23}$,
although for $K_{\overline{v},w}$ these generators are interpreted in the ring $\mathbb{K}[z_{11},z_{12},z_{13},z_{22},z_{23},z_{33}]$,
and for $K_{\overline{vc_k},w}$ in the ring $\mathbb{K}[z_{11},z_{12},z_{13},z_{22},z_{23}]$.

Now, let $v$ and $w$ be as in \Cref{ex: vtx dec 2}.
Direct computation shows that \[K_{\overline{v},w}=\langle z_{11}z_{33},z_{12}z_{33},z_{22}z_{33},  z_{11}z_{22}, z_{11}z_{23}, z_{12}z_{23} \rangle,\] and $K_{\overline{vc_0},w} = \langle z_{11}, z_{12}, z_{22}\rangle$, and $K_{\overline{vc_0},wc_0} = \langle z_{11}z_{22},z_{11}z_{23}, z_{12}z_{23}\rangle$. Then $K_{\overline{v},w} = z_{33}K_{\overline{vc_0},w}+K_{\overline{vc_0},wc_0}$.
\end{ex}

\section{Proof of Theorem~\ref{thm: main}}\label{sec: proof}

In this section we prove Theorem~\ref{thm: main}.
We explain the overall structure of the proof now, and dedicate subsections to the details.

Given $v,w\in C_n$ with $v\ge v_\square$ in left-right weak order, a factorization $\overline{v}$ for $v$,
and a diagonal term order $\prec$ on $R_{\overline{v}}$, 
let $\mathcal G_{\overline{v},w,\prec}$ be the set of initial monomials of the generators we used to define $I_{\overline{v},w}$.
We recall that these generators are the size $r_w(p,q)+1$ minors of the truncated matrix $\tau_{p,q}(M_{\overline{v}})$,
running over all $(p,q)$ in~$E(w)$.
Let $J_{\overline{v},w,\prec}$ be the ideal generated by~$\mathcal G_{\overline{v},w,\prec}$.
We will show that
\[K_{\overline{v},w}\subseteq J_{\overline{v},w,\prec}\subseteq\init_\prec I_{\overline{v},w},\]
the former containment being Proposition~\ref{prop: K in J} below and the latter clear from the definition of~$J_{\overline{v},w}$. In Proposition~\ref{prop: K(K) = K(I)}, we will prove that the $K$-polynomials
$\mathcal{K}(R_{\overline{v}}/K_{\overline{v},w};\mathbf{t})$ and 
$\mathcal{K}(R_{\overline{v}}/I_{\overline{v},w};\mathbf{t})=\mathcal{K}(R_{\overline{v}}/\init_\prec I_{\overline{v},w};\mathbf{t}) $ are equal.
The containments above then imply
\[K_{\overline{v},w}=J_{\overline{v},w,\prec}=\init_\prec I_{\overline{v},w},\]
and the latter equality is the statement of Theorem~\ref{thm: main}.
Note that this shows $J_{\overline{v},w,\prec}$ is independent of the choice of diagonal term order $\prec$.

\subsection{The Stanley-Reisner ideal is contained in the initial ideal}
This subsection proves the containment $K_{\overline{v},w}\subseteq J_{\overline{v},w,\prec}$, which is Proposition~\ref{prop: K in J}.
The proof will be by induction on the length $\ell(w_0v)$.
A factorization $v=u_lv_\square u_r$, where $\ell(v) = \ell(u_l)+\ell(v_\square)+\ell(u_r)$,
can be extended to a factorization $w_0 = u_lv_\square u_r(w_0v)^{-1}$,
and if the induction were unrolled, it would descend to~$v$ from its base case $v=w_0$ in right weak order
by acting by simple reflections at the right of this factorization.
Thus, we can use right weak order to induct down from $w_0$ to any $123$-avoiding permutation where every permutation along the way is $123$-avoiding.

Throughout this section, we let $\overline{v}$ be the factorization $v=u_lv_\square u_r$, and, for $c_k$ an ascent of $v$, we let $\overline{vc_k}$ be the factorization $v=u_lv_\square (u_rc_k)$.

Every term of the Leibniz formula for a minor of~$M_{\overline{v}}$ is zero or a signed monomial in the variables $z_{ij}$.
Our proofs in this section will rely on the fact that there are no cancellations among these terms.
This is essentially the fact known in spectral graph theory as the Harary--Sachs theorem \cite{Harary,Sachs}.

\begin{lemma}\label{lem:powerof2}
Every coefficient of any minor of~$M_{\overline{v}}$ is a signed power of~$2$.
If $\prod_{j=1}^r (M_{\overline{v}})_{p_jq_j}$ is nonzero,
then it is a monomial contained in the support of the $(\{p_1,\ldots,p_r\},\{q_1,\ldots,q_r\})$ minor of~$M_{\overline{v}}$.
\end{lemma}

\begin{proof}
If a square submatrix $N$ of~$M_{\overline{v}}$ contains an entry~$1$,
then by Proposition~\ref{rows and columns with ones},
expansion along either its row or column shows that 
every nonzero term in $\det N$ involves that entry~$1$ 
and $N$ has the same determinant as a smaller submatrix, up to sign.
So we may assume that $N$ contains no~$1$s.

Give the rows of $N$ the names $\epsilon_1, \ldots, \epsilon_r$ and its columns the names $\delta_1, \ldots, \delta_r$,
in such a way that whenever $N$ contains a row and column which contribute the same weight to the action of $T$ on~$\Sigma_v$ (see Proposition~\ref{prop: weights}),
then this row and column are $\epsilon_i$ and $\delta_i$ for some $i$.
This ensures that, if a variable is repeated in $N$,
the two positions at which it appears are $(\epsilon_i, \delta_j)$
and $(\epsilon_j, \delta_i)$ for some $i$ and $j$.

We have
\[\pm\det(N) = \sum_{u\in S_r}\operatorname{sgn}(u)\prod_j N_{\epsilon_j\delta_{u(j)}}.\]
Suppose $u,u'\in S_r$ index terms of this sum which are both equal, up to sign, to a monomial $m$.
Because each nonzero entry of $N$ is a variable, the multisets of variables entering the product for $u$ and that for~$u'$ must be equal.
Thus for every $i, j\in [r]$, if $u(i) = j$, then $u'(i) = j$ or $u'(j) = i$
(and if $u(i)=j$ and $u(j)=i$, then $u'(i)=j$ and $u'(j)=i$).
It follows that $u'$ is obtained from $u$ by inverting some of the cycles in its disjoint cycle decomposition.
A cycle and its inverse have the same sign, so $\operatorname{sgn}(u)$ is constant over all terms with $\prod_j N_{\epsilon_j\delta_{u(j)}}=m$.
The coefficient of $m$ in $\det N$ is $\operatorname{sgn}(u)$ times the number of such terms.
Because any set of disjoint cycles can be inverted independently, this number
is $2^k$, where $k$ is the number of cycles $C$ with $\prod_{j\in C} N_{\epsilon_j\delta_{c(j)}} = \prod_{j\in C} N_{\epsilon_{c(j)}\delta_{j}}$.
This is nonzero in $\mathbb{K}$ because $\operatorname{char}\mathbb{K}\ne2$.
\end{proof}

We now prove the main technical lemma of this section.

\begin{lemma}\label{lem: before lunch}
Fix a diagonal term order $\prec$ on $R_{\overline{v}}$.  Let $c_k$ be the last ascent of $v$ and
fix $\prec'$, the restriction of $\prec$, as our diagonal term order on $R_{\overline{vc_k}}$.
Suppose $m$ is the leading term (with respect to $\prec'$) of the minor of $M_{\overline{vc_k}}$ on row set $A=\{\epsilon_1,\ldots,\epsilon_r\}$ and column set $B=\{\delta_1,\ldots,\delta_r\}$,
labeled so that the entries giving the leading term are the $(\epsilon_j,\delta_j)$ entries, i.e.\ $m = \prod_{j=1}^r (M_{\overline{vc_k}})_{\epsilon_j\delta_j}$.
Let $p,q$ be such that $A\subset[2n]\setminus[p-1]$ and $B\subset [q]$, and assume row $\epsilon_j$ of $\tau_{p,q}(M_{\overline{vc_k}})$ contains an entry 1
exactly when $j=s+1,\ldots,r$.
Define
\[B' = \{c_k(\delta_1), \ldots, c_k(\delta_s), v^{-1}(\epsilon_{s+1}), \ldots, v^{-1}(\epsilon_r)\}.\]
\begin{enumerate}
\item The leading term $m'$ (with respect to $\prec$) of the minor of $M_{\overline{v}}$ on row set $A$ and column set $B'$ divides $m$.
\item 
If $c_k$ is a descent of $w$, $m\in\mathcal G_{\overline{vc_k},w,\prec'}$, $(p,q)\in E(w)$, and $r=r_w(p,q)+1$ then $m'\in\mathcal G_{\overline{v},w,\prec}$.
\item
If $c_k$ is an ascent of $w$, $m\in\mathcal G_{\overline{vc_k},w,\prec'}$, $(p,q)\in E(w)$, $r=r_w(p,q)+1$, and $V_{\overline{v}}\setminus V_{\overline{vc_k}} = \{z_{ij}\}$,
then there exists $m''\in\mathcal G_{\overline{v},w,\prec}$ with $m '' \mid z_{ij}m'$.
\end{enumerate}
\end{lemma}

\begin{proof}

Let $1\leq \delta\leq 2n$ be a column index.
By \Cref{rows and columns with ones}, 
in any minor of~$M_{\overline{v}}$ using row $v(\delta)$ and column $\delta$, 
all nonvanishing terms use the $(v(\delta),\delta)$ entry.
Applied to the $(A,B')$ minor of $M_{\overline{v}}$, we conclude that this minor equals, up to sign, the 
\[(\{\epsilon_1,\ldots,\epsilon_s\},\{c_k(\delta_1), \ldots, c_k(\delta_s)\})\] 
minor of~$M_{\overline{v}}$.

Let us first show that the $(\{\epsilon_1,\ldots,\epsilon_s\},\{c_k(\delta_1), \ldots, c_k(\delta_s)\})$ minor of $M_{\overline{v}}$ equals, up to sign,
the $(\{\epsilon_1,\ldots,\epsilon_s\},\{\delta_1, \ldots, \delta_s\})$ minor of~$M_{\overline{vc_k}}$.
By \Cref{cor: deleted variable},
$M_{\overline{vc_k}}$ is obtained from $M_{\overline{v}}P(c_k)$ by setting 
a single variable $z_{ij}$ to~0, which appears in columns $n\pm k$ of~$M_{\overline{v}}$.
\Cref{prop: partial symm matrices} shows that none of the other variables change names.
Therefore, the claim is straightforward except in the case where $z_{ij}$ shows up in the 
$(\{\epsilon_1,\ldots,\epsilon_s\},\{c_k(\delta_1), \ldots, c_k(\delta_s)\})$ minor of $M_{\overline{v}}$.
This occurs if $n\pm k+1\in \{\delta_1,\ldots,\delta_s\}$.  Without loss of generality, let $\delta_1= n\pm k+1$;
we must then have $vc_k(\delta_1)=v(\delta_1-1)\in \{\epsilon_1,\ldots,\epsilon_s\}$.  Let $\epsilon_t=v(\delta_1-1)$.
Since $\delta_1\leq q$ and $p\leq \epsilon_i$, $\tau_{p,q}(M_{\overline{v}})$ contains a $1$ in row $\epsilon_t$, namely
at $(\epsilon_t, \delta_1)$.  This contradicts the definition of $s$ and we can thus conclude that $n\pm k+1\notin \{\delta_1,\ldots,\delta_s\}$.

The leading term $m'$ of this minor must be $\prod_{l=1}^s (M_{\overline{vc_k}})_{\epsilon_l\delta_l}$.
This is because any term $m''$ of the last minor can be extended to a term $m''m/m'$ of the $(A,B)$ minor of $M_{\overline{vc_k}}$
by multiplying by $m/m'=\prod_{l=s+1}^r (M_{\overline{vc_k}})_{\epsilon_l\delta_l}$,
but $m\geq m''m/m'$ by choice of~$m$
and the fact that monomial orders respect multiplication implies that $m'\geq m''$.
We conclude that $m'\mid m$ and (1) follows.

We now show (2). 
Suppose that $c_k$ is a descent of $w$ and $m\in\mathcal G_{\overline{vc_k},w,\prec'}$.
Since $(p,q)\in E(w)$, we cannot have $q=n\pm k$ and it follows that $B'\subset [q]$.
We therefore have that $m'$ is the leading term of a minor of $\tau_{p,q}(M_{\overline{v}})$ of size $r_w(p,q)+1$ and (2) follows.

To show (3), first note that, if $B'\subseteq [q]$, then, as in (2), $m'$ is the leading term of a minor of size $r_w(p,q)+1$ lying inside $\tau_{p,q}(M_{\overline{v}})$, so we can take $m''=m'$, and $m'' \mid  z_{ij} m'$.

If $B'\not\subseteq [q]$, then we must have $q=n\pm k$.
Since $c_k$ is an ascent of~$v$, we have $v_q<v_{q+1}$, so $q+1$ is not a left-to-right minimum of~$v$.
Hence $q+1$ is a right-to-left maximum of~$v$ by \Cref{lem: 123 shuffle}, and, by \Cref{rows and columns with ones},
the only nonzero entry in column~$q+1$ of is a~$1$ in position $(v(q+1),q+1)$.
This implies that $v^{-1}(q+1)=\epsilon_j$ for some $j>s$.  Without loss of generality,
suppose $v(\epsilon_{s+1})=q+1$.

Let $B''=B'\setminus \{q+1\} \cup \{q\}$.  Note that $\#B'' = r_w(p,q)+1$
since $q\not\in B'$ as $c_k(q)=q+1>q$.  Let $\delta_{s+1}=q$, so $c_k(\delta_{s+1})=q+1$.
We have $(M_{\overline{v}})_{\epsilon_s,c_k(\delta_{s+1})}=z_{ij}$.

We now construct a row set $A'$ so that the $(A',B'')$-minor of
$\tau_{p,q}(M_{\overline{v}})$ has leading term $m''$ such that $m'' \mid  z_{ij} m'$.
Without loss of generality, assume $p\leq vc_k(\delta_j) <\epsilon_{s+1}$ when
$j\leq t$,
and $p> vc_k(\delta_j)$ or $vc_k(\delta_j)\geq \epsilon_{s+1}$ when $t<j\leq s$.
Now let $A'=\{vc_k(\delta_1),\ldots,vc_k(\delta_t), \epsilon_{t+1},\ldots,\epsilon_r\}$
We show that the leading term of the $(A',B'')$-minor of $M_{\overline{v}}$ is
$m''=\prod_{\ell=t+1}^{s+1} (M_{\overline{v}})_{\epsilon_\ell c_k(\delta_\ell)}$ and that $m'' \mid  z_{ij} m'$.

First we show by contradiction that $(M_{\overline{v}})_{\epsilon_{s+1}c_k(\delta_{s+1})}=z_{ij}$ must be part of the leading term.
Since $c_k$ is the last ascent, every entry in column $c_k(\delta_{s+1})=q$ below row $\epsilon_{s+1}$ must be $0$.
Hence, if the leading term of the $(A',B'')$-minor of $M_{\overline{v}}$
does not contain $z_{ij}$, then there must exist some row $a\in A'$ with $a<\epsilon_{s+1}$
and some column $b\in B''$ with $b<q$ such that
the $(a,q)$ and $(\epsilon_{s+1},b)$ entries are in the leading term of the minor.
Now note that $(M_{\overline{v}})_{ab}=0$, as, otherwise,
by the diagonalness of the term order, this entry and $z_{ij}$ would give a larger term, as seen in Figure \ref{fig: w ascent zij}(i).
Since $(M_{\overline{v}})_{ab}=0$ but $(M_{\overline{v}})_{aq}\neq 0$ and $(M_{\overline{v}})_{\epsilon_{s+1}b}\neq 0$,
we must have that $a<v(b)<\epsilon_{s+1}$, as seen in Figure \ref{fig: w ascent zij}(ii).  By our labelling, we must have
$b=c_k(\delta_j)$ for some $j\leq t$, and $v(b)\in A'$.  Now note that
$b$ must be a left-to-right minimum, so the only nonzero entry in row $v(b)$ is the $1$ at $(v(b),b)$.
This contradicts our assumption that the $(\epsilon_{s+1},b)$ entry of $M_{\overline{v}}$ is in the leading term.

\begin{figure}[ht]
\[
\mbox{(i)}\quad
\begin{matrix}
	& b & & q \\
	a & z_{\diamondsuit\heartsuit} & \ldots & z_{\diamondsuit j}\\
	& \vdots & & \vdots\\
	\epsilon_{s+1} & z_{i\heartsuit}&\ldots& z_{ij}
\end{matrix}
\hskip0.25\textwidth 
\mbox{(ii)}\quad
\begin{matrix}
	& b & & q \\
	a & 0 & \ldots & (\ne0)\\
	& \vdots & & \vdots\\
	v(b) & 1 & \ldots & 0\\
	& \vdots & & \vdots\\
	\epsilon_{s+1} & z_{i\heartsuit}&\ldots& z_{ij}
\end{matrix}
\]
\caption{The entries of~$M_{\overline{v}}$ in relevant rows and columns.}
\label{fig: w ascent zij}
\end{figure}

Next, note that $\tilde{m}=\prod_{\ell=t+1}^{s} (M_{\overline{v}})_{\epsilon_\ell c_k(\delta_\ell)}$ must be the leading term of the
$(\{\epsilon_{t+1},\ldots,\epsilon_s\},$ $\{c_k(\delta_{t+1}),\ldots,c_k(\delta_s)\})$ minor
of $M_{\overline{v}}$, because any other term $\tilde{m}'$ can be extended to a
term $\tilde{m}'m'/\tilde{m}$ of the $(A,B')$ minor of $M_{\overline{v}}$
by multiplying by $m'/\tilde{m}=\prod_{\ell=1}^t (M_{\overline{v}})_{\epsilon_\ell c_k(\delta_\ell)}$,
but $m'\geq \tilde{m}'m'/\tilde{m}$ by (1), and the fact that monomial orders
respect multiplication implies that $\tilde{m}\geq \tilde{m}'$.

We have left-to-right minima at $c_k(\delta_1),\ldots, c_k(\delta_t)$.  Hence
by \Cref{rows and columns with ones}, every term in the $(A',B'')$
minor of $M_{\overline{v}}$ must use the $(vc_k(\delta_\ell), c_k(\delta_\ell))$-th
entries, which are all 1s.  Therefore the leading term is $z_{ij} \tilde{m}$.
By definition, $m''=z_{ij}\tilde{m}$.  By construction, we have $m'' \mid  z_{ij}m'$.
\end{proof}

\begin{ex}
Let $v = {326154}$ and $w = 465213$ in~$C_3$. 
Then the last ascent of $v$ is~$c_1$: that is, in our one-line notation for~$v$, 
position $4=3+1$ is the rightmost position of a digit that is followed by a larger digit.
Since $w$ has a descent at~$c_1$, we are in part (2) of Lemma~\ref{lem: before lunch}. We have 

\[
M_{\overline{v}} = \begin{bmatrix}
	0&0&0 &1&0&0 \\
	0&1&0 &0&0&0\\
	1&0&0 &0&0&0\\
	 z_{11} & z_{12} &0 & z_{13}&0&1 \\
	 z_{12} & z_{22} &0&z_{23}&1&0 \\
	  z_{13} & z_{23} &1 &  0&0&0
	\end{bmatrix}, \quad
	M_{\overline{vc_1}} =
	\begin{bmatrix}
	0&0&0 &0&1&0 \\
	0&0&1 &0&0&0\\
	1&0&0 &0&0&0\\
	 z_{11} & 0 &z_{12} & 0&z_{13}&1 \\
	 z_{12} & 0 &z_{22}&1&0&0 \\
	  z_{13} & 1 &0 &  0&0&0
	\end{bmatrix}, 
\]
\[
	D(w) =
	\begin{tikzpicture}[baseline=(O.base)]
	\node(O) at (1,1.5) {};
	\draw (0,0)--(0.5,0)--(0.5,1.0)--(0,1.0)--(0,0)
		(0,0.5)--(.5,0.5)
		(1.5,1.5)--(2.5,1.5)--(2.5,2)--(1.5,2)--(1.5,1.5)
		(2,2)--(2,1.5); 
	\draw (.25, 3)--(.25,1.25) node {$\bullet$}--(3,1.25)
		(.75,3)--(.75,.25) node {$\bullet$}--(3,.25)
		(1.25,3)--(1.25,.75) node {$\bullet$}--(3,.75)
		(1.75,3)--(1.75,2.25) node {$\bullet$}--(3,2.25)
		(2.25,3)--(2.25,2.75) node {$\bullet$}--(3,2.75)
		(2.75,3)--(2.75,1.75) node {$\bullet$}--(3,1.75);
	\end{tikzpicture}
	\:.
\]
So, the type~C essential set of $w$ is $E(w) = \{(5,1)\}$. Thus,  $\mathcal{G}_{\overline{v},w} = \{z_{13}, z_{12}\} = \mathcal{G}_{\overline{vc_1},w}$ in this case, and so each element of $\mathcal{G}_{\overline{vc_1},w}$ is divisible by an element of $\mathcal{G}_{\overline{v},w}$.
\end{ex}

\begin{ex}
Let $v = v_\square c_0 \in C_5$ and let $w = \rm{a}937654821$, where $\rm{a} = 10$. Then, $E(w) = \{(8,7)\}$, $r_{w}(8,7) = 2$, the last ascent of $v$ is $c_1$, and this is a descent of $w$. We have
\[
\tau_{8,7}(M_{\overline{v}}) = \begin{bmatrix}z_{13}&z_{23}&z_{33}&z_{34}&0&z_{35}&0\\
z_{14}&z_{24}&z_{34}&z_{44}&0&z_{45}&1\\
z_{15}&z_{25}&z_{35}&z_{45}&1&0&0\\
\end{bmatrix}
\] 
and
\[
\tau_{8,7}(M_{\overline{vc_1}}) = \begin{bmatrix}z_{13}&z_{23}&z_{33}&0&z_{34}&0&z_{35}\\
z_{14}&z_{24}&z_{34}&0&z_{44}&1&0\\
z_{15}&z_{25}&z_{35}&1&0&0&0\\
\end{bmatrix}.
\] 
Observe that the $3\times 3$ minor of $\tau_{8,7}(M_{\overline{vc_1}})$ coming from columns $3,5,7$ is $z_{35}z_{44}z_{35}$, so that this term is in $\mathcal{G}_{{\overline{vc_1}},w, \prec'}$ (no matter what $\prec$ is). 
This term is not in $\mathcal{G}_{\overline{v},w, \prec}$. 
In the proof of Lemma~\ref{lem: before lunch}, we take the minor from the last three columns of $\tau_{8,7}(M_{\overline{v}})$, which is $z_{35}$, which divides $z_{35}z_{44}z_{35}$.
\end{ex}

\begin{ex}
Let $v=c_1c_0v_\square = 64218753 \in C_4$ and $w=87436521$.  Then $E(w)=\{(5,4)\}$, $r_w(5,4)=2$,
the last ascent of $v$ is $c_0$, and this is an ascent of $w$.  We have:
\[
\tau_{5,4}(M_{\overline{v}}) = \begin{bmatrix} 0 & z_{22} & z_{23} & z_{24} \\
1 & 0 & 0 & 0 \\
z_{13} & z_{23} & z_{33} & z_{34} \\
z_{14} & z_{24} & z_{34} & z_{44} \\
\end{bmatrix}
\]
and
\[
\tau_{5,4}(M_{\overline{vc_0}}) = \begin{bmatrix} 0 & z_{22} & z_{23} & 0 \\
1 & 0 & 0 & 0 \\
z_{13} & z_{23} & z_{33} & 0 \\
z_{14} & z_{24} & z_{34} & 1 \\
\end{bmatrix}
\]

Observe that the $3\times 3$ minor of $\tau_{5,4}(M_{\overline{vc_k}})$
using columns 1, 3, 4 and rows 5, 7, 8 is $z_{13}z_{23}$.  The leading term of the corresponding
minor in $M_{\overline{v}}$ uses column 5, which is outside $\tau_{5,4}(M_{\overline{v}})$.  In the
proof of \Cref{lem: before lunch}(3), we first replace column 5 with column 4, then row 7 with row 6.
The leading term of this minor of $\tau_{5,4}(M_{\overline{v}})$ is $z_{23}z_{44}$, which divides $z_{13}z_{23}z_{44}$.

\end{ex}

\begin{lemma}\label{lem: w ascent essential box}
Let $w\in C_n$ and $c_k$ be an ascent of~$w$.  Let $(p,q)\in E(wc_k)$.
\begin{enumerate}\renewcommand{\theenumi}{\roman{enumi}}
\item If $q\neq n+k+1$ and $q\neq n+1-k$, then $(p,q)\in E(w)$ and $r_w(p,q)=r_{wc_k}(p,q)$.
\item If $q=n+k+1$ or $q=n+1-k$ and $(p,q-1)\in D(wc_k)$, then $(p,q)\in E(w)$ and $r_w(p,q)=r_{wc_k}(p,q)$. 
\item If $q=n+k+1$ or $q=n+1-k$ and $(p,q-1)\not\in D(wc_k)$, then $(p,q-1)\in E(w)$ and $r_w(p,q-1)=r_{wc_k}(p,q)-1$.
\end{enumerate}
\end{lemma}

\begin{proof}
Note that we cannot have $q=n+k$ or $q=n-k$ since $c_k$ is a descent of~$wc_k$.  
The first statement then follows from \Cref{lem: weak cover diagram}.

Suppose $q=n+k+1$ or $q=n+1-k$ and $(p,q-1)\in D(wc_k)$.  Since $(p,q)\in E^A(wc_k)$, $wc_k(q)<p\leq wc_k(q+1)$.  Since $(p,q-1)\in D(wc_k)$, $wc_k(q-1)<p$.  Hence, $w(q)=wc_k(q-1)<p\leq w(q+1)=wc_k(q+1)$, and $(p,q)\in E^A(w)$.  Also, $r_w(p,q)=r_{wc_k}(p,q)$.  
Since $(p,q)\in E(wc_k)$, $p\geq n+1$.  Furthermore, if $q=n+k+1$ and $(p,2n-q)=(p,n-k-1)\in E^A(w)$, then $(p,n-k-1)\in E^A(wc_k)$.
Applying the second condition of \Cref{def: type C essential} to $(p,q)\in E(wc_k)$ we have that $r_{wc_k}(p,n-k-1)>r_{wc_k}(p,n+k+1)-k-1$.  
Since $r_w(p,n-k-1)=r_{wc_k}(p,n-k-1)$, we also have $r_w(p,2n-q)>r_w(p,q)+n-q$.  Hence, $(p,q)\in E(w)$.

Now suppose $q=n+k+1$ or $q=n+1-k$ and $(p,q-1)\not\in D(wc_k)$.  Then $w(q)=wc_k(q-1)>p$.  Hence, $w(q-1)<p<w(q)$, and $(p,q-1)\in E^A(w)$.  Also, $r_w(p,q-1)=r_{wc_k}(p,q)-1$.  
If $q=n+k+1$ and $(p,2n-(q-1))=(p,n-k)\in E^A(w)$, then $w(n+1-k)\geq p$, so $wc_k(n-k)\geq p$.  
We first wish to show that $(p,n-1-k)\in E^A(wc_k)$.
Since $(p, n+k+1)\in E(wc_k)$ we have $wc_k(n+k+2)\geq p>n$.
Using the symmetry of type~C permutations, as noted in \eqref{eq: symmetry C}, $wc_k(n-1-k)\leq n<p$.
This implies $(p,n-1-k)\in D(wc_k)$ and $(p,n-k)\notin D(wc_k)$.
To conclude that $(p,n-1-k)\in E^A(wc_k)$, we must show that $(p-1,n-1-k)\notin D(wc_k)$.
Note that $(p-1,n-1-k)\notin D(wc_k)$ if and only if $(p-1,n-1-k)\notin D(w)$. We are assuming that $(p,n-k)\in E^A(w)$. Therefore, $w(j)\neq p-1$ for $j>n-k$ and thus we can only have $(p-1,n-1-k)\in D(w)$ if $w(n-k)=p-1$.
However, since $w(q)>p>n$ and $w(n-k)=w(2n+1-q)$, \eqref{eq: symmetry C} implies that $w(n-k)< n\le p-1$.
It follows that $(p,n-1-k)\in E^A(wc_k)$.

Applying the second condition of \Cref{def: type C essential} to $(p,q)\in E(wc_k)$ we have that $r_{wc_k}(p, n-1-k)>r_{wc_k}(p,n+k+1)-k-1$.  
Since $r_w(p,n-k)=r_w(p,n-k-1)=r_{wc_k}(p,n-k-1)$ and $r_w(p,n+k)=r_{wc_k}(p,n+k+1)-1$, we see that $r_w(p,n-k)>r_w(p,n+k)-k$, and $(p,q-1)\in E(w)$.  
\end{proof}

\begin{prop}\label{prop: K in J}
Let $v\ge v_\square$ and $w\in C_n$.  Then $K_{\overline{v},w}\subseteq J_{\overline{v},w, \prec}$
for any diagonal term order $\prec$.
\end{prop}

\begin{proof}
We induct on $\ell(w_0v)$.
In the base case, $v=w_0$, these are both ideals of the polynomial ring $R_{w_0}=\mathbb K$ in zero variables,
namely the zero ideal if $w=w_0$ and the unit ideal otherwise.

If $v\neq w_0$ then it may occur that $v\not\leq w$ so that $K_{\overline{v},w}$ is the unit ideal. 
In this case, $r_{v}(p,q)>r_w(p,q)$ for some $(p,q)\in E(w)$.  
Hence there will be at least
$r_w(p,q)+1$ entries equal to~$1$ in $\tau_{p,q}(M_{\overline{v}})$.  
Taking a size $r_w(p,q)+1$ minor of $\tau_{p,q}(M_{\overline{v}})$ that contains a $1$ in each column,
we see that the minor evaluates to~$1$ (every entry above and to the right of a $1$ being~$0$).
Thus $1\in J_{\overline{v},w}$ as desired.

Henceforth we assume $v\leq w$.
Let $c_k$ be the last ascent of~$v$, let $\{z_{ij}\}=V_{\overline{v}}\setminus V_{\overline{vc_k}}$, and let $\prec'$ be the restriction of $\prec$ to $R_{\overline{vc_k}}$.
We consider two cases according to whether $c_k$ is an ascent or a descent of~$w$.
If a descent, \Cref{thm: vertex decomposition3} says
\[ K_{\overline{v},w} = K_{\overline{vc_k},w}R_{\overline{v}}. \]
By induction, $K_{\overline{vc_k},w}\subseteq J_{\overline{vc_k},w, \prec'}$, since $\prec'$ is a diagonal term order by \Cref{prop:diagonalcompatibility}.
So it suffices to show that $J_{\overline{vc_k},w,\prec'}R_{\overline{v}}\subseteq J_{\overline{v},w,\prec}$; but this is part (2) of \Cref{lem: before lunch}.

If $c_k$ is an ascent of~$w$, then \Cref{thm: vertex decomposition3} says
\begin{equation}\label{eq: VD ascent}
K_{\overline{v},w} = \langle z_{ij}m : m\in K_{\overline{vc_k},w} \rangle + K_{\overline{vc_k},wc_k}R_{\overline{v}}.
\end{equation}
Let $\mhat$ be a monomial generator of~$K_{\overline{v},w}$.
We split into two cases again according to whether $\mhat\in K_{\overline{vc_k},wc_k}$.

If $\mhat\in K_{\overline{vc_k},wc_k}$, then by induction $\mhat\in J_{\overline{vc_k},wc_k, \prec'}$.
Hence there exists $m\in \mathcal{G}_{\overline{vc_k},wc_k, \prec'}$ dividing $\mhat$.
Since $c_k$ is an ascent of $w$, it is a descent of~$wc_k$, and therefore by Lemma~\ref{lem: before lunch},
there exists $m'\in\mathcal G_{\overline{v},wc_k,\prec}$ dividing $m$ and hence $\mhat$.
Since $m'\in\mathcal{G}_{\overline{v},wc_k, \prec}$, it is the leading term of some minor of $\tau_{p,q}(M_{\overline{v}})$, of size $r_{wc_k}(p,q)+1$, for some $(p,q)\in E(wc_k)$. 
Call this minor $D$.
We now apply Lemma~\ref{lem: w ascent essential box}, breaking the argument into its three cases.

If $q\neq n\pm k+1$, then $(p,q)\in E(w)$ and $r_w(p,q)=r_{wc_k}(p,q)$, so $D\in I_{\overline{v},w}$, $m'\in \mathcal G_{\overline{v},w, \prec}$, and therefore $m\in J_{\overline{v},w, \prec}$.
Similarly, if $q= n\pm k+1$ and $(p,q-1)\in D(wc_k)$ then $m\in J_{\overline{v},w, \prec}$.
Lastly, if $q= n\pm k+1$ and $(p,q-1)\not\in D(wc_k)$, the lemma implies that $(p,q-1)\in E(w)$ and $r_w(p,q-1)=r_{wc_k}(p,q)-1$.
If column $q$ is not used in $D$, then $D$ is a minor of $\tau_{p,q-1}(M_{\overline{v}})$ of size $r_w(p,q-1)+2$.
By Laplace expansion, $m'$, being the leading term of~$D$, is divisible by the leading term $m''$ of some minor of size $r_w(p,q-1)+1$.
Hence $m''\in \mathcal{G}_{\overline{v},w, \prec}$ and $\mhat\in J_{\overline{v},w, \prec}$.
Otherwise, if column $q$ is used in $D$, let $(a,q)$ be the position in column~$q$ appearing in the leading term.
Then $m'':=m'/(M_{\overline{v}})_{a,q}$ is the leading term of the minor of size $r_{wc_k}(p,q)-1$ obtained by omitting row~$a$ and column~$q$.
Therefore $m''\in\mathcal G_{\overline{v},w, \prec}$ and $\mhat\in J_{\overline{v},w, \prec}$. 

If instead $\mhat\not\in K_{\overline{vc_k},wc_k}$, then
$\mhat/z_{ij}\in K_{\overline{vc_k},w}$ and by induction $\mhat/z_{ij}\in J_{\overline{vc_k},w, \prec'}$.
Let $m\in\mathcal G_{\overline{vc_k},w}$ be a generator dividing $\mhat/z_{ij}$, arising from a minor of $\tau_{p,q}(M_{\overline{vc_k}})$.
Part (1) of Lemma~\ref{lem: before lunch} (which does not literally apply because $c_k$ is not a descent of~$w$)
produces a minor $D$ of~$M_{\overline{v}}$ of size $r_w(p,q)+1$
with a leading term $m'$ that divides $m$, and part (3) produces $m''\in\mathcal{G}_{\overline{v},w, \prec}$ dividing $z_{ij}m'$.
Hence we have $m'' \mid  z_{ij} m' \mid  z_{ij} m \mid  \mhat$, and $\mhat \in J_{\overline{v},w, \prec}$.
\end{proof}

\subsection{Equality of \texorpdfstring{$K$}{K}-polynomials}
We wish to show that, with respect to the weighting of the variables on $R_{\overline{v}}$  introduced in Section \ref{sect:TActionSmall}, the $K$-polynomial of the $R_{\overline{v}}$-module $R_{\overline{v}}/I_{\overline{v},w}$ is equal to the $K$-polynomial of $R_{\overline{v}}/K_{\overline{v},w}$. We do this by showing that the $K$-polynomial of $R_{\overline{v}}/K_{\overline{v},w}$ satisfies the recursion of the next theorem, due to B.~Kostant and S.~Kumar \cite[Proposition 2.4]{KostantKumar}.  Kumar showed that the local $K$-classes of $X_w$, and hence the $K$-polynomials of $R_{\overline{v}}/I_{\overline{v},w}$, follow this recursion.  The version below is due to Knutson (see \cite[Theorem 1]{Knutson-degenerate}) and has been translated to our specific setting. See also \cite[Theorem 6.12]{WooYongGrobner} for the type~A analog of the statement below. 

Throughout this subsection we assume that $\mathbb{K}= \mathbb{C}$ as the recursion on $K$-polynomials that we reference is only proved in that setting. In the next subsection we show that, despite this assumption, Theorem \ref{thm: main} holds over arbitrary $\mathbb{K}$ of characteristic zero.

\begin{thm}\label{thm:KostKum}
Let $v,w\in C_n$. 
\begin{itemize}
\item If $v\not\leq w$, then
\[
\mathcal{K}(R_{\overline{v}}/I_{\overline{v},w};\mathbf{t}) = 0.
\]
\item If $v = w_0$, then $w = w_0$ (or we are in the previous case). Then
\[
\mathcal{K}(R_{\overline{v}}/I_{\overline{v},w};\mathbf{t}) = 1.
\]
\item Otherwise, let $k$ be the last right ascent of $v$, so $vs_k>v$. 
\begin{enumerate}
\item If $k$ is a descent of $w$, so $cs_k<w$, then
\[
\mathcal{K}(R_{\overline{v}}/I_{\overline{v},w};\mathbf{t}) = \mathcal{K}( R_{\overline{vc_k}} / I_{\overline{vc_k},w};\mathbf{t});
\]
\item If $k$ is an ascent of $w$, so $cs_k>w$, then
\[
\mathcal{K}(R_{\overline{v}}/I_{\overline{v},w};\mathbf{t}) = \mathcal{K}(R_{\overline{vc_k}}/I_{\overline{vc_k},w};\mathbf{t})+(1-t_it_j)\mathcal{K}(R_{\overline{vc_k}}/I_{\overline{vc_k},wc_k};\mathbf{t})-(1-t_it_j)\mathcal{K}(R_{\overline{vc_k}}/I_{\overline{vc_k},w}; \mathbf{t}),
\]
where $(i,j)$ is such that $z_{ij}$ is the smallest variable with respect to $\prec_{\mathrm{lex}}$ on~$R_{\overline{v}}$.
\end{enumerate}
\end{itemize}
\end{thm}

\begin{proof} This is \cite[Theorem 1]{Knutson-degenerate} in the type~C setting, but we should explain two changes that arise in translating the statement to our setting.

First, the original statement is about the pullback of the class $[\mathcal{O}_{X_w}]\in K_T(G/B_G^+)$ to $K_T(vB_G^+/B_G^+)$, where $vB_G^+/B_G^+$ is the point in the stratification of~$G/B_G^+$.  In commutative algebra terms, this is 
\[\sum_{i} (-1)^i[\Tor_i (\mathcal{O}_{X_w}, \mathcal{O}_{vB_G^+/B_G^+})]\in K_T(vB_G^+/B_G^+).\]
A class in $K_T(\mathrm{pt})$ can be identified with its formal character.  On the other hand, since $vB_G^+/B_G^+$ is the point in $M_{\overline{v}}$ with all coordinates set to 0, taking $\Tor$ with $\mathcal{O}_{vB_G^+/B_G^+}$ is the same as taking the $K$-polynomial.

Second, we need to match $e^{v(\alpha)}$ with $t_it_j$.  As shown in the proof of
\Cref{prop: partial symm matrices}, we have $n+i=v(n+k+1)$ and $n+1-j=v(n+k)$, so $v(e_k)=-e_j$ and $v(e_{k+1})=e_i$.  Since $\alpha=e_{k+1}-e_k$, $v(\alpha)=e_i+e_j$, so $e^{v(\alpha)}=t_it_j$.  (To be precise,
there are two cancelling sign differences from the original statement, one from our definition of $X_w$ as $B_G^+$-orbit closures rather than $B_G^{-}$-orbit closures,
and the second from our use of $w_0v$ instead of $v$.)
\end{proof}

\begin{prop}\label{prop: K(K) = K(I)}
Given $v,w\in C_n$, 
\[\mathcal{K}(R_{\overline{v}}/I_{\overline{v},w};\mathbf{t}) = \mathcal{K}(R_{\overline{v}}/K_{\overline{v},w};\mathbf{t}).
\]
\end{prop}

\begin{proof}
\cite[Theorem 1.13]{MillerSturmfels}
implies that given a simplicial complex $\Delta$ on the vertex set $V\mathbin{\dot\cup}\{z\}$ with its vertex decomposition at~$z$,
say $\Delta = \cone_z\Lambda\cup\Pi$ where $\Lambda\subseteq\Pi$ are simplicial complexes on~$V$ (respectively the link and deletion of $z$ in~$\Delta$), we have
\[\mathcal{K}(R[V\cup\{z\}]/I_\Delta;\mathbf{t}) = \mathcal{K}(R[V]/I_\Lambda;\mathbf{t}) + (1-\mathbf{t}^{\deg(z)}) \mathcal{K}(R[V]/I_\Pi;\mathbf{t}) - (1-\mathbf{t}^{\deg(z)}) \mathcal{K}(R[V]/I_\Lambda;\mathbf{t}).\]
Applied to Proposition~\ref{thm: vertex decomposition2},
this produces a recursive formula for $\mathcal{K}(R/K_{\overline{v},w};\mathbf{t})$,
which comes out identical to Theorem~\ref{thm:KostKum} with every appearance of $I$ replaced by~$K$. 
Thus, $\mathcal{K}(R/I_{\overline{v},w};\mathbf{t})$ and $\mathcal{K}(R/K_{\overline{v},w};\mathbf{t})$ satisfy the same recursion, and so are the same. 
\end{proof}

\subsection{Proof of Theorem \ref{thm: main}}

\begin{proof}
We first assume $\mathbb{K} = \mathbb{C}$, and return to an arbitrary field of characteristic zero in the last paragraph. 
Fix a diagonal term order $\prec$.  By \Cref{prop: K in J},
$K_{\overline{v},w}\subseteq J_{\overline{v}, w, \prec}$.
By definition, $J_{\overline{v}, w,\prec}\subseteq \init_\prec I_{\overline{v},w}$.
Hence we have surjections
$$R/K_{\overline{v},w} \twoheadrightarrow R/J_{\overline{v},w,\prec} \twoheadrightarrow R/\init_\prec I_{\overline{v},w}.$$

Now \Cref{prop: K(K) = K(I)} states that 
\[\mathcal{K}(R_{\overline{v}}/I_{\overline{v},w};\mathbf{t}) = \mathcal{K}(R_{\overline{v}}/K_{\overline{v},w};\mathbf{t}).
\]
Since
\[\mathcal{K}(R_{\overline{v}}/I_{\overline{v},w};\mathbf{t}) = \mathcal{K}(R_{\overline{v}}/\init_\prec I_{\overline{v},w};\mathbf{t}),
\]
the above containments are actually equalities, and
$$J_{\overline{v},w,\prec}=\init_\prec I_{\overline{v},w},$$
as desired.

To complete the proof, we note that since the essential minors in $I_{\overline{v},w}$ are polynomials with $\mathbb{Z}$ coefficients, the essential minors are a Gr\"obner basis over any field $\mathbb{K}$ of characteristic zero. Indeed, if $f$ and $g$ are essential minors, then the $S$-polynomial $S(f,g)$ reduces to $0$ under division by the essential minors when working over $\mathbb{Q}$; hence it does over any field of characteristic zero.
\end{proof}

The above proof also gives the following corollary.

\begin{cor}\label{cor: K is initial}
Under any diagonal term order, the initial ideal of $I_{\overline{v},w}$ is $K_{\overline{v},w}$.
\end{cor}

By Lemma \ref{lem:powerof2}, all coefficients of essential minors are powers of $2$. 

\begin{conj}
Theorem \ref{thm: main} holds over an arbitrary field $\mathbb{K}$ of characteristic not $2$.
\end{conj}

\section{\texorpdfstring{$K$}{K}-polynomial formulas via pipe dreams}\label{sec: formulas}

\subsection{Type~C pipe dreams on small patches}

In this section, we recall the notion of pipe dreams associated to pairs of type~A permutations $v,w$. We then define a closely related notion of type~C pipe dreams, specifically for pairs of permutations $v,w\in C_n$ with $v\ge v_\square$. These are simply type~A pipe dreams with symmetry imposed about the diagonal.

We begin with pipe dream complexes for pairs of permutations in $S_{m}$.
Pipe dreams were invented by S.~Fomin and A.~Kirillov in \cite{Fomin-Kirillov} and further studied by N.~Bergeron and S.~Billey in \cite{Bergeron-Billey}.  Knutson and Miller \cite{KnutsonMillerAdvances,KnutsonMillerAnnals} endowed them with the structure of a simplicial complex, namely a subword complex.
A (type~A) {\bf pipe dream} is a tiling of the entries in the southwest triangle of an $m\times m$ matrix with the tiles \textbf{cross} \textcross, \textbf{elbow} \textelbow, and \textbf{half elbow} \textcelbow\ such that:
	\begin{enumerate}
	\item the diagonal is tiled with \textcelbow, and
	\item the weak southwest triangle only uses \textcross\ and \textelbow.
\end{enumerate}
A pipe dream $\rho$ induces an arrangement of $m$ pseudolines and $\rho$ is \textbf{reduced} if no two pseudolines cross twice.
A pipe dream $\rho$ is \textbf{contained} in another pipe dream $\rho'$ if the set of positions of the elbows of $\rho$ is contained in the set of positions of the elbows of $\rho'$: see Figure~\ref{fig: type A pipe dreams} for an example.

Label the west ends of the pseudolines $1,\ldots,m$ bottom to top in the order of their incidence with the west boundary, as in Figure~\ref{fig: type A pipe dreams}.
Also label the south ends of the pseudolines with $1,\ldots,m$
by transporting the west labels along the pseudolines, except
ignoring all crossings subsequent to the first between each pair of pseudolines,
i.e.\ moving the labels as if such crosses were elbows instead.
Then a pipe dream for $w\in S_m$ is a pipe dream whose labels along the south boundary read $w$.
See Figure~\ref{fig: type A pipe dreams} for two examples of pipe dreams for $1432$.

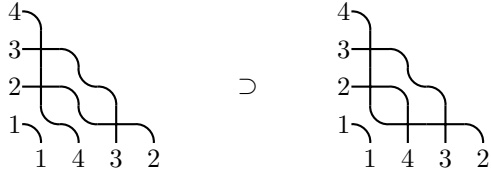
\begin{figure}[h]
\begin{tikzpicture}
	 \draw node at (-.1,.25) {$1$}
	 	node at (-.1,.75) {$2$}
	 	node at (-.1,1.25) {$3$}
	 	node at (-.1,1.75) {$4$};
	  \begin{scope}[shift={(0,0)}]{\elbow}  \end{scope}
	  \begin{scope}[shift={(.5,0)}]{\elbow}  \end{scope}
	  \begin{scope}[shift={(1,0)}]{\+}  \end{scope}
	  \begin{scope}[shift={(1.5,0)}]{\celbow}  \end{scope}
	  \begin{scope}[shift={(0,.5)}]{\+}  \end{scope}
	  \begin{scope}[shift={(.5,.5)}]{\elbow}  \end{scope}
	  \begin{scope}[shift={(1,.5)}]{\celbow}  \end{scope}
	  \begin{scope}[shift={(0,1)}]{\+}  \end{scope}
	  \begin{scope}[shift={(.5,1)}]{\celbow}  \end{scope}
	  \begin{scope}[shift={(0,1.5)}]{\celbow}  \end{scope}
	 \draw node at (.25,-.2) {$1$}
	 	node at (.75,-.2) {$4$}
	 	node at (1.25,-.2) {$3$}
	 	node at (1.75,-.2) {$2$};
	\draw node at (3,.75) {$\supset$};
\end{tikzpicture}
\qquad
\begin{tikzpicture}
	 \draw node at (-.1,.25) {$1$}
	 	node at (-.1,.75) {$2$}
	 	node at (-.1,1.25) {$3$}
	 	node at (-.1,1.75) {$4$};
	  \begin{scope}[shift={(0,0)}]{\elbow}  \end{scope}
	  \begin{scope}[shift={(.5,0)}]{\+}  \end{scope}
	  \begin{scope}[shift={(1,0)}]{\+}  \end{scope}
	  \begin{scope}[shift={(1.5,0)}]{\celbow}  \end{scope}
	  \begin{scope}[shift={(0,.5)}]{\+}  \end{scope}
	  \begin{scope}[shift={(.5,.5)}]{\elbow}  \end{scope}
	  \begin{scope}[shift={(1,.5)}]{\celbow}  \end{scope}
	  \begin{scope}[shift={(0,1)}]{\+}  \end{scope}
	  \begin{scope}[shift={(.5,1)}]{\celbow}  \end{scope}
	  \begin{scope}[shift={(0,1.5)}]{\celbow}  \end{scope}
	 \draw node at (.25,-.2) {$1$}
	 	node at (.75,-.2) {$4$}
	 	node at (1.25,-.2) {$3$}
	 	node at (1.75,-.2) {$2$};
\end{tikzpicture}
\caption{The pipe dream on the left is reduced and contains the pipe dream on the right, which is not reduced. These are both pipe dreams for $1432$.}
\label{fig: type A pipe dreams}
\end{figure}

Let $v\in S_m$ and consider the pictorial description for $D(v)$ described in Section~\ref{sec: weak order}.
Let us denote by $D_{\sf L}(v)$ the diagram obtained from $D(v)$ after left-aligning.
The (type~A) \textbf{pipe dream complex} $PD^A_{v,w}$ for $v,w\in S_{m}$ is the simplicial complex with vertices given by the boxes of $D_{\sf L}(v)$, 
and one facet for each reduced pipe dream $\rho$ for $w$ whose crosses are contained in $D_{\sf L}(v)$,
the set of vertices in the facet being the set of positions of elbows in~$\rho$.
We will abuse notation and reuse the name $\rho\in PD^A_{v,w}$ for the facet.

Associate to a pipe dream $\rho$ a word in the nilHecke algebra by assigning simple reflections $s_k$ to cross tiles as follows:
if a cross tile appears in position $(i,j)$, then assign $s_{i-j}$. 
(This has the effect of making all instances of~$s_k$ fall on the $k$\/th diagonal below the main diagonal.)
We obtain a word by reading the reflections $s_k$ in the leftmost column from top to bottom, then the next column from top to bottom, etc. 
We remark that if $\rho$ is a pipe dream for $w\in S_m$, then the resulting word is a word in the nilHecke algebra for $w_0w$,
and that the pipe dream with crosses exactly in $D_{\sf L}(v)$ is a reduced pipe dream for $v$.

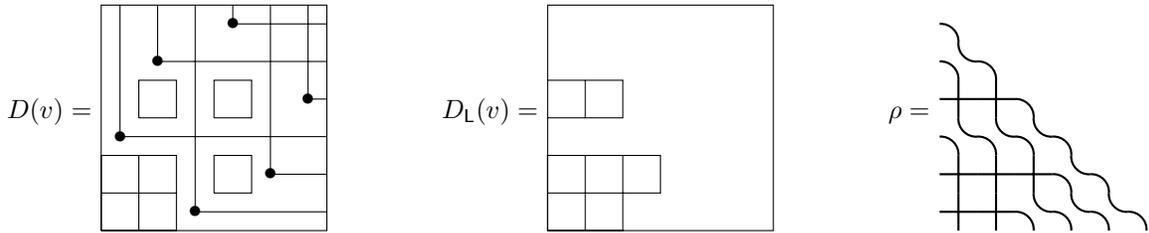
\begin{figure}[h]
	$D(v) =$
	\begin{tikzpicture}[baseline=(O.base)]
	\node(O) at (1,1.5) {};
	\draw (0,0)--(3,0)--(3,3)--(0,3)--(0,0)
	        (0,0)--(1,0)--(1,1.0)--(0,1.0)--(0,0)
		(0,0.5)--(1,0.5)
		(0.5,0)--(0.5,1)
		(1.5,1.5)--(2,1.5)--(2,2)--(1.5,2)--(1.5,1.5)
		(.5,1.5)--(1,1.5)--(1,2)--(.5,2)--(.5,1.5)
		(1.5,.5)--(2,.5)--(2,1)--(1.5,1)--(1.5,.5)
		; 
	\draw (.25, 3)--(.25,1.25) node {$\bullet$}--(3,1.25)
		(.75,3)--(.75,2.25) node {$\bullet$}--(3,2.25)
		(1.25,3)--(1.25,.25) node {$\bullet$}--(3,.25)
		(1.75,3)--(1.75,2.75) node {$\bullet$}--(3,2.75)
		(2.25,3)--(2.25,.75) node {$\bullet$}--(3,.75)
		(2.75,3)--(2.75,1.75) node {$\bullet$}--(3,1.75);
\end{tikzpicture}
	\qquad
	\qquad
		$D_{\sf L}(v) =$
	\begin{tikzpicture}[baseline=(O.base)]
	\node(O) at (1,1.5) {};
	\draw (0,0)--(3,0)--(3,3)--(0,3)--(0,0)
		(0,0)--(1,0)--(1,1.0)--(0,1.0)--(0,0)
		(0,0.5)--(1,0.5)
		(0.5,0)--(0.5,1)
		(0,1.5)--(1,1.5)--(1,2)--(0,2)
		(.5,1.5)--(.5,2)
		(1,.5)--(1.5,.5)--(1.5,1)--(1,1)
		; 
\end{tikzpicture}
	\qquad
	\qquad
	$\rho=$
	\begin{tikzpicture}[baseline=(O.base)]
	\node(O) at (1,1.5) {};
	  \begin{scope}[shift={(0,0)}]{\+}  \end{scope}
	  \begin{scope}[shift={(.5,0)}]{\+}  \end{scope}
	  \begin{scope}[shift={(1,0)}]{\elbow}  \end{scope}
	  \begin{scope}[shift={(1.5,0)}]{\elbow}  \end{scope}
	  \begin{scope}[shift={(2,0)}]{\elbow}  \end{scope}
	  \begin{scope}[shift={(2.5,0)}]{\celbow}  \end{scope}
	  \begin{scope}[shift={(0,.5)}]{\+}  \end{scope}
	  \begin{scope}[shift={(.5,.5)}]{\+}  \end{scope}
	  \begin{scope}[shift={(1,.5)}]{\+}  \end{scope}
	  \begin{scope}[shift={(1.5,0.5)}]{\elbow}  \end{scope}
	  \begin{scope}[shift={(2,0.5)}]{\celbow}  \end{scope}
	  \begin{scope}[shift={(0,1)}]{\elbow}  \end{scope}
	  \begin{scope}[shift={(.5,1)}]{\elbow}  \end{scope}
	  \begin{scope}[shift={(1,1)}]{\elbow}  \end{scope}
	  \begin{scope}[shift={(1.5,1)}]{\celbow}  \end{scope}
	  \begin{scope}[shift={(0,1.5)}]{\+}  \end{scope}
	  \begin{scope}[shift={(.5,1.5)}]{\+}  \end{scope}
	  \begin{scope}[shift={(1,1.5)}]{\celbow}  \end{scope}
	  \begin{scope}[shift={(0,2)}]{\elbow}  \end{scope}
	  \begin{scope}[shift={(.5,2)}]{\celbow}  \end{scope}
	  \begin{scope}[shift={(0,2.5)}]{\celbow}  \end{scope}
\end{tikzpicture}
\caption{The diagram $D(v)$, left-aligned diagram $D_{\sf L}(v)$, and corresponding reduced pipe dream $\rho$ for $v=426153$. The word corresponding to $\rho$ is $s_2s_4s_5s_1s_3s_4s_2$ which is a reduced expression for $w_0v$.}
\label{fig: type A reduced word}
\end{figure}

	\begin{figure}[h]
	\begin{tikzpicture}
  \filldraw[thick,fill=cyan!30]  (-2,5) node {$\bullet$}--(-4.6-1,-.1) node {$\bullet$}--(-.5,-.1) node {$\bullet$}--(3.5,-.1) node {$\bullet$}--(-2,5)--(-.5,-.1);
	  \begin{scope}[shift={(0,0)}]{\+}  \end{scope}
	  \begin{scope}[shift={(.5,0)}]{\+}  \end{scope}
	  \begin{scope}[shift={(1,0)}]{\elbow}  \end{scope}
	  \begin{scope}[shift={(1.5,0)}]{\elbow}  \end{scope}
	  \begin{scope}[shift={(2,0)}]{\elbow}  \end{scope}
	  \begin{scope}[shift={(2.5,0)}]{\celbow}  \end{scope}
	  \begin{scope}[shift={(0,.5)}]{\+}  \end{scope}
	  \begin{scope}[shift={(.5,.5)}]{\elbow}  \end{scope}
	  \begin{scope}[shift={(1,.5)}]{\elbow}  \end{scope}
	  \begin{scope}[shift={(1.5,0.5)}]{\elbow}  \end{scope}
	  \begin{scope}[shift={(2,0.5)}]{\celbow}  \end{scope}
	  \begin{scope}[shift={(0,1)}]{\elbow}  \end{scope}
	  \begin{scope}[shift={(.5,1)}]{\elbow}  \end{scope}
	  \begin{scope}[shift={(1,1)}]{\elbow}  \end{scope}
	  \begin{scope}[shift={(1.5,1)}]{\celbow}  \end{scope}
	  \begin{scope}[shift={(0,1.5)}]{\+}  \end{scope}
	  \begin{scope}[shift={(.5,1.5)}]{\elbow}  \end{scope}
	  \begin{scope}[shift={(1,1.5)}]{\celbow}  \end{scope}
	  \begin{scope}[shift={(0,2)}]{\elbow}  \end{scope}
	  \begin{scope}[shift={(.5,2)}]{\celbow}  \end{scope}
	  \begin{scope}[shift={(0,2.5)}]{\celbow}  \end{scope}
	  \begin{scope}[shift={(-3.5+0,0)}]{\+}  \end{scope}
	  \begin{scope}[shift={(-3.5+.5,0)}]{\+}  \end{scope}
	  \begin{scope}[shift={(-3.5+1,0)}]{\elbow}  \end{scope}
	  \begin{scope}[shift={(-3.5+1.5,0)}]{\elbow}  \end{scope}
	  \begin{scope}[shift={(-3.5+2,0)}]{\elbow}  \end{scope}
	  \begin{scope}[shift={(-3.5+2.5,0)}]{\celbow}  \end{scope}
	  \begin{scope}[shift={(-3.5+0,.5)}]{\+}  \end{scope}
	  \begin{scope}[shift={(-3.5+.5,.5)}]{\elbow}  \end{scope}
	  \begin{scope}[shift={(-3.5+1,.5)}]{\+}  \end{scope}
	  \begin{scope}[shift={(-3.5+1.5,0.5)}]{\elbow}  \end{scope}
	  \begin{scope}[shift={(-3.5+2,0.5)}]{\celbow}  \end{scope}
	  \begin{scope}[shift={(-3.5+0,1)}]{\elbow}  \end{scope}
	  \begin{scope}[shift={(-3.5+.5,1)}]{\elbow}  \end{scope}
	  \begin{scope}[shift={(-3.5+1,1)}]{\elbow}  \end{scope}
	  \begin{scope}[shift={(-3.5+1.5,1)}]{\celbow}  \end{scope}
	  \begin{scope}[shift={(-3.5+0,1.5)}]{\elbow}  \end{scope}
	  \begin{scope}[shift={(-3.5+.5,1.5)}]{\elbow}  \end{scope}
	  \begin{scope}[shift={(-3.5+1,1.5)}]{\celbow}  \end{scope}
	  \begin{scope}[shift={(-3.5+0,2)}]{\elbow}  \end{scope}
	  \begin{scope}[shift={(-3.5+.5,2)}]{\celbow}  \end{scope}
	  \begin{scope}[shift={(-3.5+0,2.5)}]{\celbow}  \end{scope}
\end{tikzpicture}
\caption{The pipe dream complex $PD^A_{v,w}$ for $v=426153$ and $w=456231$.}
\end{figure}
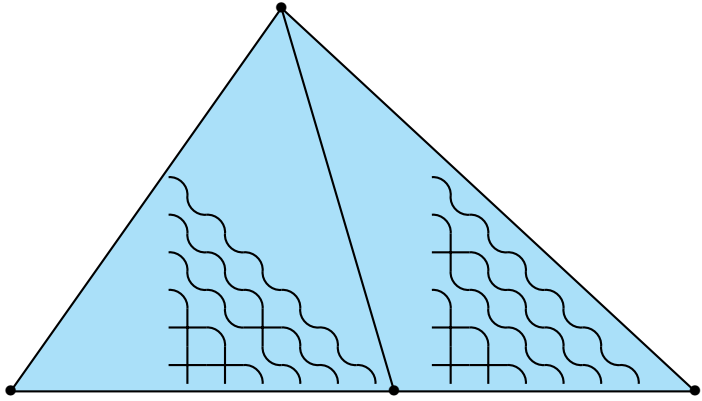

We now describe type~C pipe dream complexes for permutations $v\ge v_\square$ and $w$.
A \textbf{type~C pipe dream} is a (type~A) pipe dream whose crosses lie within $D_{\sf L}(v_\square)=D(v_\square)$ and which is symmetric about the diagonal of $D(v_\square)$.
Since all the tiles outside of $D(v_\square)$ are elbows, we only draw the tiles inside this region, which is an $n\times n$ square.
We give the positions in this region coordinates from $(1,1)$ to $(n,n)$,
rather than using the coordinates $(n+1,1)$ to $(2n,n)$ they would inherit from their inclusion in the diagrams for type~A pipe dreams.
	\begin{figure}[h]
	\begin{tikzpicture}
	 \draw node at (-.1,.25) {$1$}
	 	node at (-.1,.75) {$2$}
	 	node at (-.1,1.25) {$3$};
	 \draw node at (.25,1.7) {$4$}
	 	node at (.75,1.7) {$5$}
	 	node at (1.25,1.7) {$6$};
	  \begin{scope}[shift={(0,0)}]{\elbow}  \end{scope}
	  \begin{scope}[shift={(.5,0)}]{\elbow}  \end{scope}
	  \begin{scope}[shift={(1,0)}]{\elbow}  \end{scope}
	  \begin{scope}[shift={(0,.5)}]{\+}  \end{scope}
	  \begin{scope}[shift={(.5,.5)}]{\elbow}  \end{scope}
	  \begin{scope}[shift={(1,.5)}]{\elbow}  \end{scope}
	  \begin{scope}[shift={(0,1)}]{\+}  \end{scope}
	  \begin{scope}[shift={(.5,1)}]{\+}  \end{scope}
	  \begin{scope}[shift={(1,1)}]{\elbow}  \end{scope}
	 \draw node at (.25,-.2) {$1$}
	 	node at (.75,-.2) {$4$}
	 	node at (1.25,-.2) {$2$};
	 \draw node at (1.65,.25) {$5$}
	 	node at (1.65,.75) {$3$}
	 	node at (1.65,1.25) {$6$};
\end{tikzpicture}
	\qquad \qquad \qquad
	\begin{tikzpicture}
	 \draw node at (-.1,.25) {$1$}
	 	node at (-.1,.75) {$2$}
	 	node at (-.1,1.25) {$3$};
	 \draw node at (.25,1.7) {$4$}
	 	node at (.75,1.7) {$5$}
	 	node at (1.25,1.7) {$6$};
	  \begin{scope}[shift={(0,0)}]{\elbow}  \end{scope}
	  \begin{scope}[shift={(.5,0)}]{\+}  \end{scope}
	  \begin{scope}[shift={(1,0)}]{\elbow}  \end{scope}
	  \begin{scope}[shift={(0,.5)}]{\elbow}  \end{scope}
	  \begin{scope}[shift={(.5,.5)}]{\elbow}  \end{scope}
	  \begin{scope}[shift={(1,.5)}]{\elbow}  \end{scope}
	  \begin{scope}[shift={(0,1)}]{\+}  \end{scope}
	  \begin{scope}[shift={(.5,1)}]{\+}  \end{scope}
	  \begin{scope}[shift={(1,1)}]{\elbow}  \end{scope}
	 \draw node at (.25,-.2) {$1$}
	 	node at (.75,-.2) {$4$}
	 	node at (1.25,-.2) {$2$};
	 \draw node at (1.65,.25) {$5$}
	 	node at (1.65,.75) {$3$}
	 	node at (1.65,1.25) {$6$};
\end{tikzpicture}
\caption{The diagram on the left is a type~C pipe dream for $635241$, however the diagram on the right is not.}
\end{figure}
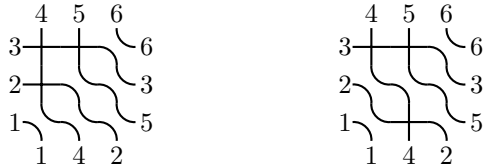

Just as in the type~A case, we associate to a type~C pipe dream $\rho$ a word in the nilHecke algebra.
We start by assigning a simple reflection $c_k$ to each tile in
$D(v_\square)$: assign the generator $c_{j-i}$ to the positions $(j,i)$ and $(i,j)$ in $D(v_\square)$,
these being the positions where the variable $z_{ij}$ appears in the southwest quarter of~$M_{v_\square}$.
Next, given a type~C pipe dream $\rho$, assign to each cross the generator $c_k$ corresponding to the position of  the cross.
Last, for each cross in the weak lower triangular part of $D(v_\square)$ read the generators
in the leftmost column from top to bottom, then the next column from top to bottom, etc. 
The result is the word associated to $\rho$.
We say that $\rho$ is a type~C pipe dream for $w\in C_n$ if the word associated to $\rho$ is a word for $w_0w$.
A type~C pipe dream is \textbf{reduced} if the word is reduced.
We will prove in \Cref{sec: combo pd C} that, just as for type~A, 
if we transport labels $1,\ldots,2n$ from the north and west sides of the picture along the pseudolines,
ignoring all crossings subsequent to the first between each pair of pseudolines,
the resulting labels on the south side read $w$.

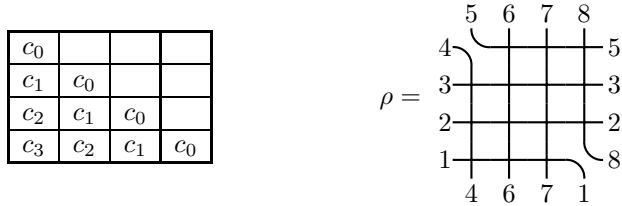
\begin{figure}[h]
	\begin{tabular}{|l|l|l|l|}
	\hline
$c_0$ &  &  &   \\
\hline
$c_1$ & $c_0$ &  &   \\
\hline
$c_2$ & $c_1$ & $c_0$ &  \\
\hline
$c_3$ & $c_2$ & $c_1$ & $c_0$ \\
 \hline 
\end{tabular}
	\qquad
	\qquad
	\qquad
	$\rho=$
		\begin{tikzpicture}[baseline=(O.base)]
	\node(O) at (1,1) {};
	 \draw node at (-.1,.25) {$1$}
	 	node at (-.1,.75) {$2$}
	 	node at (-.1,1.25) {$3$}
	 	node at (-.1,1.75) {$4$};
	 \draw node at (.25,2.2) {$5$}
	 	node at (.75,2.2) {$6$}
	 	node at (1.25,2.2) {$7$}
	 	node at (1.75,2.2) {$8$};
	  \begin{scope}[shift={(0,0)}]{\+}  \end{scope}
	  \begin{scope}[shift={(.5,0)}]{\+}  \end{scope}
	  \begin{scope}[shift={(1,0)}]{\+}  \end{scope}
	  \begin{scope}[shift={(1.5,0)}]{\elbow}  \end{scope}
	  \begin{scope}[shift={(0,.5)}]{\+}  \end{scope}
	  \begin{scope}[shift={(.5,.5)}]{\+}  \end{scope}
	  \begin{scope}[shift={(1,.5)}]{\+}  \end{scope}
	  \begin{scope}[shift={(1.5,.5)}]{\+}  \end{scope}
	  \begin{scope}[shift={(0,1)}]{\+}  \end{scope}
	  \begin{scope}[shift={(.5,1)}]{\+}  \end{scope}
	  \begin{scope}[shift={(1,1)}]{\+}  \end{scope}
	  \begin{scope}[shift={(1.5,1)}]{\+}  \end{scope}
	  \begin{scope}[shift={(0,1.5)}]{\elbow}  \end{scope}
	  \begin{scope}[shift={(.5,1.5)}]{\+}  \end{scope}
	  \begin{scope}[shift={(1,1.5)}]{\+}  \end{scope}
	  \begin{scope}[shift={(1.5,1.5)}]{\+}  \end{scope}
	 \draw node at (.25,-.2) {$4$}
	 	node at (.75,-.2) {$6$}
	 	node at (1.25,-.2) {$7$}
	 	node at (1.75,-.2) {$1$};
	 \draw node at (2.15,.25) {$8$}
	 	node at (2.15,.75) {$2$}
	 	node at (2.15,1.25) {$3$}
	 	node at (2.15,1.75) {$5$};
\end{tikzpicture}
\caption{On the left is the assignment of generators to the weak lower triangular part of $D(v_\square)$.
On the right is a reduced type~C pipe dream $\rho$ for $v=53281764$. The word corresponding to $\rho$ is $c_1c_2c_3c_0c_1c_2c_0c_1$ which is a reduced expression for $w_0v$.}
\label{fig: type C reduced word}
\end{figure}

Let $v\in C_n$ with $v\ge v_\square$.
Note that there is at least one type~C pipe dream for $v$. For example, a factorization $\overline{v}$ of $v$ gives a type~C pipe dream, which we will call $\rho(\overline{v})$, consisting of crosses in the positions of $z_{ij}$ in the southwest quarter of~$M_{v_\square}$ for each $z_{ij}\in V_{\overline{v}}$.

Given a type~C pipe dream $\rho$, we denote by $\rho_{\rm L}$ the set of positions of its crosses in the weak lower triangular part of $D(v_\square)$.
Given a reduced pipe dream $\rho$ for $v$ and some $w\in C_n$, the \textbf{type~C pipe dream complex} $PD^C_{\rho,w}$ is the simplicial complex whose vertices are the boxes in $\rho_{\rm L}$.
The vertices of $PD^C_{\overline{v},w}$ are thus a subset of the entries in the weak lower triangular part of $D(v_\square)$.
The facets of $PD^C_{\rho,w}$ correspond to the reduced type~C pipe dreams $\sigma$ for $w$ such that $\rho\subset \sigma$;
the vertices in such a facet are the positions of the elbows of~$\sigma$ that lie in $\rho_{\rm L}$.
If $\rho = \rho(\overline{v})$, we abbreviate $PD^C_{\rho(\overline{v}),w}$ to $PD^C_{\overline{v},w}$.
From the definition, we see that type~C pipe dreams which contain a type~C pipe dream for a reduced type~C word for $w$ are in bijection with faces of $\Delta_{\overline{v},w}$.

	\begin{figure}[h]
	\begin{tikzpicture}
	\draw[thick] (0-1.25,-.1) node {$\bullet$}--(2-.4,0-.1) node {$\bullet$}--(4.55,-.1) node {$\bullet$}--(7.5,0-.1) node {$\bullet$}
		;
	  \begin{scope}[shift={(-.75+0,0+0)}]{\+}  \end{scope}
	  \begin{scope}[shift={(-.75+.5,0+0)}]{\elbow}  \end{scope}
	  \begin{scope}[shift={(-.75+1,0+0)}]{\elbow}  \end{scope}
	  \begin{scope}[shift={(-.75+1.5,0+0)}]{\elbow}  \end{scope}
	  \begin{scope}[shift={(-.75+0,0+.5)}]{\+}  \end{scope}
	  \begin{scope}[shift={(-.75+.5,0+.5)}]{\+}  \end{scope}
	  \begin{scope}[shift={(-.75+1,0+.5)}]{\elbow}  \end{scope}
	  \begin{scope}[shift={(-.75+1.5,0+.5)}]{\elbow}  \end{scope}
	  \begin{scope}[shift={(-.75+0,0+1)}]{\+}  \end{scope}
	  \begin{scope}[shift={(-.75+.5,0+1)}]{\+}  \end{scope}
	  \begin{scope}[shift={(-.75+1,0+1)}]{\+}  \end{scope}
	  \begin{scope}[shift={(-.75+1.5,0+1)}]{\elbow}  \end{scope}
	  \begin{scope}[shift={(-.75+0,0+1.5)}]{\elbow}  \end{scope}
	  \begin{scope}[shift={(-.75+.5,0+1.5)}]{\+}  \end{scope}
	  \begin{scope}[shift={(-.75+1,0+1.5)}]{\+}  \end{scope}
	  \begin{scope}[shift={(-.75+1.5,0+1.5)}]{\+}  \end{scope}
	  \begin{scope}[shift={(2+0,0+0)}]{\+}  \end{scope}
	  \begin{scope}[shift={(2+.5,0+0)}]{\elbow}  \end{scope}
	  \begin{scope}[shift={(2+1,0+0)}]{\+}  \end{scope}
	  \begin{scope}[shift={(2+1.5,0+0)}]{\elbow}  \end{scope}
	  \begin{scope}[shift={(2+0,0+.5)}]{\+}  \end{scope}
	  \begin{scope}[shift={(2+.5,0+.5)}]{\elbow}  \end{scope}
	  \begin{scope}[shift={(2+1,0+.5)}]{\elbow}  \end{scope}
	  \begin{scope}[shift={(2+1.5,0+.5)}]{\+}  \end{scope}
	  \begin{scope}[shift={(2+0,0+1)}]{\+}  \end{scope}
	  \begin{scope}[shift={(2+.5,0+1)}]{\+}  \end{scope}
	  \begin{scope}[shift={(2+1,0+1)}]{\elbow}  \end{scope}
	  \begin{scope}[shift={(2+1.5,0+1)}]{\elbow}  \end{scope}
	  \begin{scope}[shift={(2+0,0+1.5)}]{\elbow}  \end{scope}
	  \begin{scope}[shift={(2+.5,0+1.5)}]{\+}  \end{scope}
	  \begin{scope}[shift={(2+1,0+1.5)}]{\+}  \end{scope}
	  \begin{scope}[shift={(2+1.5,0+1.5)}]{\+}  \end{scope}
	  \begin{scope}[shift={(5+0,0+0)}]{\+}  \end{scope}
	  \begin{scope}[shift={(5+.5,0+0)}]{\elbow}  \end{scope}
	  \begin{scope}[shift={(5+1,0+0)}]{\+}  \end{scope}
	  \begin{scope}[shift={(5+1.5,0+0)}]{\elbow}  \end{scope}
	  \begin{scope}[shift={(5+0,0+.5)}]{\+}  \end{scope}
	  \begin{scope}[shift={(5+.5,0+.5)}]{\elbow}  \end{scope}
	  \begin{scope}[shift={(5+1,0+.5)}]{\+}  \end{scope}
	  \begin{scope}[shift={(5+1.5,0+.5)}]{\+}  \end{scope}
	  \begin{scope}[shift={(5+0,0+1)}]{\+}  \end{scope}
	  \begin{scope}[shift={(5+.5,0+1)}]{\elbow}  \end{scope}
	  \begin{scope}[shift={(5+1,0+1)}]{\elbow}  \end{scope}
	  \begin{scope}[shift={(5+1.5,0+1)}]{\elbow}  \end{scope}
	  \begin{scope}[shift={(5+0,0+1.5)}]{\elbow}  \end{scope}
	  \begin{scope}[shift={(5+.5,0+1.5)}]{\+}  \end{scope}
	  \begin{scope}[shift={(5+1,0+1.5)}]{\+}  \end{scope}
	  \begin{scope}[shift={(5+1.5,0+1.5)}]{\+}  \end{scope}
\end{tikzpicture}
\caption{The type~C pipe dream complex $PD^C_{\rho,w}$ for $w=58372615$ and $\rho$ the pipe dream in \Cref{fig: type C reduced word} is obtained from the simplicial complex above by coning the vertex corresponding to the elbow in position $(4,2)$. 
}
\end{figure}
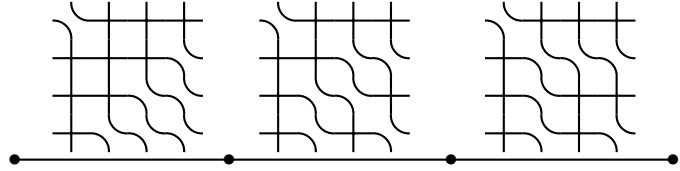

The following lemma follows from the definition of $PD^C_{\rho,w}$.

\begin{lemma}\label{lem: pd isom subword}
The simplicial complex $PD^C_{\overline{v},w}$ is the image of $\Delta_{\overline{v},w}$ 
under the isomorphism acting on vertices as $z_{ij}\mapsto(j,i)$.
In particular, the facet $F$ of $\Delta_{\overline{v},w}$ corresponds to the facet $\{(p,q)\in \rho(\overline{v})_{\rm L}: z_{qp}\in F\}$ of $PD^C_{\overline{v},w}$.
\end{lemma}

\subsection{Multidegrees and K-polynomials}

In analogy with the type~A setting \cite[Theorem 3.2]{WooYongGrobner}, we have that the prime components of the initial ideal of each $I_{\overline{v},w}$, with $v\ge v_\square$ in left-right weak order, are indexed by type~C pipe dreams.

\begin{cor}\label{cor:initIndex}
For any diagonal term order $\prec$, the initial ideal 
$\init_\prec I_{\overline{v},w}$ has the prime decomposition
\[\init_\prec I_{\overline{v},w} = \bigcap_\rho \,\langle z_{qp} : (p,q)\in\rho_{\rm L}\rangle\]
where $\rho$ ranges over all reduced elements of $PD^C_{\overline{v},w}$.
\end{cor}

\begin{proof}
By \Cref{cor: K is initial}, we have $\init_\prec I_{\overline{v},w} = K_{\overline{v},w}$, the latter of which is the Stanley-Reisner ideal of the (type~C) subword complex $\Delta_{\overline{v},w}$. Consequently,
\[
\init_\prec I_{\overline{v},w}=
\bigcap_{F\in \text{Facets}(\Delta_{\overline{v},w})}\langle \zeta([\ell]\setminus F)\rangle,
\]
where $\zeta: [\ell]\rightarrow V_{\overline{v}}$ is as defined immediately after the proof of \Cref{prop: Q is a reduced word}. The result now follows by \Cref{lem: pd isom subword}.
\end{proof}

Our next goal is to provide multidegree and $K$-polynomial formulas for our type~C Kazhdan-Lusztig varieties $X_{\overline{v},w}$, $v\ge v_{\square}$ in terms of type~C pipe dreams.
Our formulas are in the variables $t_1,\ldots,t_n$ discussed in \Cref{sect:TActionSmall}. We recall the action of $C_n$ on this variable set described there: 
given $u\in C_n$, the action is so that for $i\le n$,
$$
	u\odot t^{-1}_i=\begin{cases}
	t_{n+1-u(n+i)} &  \text{if }\ u(i+n)\le n,\\
	t^{-1}_{u(n+i)-n} &  \text{if }\ u(i+n)\ge n+1.
\end{cases}
$$
We will also want an additive version of this action:
$$
	u\oplus-t_i=\begin{cases}
	t_{n+1-u(n+i)}&  \text{if }\ u(i+n)\le n,\\
	-t_{u(n+i)-n}&  \text{if }\ u(i+n)\ge n+1.
\end{cases}
$$

Corollary \ref{cor:initIndex} immediately implies a positive multidegree formula for our type~C Kazhdan-Lusztig varieties $X_{\overline{v},w}$, $v\ge v_{\square}$ in terms of type~C pipe dreams. Note that this formula can also be recovered from the $K$-polynomial formula given below in Proposition \ref{prop: K-poly pipe}.

\begin{prop}\label{prop:multidegree formula}
The multidegree of $R_{\overline{v}}/I_{\overline{v},w}$ is
\[\mathcal{C}(R_{\overline{v}}/I_{\overline{v},w};\mathbf{t}) =
\sum_\rho \prod_{\substack{(i,j) \text{ is a} \\ \text{cross in } \rho_{\rm L}}} 
\left((u_l\oplus-t_i)+(u_l\oplus-t_j)\right)
\]
where $\rho$ ranges over all reduced elements of $PD^C_{\overline{v},w}$.
\end{prop}

As noted in the introduction, Ikeda, Mihalcea and Naruse also give formulas for these multidegrees \cite[Theorem 1.1, Definition 6.1]{Ikeda-Mihalcea-Naruse}.
To aid the reader in aligning conventions, we observe that the multidegree of our $I_{\overline{v},w}$
is what is called in \cite{Ikeda-Mihalcea-Naruse} the localization of $\mathfrak S_{w_0w}$ at $w_0v$,
but our work and theirs agree on the meaning of~$t_i$.

\begin{rmk}
Our pipe dream formula specializes to a formula for multiplicities: since each $v\geq v_\square$ is $123$-avoiding, each ideal $I_{\overline{v},w}$ is homogeneous with respect to the standard grading. Consequently, the multiplicity $\text{mult}_{P(v)}(X_w)$ of the Schubert variety $X_w$ at the point $P(v)B_G^+/B_G^+$ is equal to the number of reduced type $C$ pipe dreams in $PD^C_{\overline{v},w}$. See \cite[Fact 5.1]{WooYongGrobner} for the analogue of this observation in the type $A$ setting.

We note that other combinatorial formulas for multiplicities in type $C$ in special cases can be found in the works \cite{Anderson-Ikeda-Jeon-Kawago}, \cite{Ghorpade-Raghavan}, \cite{Ikeda-Naruse}, and \cite{Kreiman}. For example, in the recent paper \cite{Anderson-Ikeda-Jeon-Kawago}, Anderson, Ikeda, Jeon, and Kawago provide a combinatorial formula for the multiplicity of a singularity of a co-vexillary Schubert variety (in a classical-type flag variety) in terms of excited Young diagrams.
\end{rmk}

Our next goal is to give a formula for the $K$-polynomial of $R_{\overline{v}}/I_{\overline{v},w}$ in terms of type~C pipe dreams.

\begin{prop}\label{prop: K-poly pipe}
Let $\rho$ be the reduced pipe dream for $v$ associated to the factorization $\overline{v}=u_lv_\square u_r$.
The $K$-polynomial of $R_{\overline{v}}/I_{\overline{v},w}$ is 
	\begin{equation*}
	\mathcal{K}(R_{\overline{v}}/I_{\overline{v},w};\mathbf{t})
	=
	\sum_{\sigma \in PD^C_{\overline{v},w}} (-1)^{\mathrm{cr}(\sigma_{\rm L})-\ell(w_0w)}\prod_{ \substack{(i,j) \text{ is a} \\ \text{cross in } \sigma_{\rm L}}} \left(1-(u_l\odot t_i^{-1})(u_l\odot t_j^{-1})\right),
	\end{equation*}
where $\mathrm{cr}(\sigma_{\rm L})$ is the number of crosses in $\sigma_{\rm L}$,
and the sum is over the non-boundary faces of $PD^C_{\overline{v},w}$.
\end{prop}

Note that all factorizations of $v$ yield the same $K$-polynomial.

\begin{proof}
By \cite[Theorem 4.1]{KnutsonMillerAdvances}, the $K$-polynomial of the Stanley-Reisner ideal $I$ of the subword complex $S(Q,w_0w)$ is
	$$
	\mathcal{K}(R/I;\mathbf{t})=
	\sum_{F} (-1)^{|Q\setminus F|-\ell(w_0w)}\prod_{i\notin F}\left(1-\mathbf{t}^{\deg(i)}\right)
	$$
where the sum is over the faces $F$ of $S(Q,w)$ such that the word in the nilHecke algebra associated to $Q\setminus F$ is a word for $w_0w$.
This is the same index set as in the proposition by \cite[Theorem 3.7]{KnutsonMillerAdvances}.
By \Cref{lem: pd isom subword}, $PD^C_{\overline{v},w}$ is isomorphic to $\Delta_{\overline{v},w}$, which in turn is a relabelling of a subword complex $S(Q,w_0w)$, as in \Cref{sec: labeling}.
We obtain the expression in the proposition by translating the equation above into the language of type~C pipe dream complexes.
Explicitly, for $\sigma$ the pipe dream in $PD^C_{\overline{v},w}$ associated to the face $F$ in  $S(Q,w_0w)$:
	\begin{itemize}
	\item $\sigma$ is a pipe dream for $w$ if and only if the word in the nilHecke algebra corresponding to $Q\setminus F$ is a word for $w_0w$,
	\item $|Q\setminus F|=\mathrm{cr}(\sigma_{\rm L})$, and
	\item by \Cref{prop: weights}, if $(i,j)$ is a cross in $\sigma_{\rm L}$, then $\mathbf{t}^{\deg(i,j)}=(u_l\odot t_i^{-1})(u_l\odot t_j^{-1})$.\qedhere
\end{itemize}
\end{proof}

We can also give a formula without the signs of \Cref{prop: K-poly pipe}.
As noted in the previous proof, $PD^C_{\overline{v},w}$ is isomorphic to the subword complex $S(Q,w_0w)$,
which is a shellable simplicial complex \cite[Theorem 2.5]{KnutsonMillerAdvances}.
That is, there exists a \emph{shelling order} $F_1,\ldots,F_m$ of the facets of~$PD^C_{\overline{v},w}$,
namely, an order such that, for each $i=2,\ldots,m$, 
the intersection of $F_i$ with~$\bigcap_{j<i}F_j$ is pure of dimension $\dim(PD^C_{\overline{v},w})-1$.

\begin{prop}\label{prop: K-poly shelling}
With notation as in \Cref{prop: K-poly pipe}, let $F_1,\ldots,F_m$ be a shelling order for $PD^C_{\overline{v},w}$. 
Then the $K$-polynomial of $R_{\overline{v}}/I_{\overline{v},w}$ is
	\begin{equation*}
	\mathcal{K}(R_{\overline{v}}/I_{\overline{v},w};\mathbf{t})
	=
	\sum_{k=1}^m 
	\prod_{ \substack{(i,j) \text{ is a} \\ \text{cross in } (F_k)_{\rm L}}} \left(1-(u_l\odot t_i^{-1})(u_l\odot t_j^{-1})\right)
	\prod_{(i,j)\in\operatorname{Abs}(F_k)} \left((u_l\odot t_i^{-1})(u_l\odot t_j^{-1})\right)
	,
	\end{equation*}
where $\operatorname{Abs}(F_k)$ is the set of elbows in $(i,j)\in F_k$ such that
$F_k\setminus(i,j)\subset F_\ell$ 
for some $\ell<k$. 
\end{prop}

The corresponding result in \cite{KnutsonMillerAdvances} is Theorem~4.4, 
which uses the language of subword complexes and describes $\operatorname{Abs}(F_k)$ differently:
see \cite[Remark 4.5]{KnutsonMillerAdvances}.

\begin{proof}
This follows from \Cref{prop: K-poly pipe} by collecting, for each $k=1,\ldots,m$,
the summands indexed by faces $\sigma$ of the form
$\sigma = \bigcap_{\ell\in L}F_\ell$ where $k=\max L$,
as described in the ungraded case in \cite[Proposition~2.3]{Stanley96}.
\end{proof}

In the case in which $v\ge_R v_\square$, we can always take a factorization $\overline{v}$ with $u_l=\mathrm{id}$,
which simplifies the appearance of the product in the formula in \Cref{prop: K-poly pipe}.
A similar simplification can be written down for \Cref{prop: K-poly shelling}. 

\begin{cor}
For $v\ge_R v_\square$ the K-polynomial of $R_{\overline{v}}/I_{\overline{v},w}$ is
	\begin{equation*}
	\mathcal{K}(R_{\overline{v}}/I_{\overline{v},w};\mathbf{t})
	=
	\sum_{ \sigma \in PD^C_{\overline{v},w}} (-1)^{|\sigma_{\rm L}|-\ell(w_0w)}\prod_{(i,j)\in \sigma_{\rm L}} \left(1-t_i^{-1}t_j^{-1}\right).
	\end{equation*}
\end{cor}

The following example shows that this type~C $K$-polynomial is not simply obtained from the type~A $K$-polynomial
by substituting the weights as suggested by the embedding of maximal tori from $Sp_{2n}(\mathbb{K})$ to $GL_{2n}(\mathbb{K})$. 

\begin{ex}
Let $v=v_\square=2143$, $w=3412$ in $C_2$.
Notice that the only pipe dream for $w$ in $PD^C_{\overline{v},w}$ is 
	$$
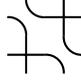
\begin{tikzpicture}
	  \begin{scope}[shift={(0,0+0)}]{\+}  \end{scope}
	  \begin{scope}[shift={(.5,0+0)}]{\elbow}  \end{scope}
	  \begin{scope}[shift={(0,0+.5)}]{\elbow}  \end{scope}
	  \begin{scope}[shift={(.5,0+.5)}]{\+}  \end{scope}
\end{tikzpicture}$$
and therefore 
\[\mathcal{K}(R_{\overline{v}}/I_{\overline{v},w};\mathbf{t}) = 1-t_1^{-1}t_2^{-1}.\]

We now compare this to the type~A $K$-polynomial using the embedding of tori above.
To avoid confusion, let us denote the variables of the $K$-polynomials of type~A Kazhdan-Lusztig varieties by $t^A_1,t^A_2,t^A_3,t^A_4$, where $t^A_i$ corresponds to the $i$-th entry of the torus $T$ consisting of diagonal matrices in $GL_{2n}(\mathbb{K})$.
The pipe dream pictured above is also the only pipe dream for $w$ in $PD^A_{v,w}$.
By the type~A version of \cite[Theorem 4.1]{KnutsonMillerAdvances} we have that the $K$-polynomial of the type~A Kazhdan-Lusztig variety associated to $v,w$ is $\left(1-t^A_3(t^A_1)^{-1}\right)\left(1-t^A_4(t^A_2)^{-1}\right)$.
The embedding of the torus in $Sp_{4}(\mathbb{K})$ is given by the substitution 
  $$
  t_1^A=t_2,\qquad t_2^A=t_1,\qquad t_3^A=t_1^{-1},\qquad t_4^A=t_2^{-1}.
  $$
Therefore, $	\mathcal{K}(R_{\overline{v}}/I_{\overline{v},w};\mathbf{t})$ is not the substitution of the type~A $K$-polynomial.
\end{ex}

\subsection{Combinatorics of type~C pipe dreams}\label{sec: combo pd C}

In this section we discuss the combinatorics of type~C pipe dreams. 
We begin by showing that we can recognize type~C pipe dreams for $w$ by following pseudolines, mirroring type~A pipe dreams.

\begin{prop}\label{prop: commutation}
A type~C pipe dream is a pipe dream for $w$ if and only if
when we transport labels $1,\ldots,2n$ from the north and west sides of the picture along the pseudolines,
ignoring all crossings subsequent to the first between each pair of pseudolines,
the resulting labels on the south side read $w$.
\end{prop}

\begin{proof}
Let $\rho$ be a type~C pipe dream for $w$, meaning that its associated word $Q=(\alpha_1,\ldots,\alpha_\ell)$ is a word in the nilHecke algebra of type~C for $w_0w$.
Let $Q^A$ be the word obtained by replacing replacing each $\alpha_k$ with either one or two entries as follows:
	\begin{equation}\label{eq: c to a word}
	\alpha_i\mapsto
	\begin{cases}
	n & \text{ if } \alpha_k=0,\\
	n-\alpha_k,n+\alpha_k & \text{ if } \alpha_k\neq 0.
	\end{cases}
	\end{equation}
Note that if $\alpha_k$ corresponds to a cross of $\rho$ in position $(p,q)$ of $D(v_\square)$ then $\alpha_k={p-q}$, and as a type~A pipe dream, $\rho$ has crosses at positions $(n+p,q)$ and $(n+q,p)$.
Our goal is to use commutation relations to transform $Q^A$ into the word, in the nilHecke algebra of type~A, associated to $\rho$.
Let us denote this latter word by $\mathcal{Q}$.
Note then that the ``only if'' part of the proposition will follow from the combinatorics of type~A pipedreams.

The first entry of $\mathcal{Q}$ is $n+q-p=n-\alpha_1$, which is also the first entry of $Q^A$.
Now consider $n-\alpha_k$ and suppose that for $i=1,\ldots,k-1$ we have used the commutation relation to move $n-\alpha_i$ in $Q^A$ to the correct position in $\mathcal{Q}$.
We wish to move $n-\alpha_k=n+q-p$ to the position in $\mathcal{Q}$ associated to the cross in position $(n+q,p)$.
If $\alpha_k=0$ this is already the case, so let's suppose that $\alpha_k\neq 0$.
We are allowed to use the commutation relation as long as we don't encounter $n+q-p+1$ or $n+q-p-1$. 
In Figure~\ref{fig: blue and gray}, the blue boxes represent the positions of the crosses that contribute $n+q-p\pm1$ to $\mathcal{Q}$.
For $i=1,\ldots,k-1$ an $n-\alpha_i$ corresponding to a cross in the gray shaded region has been moved to the correct position in $\mathcal{Q}$.
Note that all the blue boxes outside the gray region correspond to entries in $Q^A$ that appear after $n\pm \alpha_k$. 
Therefore, we can use commutation relations to move $n-\alpha_k$ to the left until it reaches the correct position in $\mathcal{Q}$.
Continuing with this process, we transform $Q^A$ into $\mathcal{Q}$ using only commutation relations.
\begin{figure}[hbt]
	$$
	\begin{tikzpicture}	
	 \begin{scope}[shift={(0,3)}]
	 { \filldraw[fill=cyan!40] (0,0) rectangle (.5,.5);
	 }  \end{scope}
	 \begin{scope}[shift={(.5,2.5)}]
	 { \filldraw[fill=cyan!40] (0,0) rectangle (.5,.5);
	 }  \end{scope}
	 \begin{scope}[shift={(1,2)}]
	 { \filldraw[fill=cyan!40] (0,0) rectangle (.5,.5);
	 }  \end{scope}
	 \begin{scope}[shift={(1.5,1.5)}]
	 { \filldraw[fill=cyan!40] (0,0) rectangle (.5,.5);
	 }  \end{scope}
	 \begin{scope}[shift={(2,1)}]
	 { \filldraw[fill=cyan!40] (0,0) rectangle (.5,.5);
	 }  \end{scope}
	 \begin{scope}[shift={(2.5,.5)}]
	 { \filldraw[fill=cyan!40] (0,0) rectangle (.5,.5);
	 }  \end{scope}
	 \begin{scope}[shift={(3,0)}]
	 { \filldraw[fill=cyan!40] (0,0) rectangle (.5,.5);
	 }  \end{scope}

	 \begin{scope}[shift={(0,2)}]
	 { \filldraw[fill=cyan!40] (0,0) rectangle (.5,.5);
	 }  \end{scope}
	 \begin{scope}[shift={(.5,1.5)}]
	 { \filldraw[fill=cyan!40] (0,0) rectangle (.5,.5);
	 }  \end{scope}
	 \begin{scope}[shift={(1,1)}]
	 { \filldraw[fill=cyan!40] (0,0) rectangle (.5,.5);
	 }  \end{scope}
	 \begin{scope}[shift={(1.5,.5)}]
	 { \filldraw[fill=cyan!40] (0,0) rectangle (.5,.5);
	 }  \end{scope}
	 \begin{scope}[shift={(2,0)}]
	 { \filldraw[fill=cyan!40] (0,0) rectangle (.5,.5);
	 }  \end{scope}

	 \begin{scope}[shift={(2.5,4.5)}]
	 { \filldraw[fill=cyan!40] (0,0) rectangle (.5,.5);
	 }  \end{scope}
	 \begin{scope}[shift={(3,4)}]
	 { \filldraw[fill=cyan!40] (0,0) rectangle (.5,.5);
	 }  \end{scope}
	 \begin{scope}[shift={(3.5,3.5)}]
	 { \filldraw[fill=cyan!40] (0,0) rectangle (.5,.5);
	 }  \end{scope}
	 \begin{scope}[shift={(4,3)}]
	 { \filldraw[fill=cyan!40] (0,0) rectangle (.5,.5);
	 }  \end{scope}
	 \begin{scope}[shift={(4.5,2.5)}]
	 { \filldraw[fill=cyan!40] (0,0) rectangle (.5,.5);
	 }  \end{scope}

	  \begin{scope}[shift={(1.5,4.5)}]
	 { \filldraw[fill=cyan!40] (0,0) rectangle (.5,.5);
	 }  \end{scope}
	 \begin{scope}[shift={(2,4)}]
	 { \filldraw[fill=cyan!40] (0,0) rectangle (.5,.5);
	 }  \end{scope}
	 \begin{scope}[shift={(2.5,3.5)}]
	 { \filldraw[fill=cyan!40] (0,0) rectangle (.5,.5);
	 }  \end{scope}
	 \begin{scope}[shift={(3,3)}]
	 { \filldraw[fill=cyan!40] (0,0) rectangle (.5,.5);
	 }  \end{scope}
	 \begin{scope}[shift={(3.5,2.5)}]
	 { \filldraw[fill=cyan!40] (0,0) rectangle (.5,.5);
	 }  \end{scope}
	 \begin{scope}[shift={(4,2)}]
	 { \filldraw[fill=cyan!40] (0,0) rectangle (.5,.5);
	 }  \end{scope}
	 \begin{scope}[shift={(4.5,1.5)}]
	 { \filldraw[fill=cyan!40] (0,0) rectangle (.5,.5);
	 }  \end{scope}

	\fill[fill=black!100,fill opacity=0.2] (0,1)--(0,5)--(4,5)--(4,3)--(3.5,3)--(3.5,1.5)--(2,1.5)--(2,1)
		;
	 
	\filldraw[fill=black!60] (1.5,1) rectangle (2,1.5)
		(3.5,3) rectangle (4,3.5)
		;
	\draw[dashed] (0,5)--(5,0);
	\draw (0,1)--(2,1)--(2,1.5)--(3.5,1.5)--(3.5,3)--(4,3)--(4,5)
		(0,0) rectangle (5,5)
		;
	\draw (0,1.25) node[left] {$q$}
		(0,3.25) node[left] {$p$} 
		(3.75,5) node[above] {$q$}
		(1.75,5) node[above] {$p$}
		;
	\end{tikzpicture}
	$$
\caption{Partway through applying commutation relations, in the proof of Proposition~\ref{prop: commutation}.}
\label{fig: blue and gray}
\end{figure}
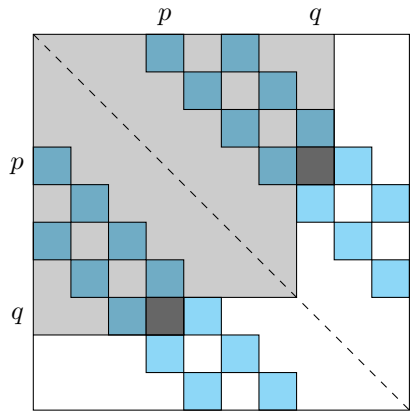

Now let $\rho$ be a type~C pipe dream such that following the pseudolines yields $w$. 
Let $\mathcal{Q}$ be the word, in the nilHecke algebra of type~A, associated to $\rho$.
By undoing the process described above, we can transform $\mathcal{Q}$ into $Q^A$ and lastly to $Q$ by undoing the substitution \eqref{eq: c to a word}.
We leave the details to the reader.
\end{proof}

\section{Beyond small patches}\label{sec:beyondSmall}
In \cite{KnutsonFrobenius} Knutson showed the defining ideal of any Kazhdan-Lusztig variety has a Gr\"obner basis whose leading terms are squarefree, and, in \cite{Knutson-degenerate}, he determined that the resulting initial ideal is the Stanley-Reisner ideal of the subword complex $S(Q,w_0w)$, where $Q$ is a reduced word for $w_0v$. 
For small patches, our coordinates agree (up to sign) with the \emph{Bott-Samelson coordinates} in \cite{KnutsonFrobenius}, and our monomial order $\prec_{\text lex}$ agrees with the monomial order in \cite{KnutsonFrobenius}. 
Thus, Theorem \ref{thm: main} shows that the type~C essential minors are a Gr\"obner basis in Knutson's set-up.
(Knutson does not provide a Gr\"obner basis in \cite{KnutsonFrobenius}.) 
In this short section, we show that things are more mysterious beyond the case of small patches as the essential minors are \emph{not} typically a Gr\"obner basis.

\begin{ex}
Let $v = 231645$ as in Example \ref{ex: type C patch}. 
Observe that $Q = (0,1,0,2,1,0,2)$ is a reduced word for $w_0v$.
Let \[
C_0^{(i)} = \begin{bmatrix}1&0&0&0&0&0\\0&1&0&0&0&0\\0&0&a_i&-1&0&0\\0&0&1&0&0&0\\0&0&0&0&1&0\\0&0&0&0&0&1\end{bmatrix},
\enskip
C_1^{(i)} = \begin{bmatrix}1&0&0&0&0&0\\0&b_i&-1&0&0&0\\0&1&0&0&0&0\\0&0&0&b_i&1&0\\0&0&0&-1&0&0\\0&0&0&0&0&1\end{bmatrix},
\enskip
C_2^{(i)} = \begin{bmatrix}c_i&-1&0&0&0&0\\1&0&0&0&0&0\\0&0&1&0&0&0\\0&0&0&1&0&0\\0&0&0&0&c_i&1\\0&0&0&0&-1&0\end{bmatrix},
\]
so that in \emph{Bott-Samelson} coordinates, the opposite cell associated to $v$ is identified with the space of matrices
\[
w_0C_0^{(1)}C_1^{(1)}C_0^{(2)}C_2^{(1)}C_1^{(2)}C_0^{(3)}C_2^{(2)} = \begin{bmatrix}
      0&0&1&0&0&0\\
      -1&0&0&0&0&0\\
      {c}_{2}&-1&0&0&0&0\\
      {c}_{2}{a}_{1}-{b}_{1}&-{a}_{1}&{c}_{1}&0&1&0\\
      {c}_{2}{b}_{1}-{a}_{2}&-{b}_{1}&{b}_{2}&0&{c}_{2}&1\\
      {c}_{2}{c}_{1}-{b}_{2}&-{c}_{1}&{a}_{3}&-1&0&0\end{bmatrix},
\]
where each $a_i, b_i, c_i$ can take arbitrary values in $\mathbb{K}$. 
Let $w = 462513$. Then, now treating $a_i, b_i, c_i$ as indeterminates, we see that the type~C essential set (see \Cref{def: type C essential}) is $\{(5,1), (5,3)\}$ and so the type~C essential minors are:
\[
\mathcal{E} = \left\{c_2c_1-b_2, \quad c_2b_1-a_2,\quad 2\times 2 \text{ minors of } \begin{bmatrix} c_2b_1-a_2 & -b_1 &b_2\\c_2c_1-b_2&-c_1&a_3 \end{bmatrix}\right\}.
\]
Using the lexicographic monomial order $c_2>a_3>b_2>c_1>a_2>b_1>a_1$, which is compatible with the vertex decomposition of the subword complex $S(Q,w_0w)$ described in Section~\ref{sec:subword complexes}, we see that 
\begin{itemize}
\item the initial ideal of the Kazhdan-Lusztig ideal $\langle \mathcal{E}\rangle$ is the Stanley-Reisner ideal of $S(Q,w_0w)$ as expected, yet
\item $\mathcal{E}$ is \emph{not} a Gr\"obner basis.
\end{itemize}
Nevertheless, the set of type~A essential minors is a Gr\"obner basis.
There also exists a Gr\"obner basis consisting of type~C essential minors which differ from the conventions introduced in \Cref{sec: type C swc}, namely the minors given by choosing the essential boxes $\{(5,1), (3,3)\}$.
\end{ex}

In the next example, we see that unlike in the previous example, the type~A essential minors needn't be a Gr\"obner basis either.

\begin{ex}
Consider $v = 213465$ so that $Q = (0,1,0,2,1,0,2,1)$ is a reduced word for $w_0v$. Let $C_0^{(i)}, C_1^{(i)}, C_2^{(i)}$ as in the previous example. Then the opposite Schubert cell associated to $v$ is identified with the space of matrices
 \begin{equation}\label{eq:nonEssentialEx}
\begin{bmatrix}
      0&1&0&0&0&0\\
      -1&0&0&0&0&0\\
      {c}_{2}&-{b}_{3}&1&0&0&0\\
      {c}_{2}{a}_{1}-{b}_{1}&-{b}_{3}{a}_{1}+{c}_{1}&{a}_{1}&-1&0&0\\
      {c}_{2}{b}_{1}-{a}_{2}&-{b}_{3}{b}_{1}+{b}_{2}&{b}_{1}&-{c}_{2}&0&1\\
      {c}_{2}{c}_{1}-{b}_{2}&-{b}_{3}{c}_{1}+{a}_{3}&{c}_{1}&-{b}_{3}&-1&0
      \end{bmatrix}
      \end{equation}
where $a_i, b_i, c_i\in \mathbb{K}$. 
Let $w = 632541$. There is a unique (type~A or type~C) essential box $\{(4,3)\}$. 
If we treat $a_i, b_i, c_i$ as indeterminates, the essential minors are then the $2\times 2$ minors of the southwest $3\times 3$ submatrix of \eqref{eq:nonEssentialEx}. 
For the lexicographic monomial order with $b_3>c_2>a_3>b_2>c_1>a_2>b_1>a_1$, we see that the ideal generated by the essential minors has initial ideal equal to the Stanley-Reisner ideal of $S(Q,w_0w)$, yet the set of essential minors is \emph{not} a Gr\"obner basis.
\end{ex}

Consequently, it is still an open problem to find combinatorially-defined Gr\"obner basis for type $C$ Kazhdan-Lusztig ideals $\mathcal{N}_{v,w}$ when $v\ngeq v_\square$ in left-right weak order. 

In type B, our methods fail because the analogue of \cite[Proposition 6.1.1.2]{Lakshmibai-Raghavan} does not hold scheme-theoretically.  In some cases, imposing the determinantal equations on the patches yields non-reduced schemes.  If we apply our methods to type B small patches, we see that we end up taking the determinants of some skew-symmetric matrices, and the obvious solution to this problem is to take the pfaffians of those skew-symmetric matrices instead.  We expect results similar to ours can be proven with this modification in that case.

Beyond small patches in type B, some unpublished preliminary work of the fourth author and Alexander Yong suggested that the appropriate equations would still form a Gr\"obner basis under a diagonal term order, but we need to impose rank conditions on some submatrices that are non-trivially similar to a skew-symmetric one.  Note that all formulas for the pfaffian require knowing the basis with respect to which a matrix is skew-symmetric, and we were not able to systematically determine the change of basis that turned these ``secretly skew-symmetric'' matrices into actually skew-symmetric matrices.  The following example illustrates some of the difficulties.

\begin{ex}
In this example we work with the type $B$ Weyl group. We embed $B_4$ into the symmetric group $S_9$ via $b_0 = s_4s_5s_4$, $b_1 = s_3s_6$, $b_2 = s_2s_7$, and $b_3 = s_1s_8$.
Consider $v=132456879\in S_9$ and observe that $Q = (0,1,2,3,0,1,2,3,0,1,2,0,1,2,3)$ is a reduced word for $w_0v$. Let
\[B_0^{(i)} = \begin{bmatrix}
1&0&0&0&0&0&0&0&0\\
0&1&0&0&0&0&0&0&0\\
0&0&1&0&0&0&0&0&0\\
0&0&0&-\frac{1}{2}a_{i}^{2}&a_{i}&1&0&0&0\\
0&0&0&a_{i}&-1&0&0&0&0\\
0&0&0&1&0&0&0&0&0\\
0&0&0&0&0&0&1&0&0\\
0&0&0&0&0&0&0&1&0\\
0&0&0&0&0&0&0&0&1
\end{bmatrix},
\quad
B_1^{(i)} = \begin{bmatrix}
1&0&0&0&0&0&0&0&0\\
0&1&0&0&0&0&0&0&0\\
0&0&b_{i}&-1&0&0&0&0&0\\
0&0&1&0&0&0&0&0&0\\
0&0&0&0&1&0&0&0&0\\
0&0&0&0&0&b_{i}&1&0&0\\
0&0&0&0&0&-1&0&0&0\\
0&0&0&0&0&0&0&1&0\\
0&0&0&0&0&0&0&0&1
\end{bmatrix},\]
\[
B_2^{(i)} = \begin{bmatrix}
1&0&0&0&0&0&0&0&0\\
0&c_{i}&-1&0&0&0&0&0&0\\
0&1&0&0&0&0&0&0&0\\
0&0&0&1&0&0&0&0&0\\
0&0&0&0&1&0&0&0&0\\
0&0&0&0&0&1&0&0&0\\
0&0&0&0&0&0&c_{i}&1&0\\
0&0&0&0&0&0&-1&0&0\\
0&0&0&0&0&0&0&0&1
\end{bmatrix}, 
\quad
B_3^{(i)} = \begin{bmatrix}
d_{i}&-1&0&0&0&0&0&0&0\\
1&0&0&0&0&0&0&0&0\\
0&0&1&0&0&0&0&0&0\\
0&0&0&1&0&0&0&0&0\\
0&0&0&0&1&0&0&0&0\\
0&0&0&0&0&1&0&0&0\\
0&0&0&0&0&0&1&0&0\\
0&0&0&0&0&0&0&d_{i}&1\\
0&0&0&0&0&0&0&-1&0
\end{bmatrix}.
\]
Then, using these Bott-Samelson coordinates, the opposite cell associated to $v$ is identified with the space of matrices
$
M_Q = w_0B_0^{(1)}B_1^{(1)}B_2^{(1)}B_3^{(1)}B_0^{(2)}B_1^{(2)}B_2^{(2)}B_3^{(2)}B_0^{(3)}B_1^{(3)}B_2^{(3)}B_0^{(4)}B_1^{(4)}B_2^{(4)}B_3^{(4)}$.

Let $w=381654927\in S_9$. The type $B$ essential set (which can be calculated in a similar way to the type $C$ essential set, see also \cite{AndersonEJC}) for $w$ consists of three boxes in locations $(7,6), (7,8),$ and $(9,6)$. Taking the associated essential minors of $M_Q$ (where $a_i, b_i, c_i, d_i$ appearing in $M_Q$ are considered as indeterminates), we obtain an ideal generated by $107$ minors. Using Macaulay2 \cite{M2}, we see that this ideal can be presented as:
\begin{align*}
I = \langle d_{1},\:c_{2},\:b_{4},\:b_{3},\:a_{4},\:c_{3}c_{4}-d_{2}d_{4},\:c_{1}c_{4}-b_{1}d_{4},\:a_{3}c_{4}-a_{2}d_{4},\:b_{1}c_{3}-c_{1}d_{2},\:a_{2}c_{3}-a_{3}d_{2}, \\
\:a_{3}b_{1}-a_{2}c_{1},\:a_{2}a_{3}d_{4}+2\,b_{2}d_{4},\:a_{2}^{2}d_{4}+2\,b_{2}c_{4},\:a_{3}^{2}d_{2}+2\,b_{2}c_{3},\:a_{2}a_{3}d_{2}+2\,b_{2}d_{2},\:a_{2}a_{3}c_{1}+2\,b_{2}c_{1},\\
\:a_{2}^{2}c_{1}+2\,b_{1}b_{2},\:a_{2}a_{3}b_{2}+2\,b_{2}^{2},\:a_{2}a_{3}^{2}+2\,a_{3}b_{2},\:a_{2}^{2}a_{3}+2\,a_{2}b_{2}\rangle.
\end{align*}
This ideal is not radical, hence does not scheme-theoretically define a Kazhdan-Lusztig variety. So, in particular the type $B$ analogue of Proposition \ref{prop: KL coord ring} is false. 
We note that radical of the ideal $I$ above is:
\[
\langle d_{1},\,c_{2},\,b_{4},\,b_{3},\,a_{4},\,c_{3}c_{4}-d_{2}d_{4},\,c_{1}c_{4}-b_{1}d_{4},\,a_{3}c_{4}-a_{2}d_{4},\,b_{1}c_{3}-c_{1}d_{2},\,a_{2}c_{3}-a_{3}d_{2},\,a_{3}b_{1}-a_{2}c_{1},\,a_{2}a_{3}+2\,b_{2}\rangle.\]

\end{ex}

\bibliography{biblio}
\bibliographystyle{plain}

\end{document}